\documentclass[a4paper,11pt]{article}
\usepackage{graphicx} 
\usepackage{amsmath,amssymb,amsthm,amsfonts,amstext,amsbsy,amscd,color,dsfont}
\usepackage{subcaption}
\graphicspath{{fig_report/}} 
\usepackage{float}
\usepackage{makecell}
\usepackage[utf8]{inputenc}
\usepackage{algorithm}
\usepackage{algorithmic}
\usepackage{textcomp}
\usepackage{stmaryrd}
\usepackage{accents}
\usepackage{hyperref}
\usepackage{tikz}
\usetikzlibrary{arrows.meta, positioning,calc,bending} 

\usepackage{cases}

\newcommand{\dd}{
		\mathop{}\mathopen{}\mathrm{d}
	}
 \usepackage{verbatim}
 
\newtheorem{thm}{Theorem}[section]
\newtheorem{assumption}{Assumption}

\newcommand{\R}{\mathbb{R}}

\newcommand{\no}{n_1}
\newcommand{\nt}{n_2}
\newcommand{\kd}{d}
\newcommand{\gamo}{c}
\newcommand{\gamt}{\tilde{c}}
\newcommand{\gamtot}{c_{\text{tot}}}

\newcommand{\CS}{\eta}
\newcommand{\CSNS}{\eta_{1}}
\newcommand{\CSS}{\eta_{2}}
\newcommand{\CtotS}{c_2^{\text{tot}}}
\newcommand{\CtotNS }{c_1^{\text{tot}}}
\newcommand{\CDNS}{c_{1}}
\newcommand{\CDS}{c_2}
\newcommand{\CDSNS}{\tilde{c}_{1}}
\newcommand{\CDSS}{\tilde{c}_{2}}
\newcommand{\ks}{k}

\newcommand{\one}{\mathds{1}}

\usepackage[top=2.5cm, bottom=3cm, left=3cm , right=3cm]{geometry}

\newcommand{\footremember}[2]{%
    \footnote{#2}
    \newcounter{#1}
    \setcounter{#1}{\value{footnote}}%
}
\newtheorem{definition}{Definition}[section]
\newtheorem{lemma}{Lemma}[section]
\newtheorem{prop}{Proposition}[section]
\newtheorem{remark}{Remark}[section]
\newtheorem{corollary}{Corollary}[section]
\begin{document}

\title{Stability analysis and long-time convergence of a partial differential equation model of two-phase ageing.}
\author{Luce Breuil \footremember{CMAP}{CMAP, CNRS, École polytechnique, Institut Polytechnique de Paris, Inria, Palaiseau, France; E-mail : luce.breuil@polytechnique.edu}}
\date{ }
\maketitle

\begin{abstract}
   Recent biological evidence suggests the presence of a two-phase ageing process in several species. We introduce a system of two age-structured partial differential equations (PDE) representing two phases of ageing of a wild population. The model includes a coupling of both equations through birth and transition between phases and non-linearities due to competition. We show the existence, positivity and uniqueness of weak solutions in a general setting. For a simplified system of ordinary differential equations (ODE), we show existence and uniqueness of a strictly positive steady state attracting all trajectories. We study another simplification, a coupled PDE-ODE model, for which we prove existence, uniqueness and local asymptotic stability of a strictly positive steady state. Under further assumptions, but without assuming weak non-linearities, we show the global asymptotic stability of that steady state. The uniqueness of steady states and absence of oscillations in these systems show that the proportion of individuals in each phase at equilibrium is a unique feature of the model. This paves the way to ecological applications as the experimental measure of such a proportion could help gain some insight on the health of a wild population.
\end{abstract}


\section{Introduction}
Although ageing is traditionally modeled as a continuous decay of an organism throughout its life, recent biological evidence suggests the presence of two consecutive ageing phases in several species \cite{rera_intestinal_2012,tricoire_new_2015,ageing_cell}. The individuals, which are all initially healthy, systematically go through an irreversible state of systemic failure before dying. This phase preceding death is called the Smurf phase \cite{rera_intestinal_2012} and can be easily detected by monitoring certain health indicators.  
Following these biological discoveries, we wish to study the mathematical properties of this new framework of ageing through a system of two coupled age-structured non-linear transport equations modelling a population of non-Smurf (healthy) and Smurf (failing) individuals. 

Age-structured transport equations are a classical way of modelling a population, the most notable of these equations are the linear renewal equation known as the McKendrick-von Forster equation (see e.g \cite{physio_equations_diek_joh}) and the non-linear Gurtin-McCamy equation \cite{gurtin1974non}.
Few articles study systems of equations representing several phases of ageing of a same population, and the presence of discrete age groups often replaces a continuous age variable (see e.g. \cite{CARRILLO2007137, Qualit_prey_pred_stage_DU_08,wikan2012nonlinear}). Most of these stage-structured population models consider a transition happening at a fixed age \cite{Ackleh_amphi_09,cushing1991juvenile,Farkas_asympt_08,Liu2017644}, which allows the system to be reduced to a single delay differential equation  \cite{KOSTOVA199965,review_Robertson_18}.
Another ingredient that is rarely present in age-structured population models is the assumption that the transition rate is density dependent, as what can be found for different differentiating cell models such as hematopoiesis \cite{ Adimy_Crauste_hema_05,Doumic_Leukemia_09,Hema_Kou_Adimy_Ducrot_2009} .
The juvenile-adult ODE models that assume a density dependent transition rate interpreted it as a maturation rate that is negatively impacted by the density-dependence, unlike competition \cite{Baer_bifurc_06,Changguo201274,KOSTOVA199965}.

To our knowledge, the only existing paper studying a PDE system directly inspired by the Smurf phenotype is \cite{meleard_roget_bd}. They present a general model where an individual is characterised by two traits : its fertility period and its survival period. They show that the model converges in time to a configuration where both periods are consecutive and non-overlapping, thus creating two distinct phases and motivating the evolutionary emergence of a two-phase ageing framework. The discontinuity of ageing being captured through the individual traits, their model does not include a transition from one phase to the next. 

 In this article, we wish to determine whether a system of equations modelling two life phases of a same population, which is a generalised Lotka-Volterra model \cite{Hopf_YAN_19,Implication_Perasso_19} and shares similarities with SIR models \cite{Li_21_SIR}, can share some of their properties. In particular, does the system have several stable steady states? Can it present sustained oscillations? We also wish to compare the system with one-phase models with the perspective of understanding the potential evolutionary advantages of such a two-phase model, and the impact of the parameters on the dynamics of the system. The recent biological experimental observations in which the model is rooted allow for a novel mathematical model, for which the behavior has not yet been studied and with results that can be interpreted from a biological perspective. In particular, knowing to what extent the proportion of Smurf individuals of a population is a unique, stable feature of the model would help gain some insight on wether or not the experimental measure of a Smurf proportion in the wild yields some information on the health of a population. This could have applications in ecology, when monitoring the health of wild populations close to human constructions. \\

\textbf{Organisation of the paper.} In the following, we first prove general properties of existence, uniqueness and positivity of weak solutions in Section \ref{sec:gen_2tudy}. We then study two simplifications of the system, for which we are able to prove stronger results. First, in Section \ref{sec:ODE},  we study an ODE reduction of our model and show using tools from stability analysis and bifurcation theory that it  has a single strictly positive globally asymptotically stable steady state. Secondly, in Section \ref{sec:PDE_ODE}, we introduce a PDE-ODE model, for which we study general existence and local stability of steady states using semi-group theory. We also show some comparison results with one-phase models. Under further assumptions we show our main result in Theorem \ref{thm:conv_global} : every trajectory of the system converges to the non-trivial steady state. This result is shown by studying a renormalised version of the system and introducing a Lyapunov type function and remains true even without assuming weak non-linearities.

\section{Presentation of the model and existence of solutions}
\label{sec:gen_2tudy}
In this section, we introduce a PDE model of two-phase ageing and prove general properties of existence, uniqueness and positivity of weak solutions.

\subsection{Presentation of the model}
We consider an age-structured model, which describes the evolution of a population in two phases. Individuals are born in phase~1, for which the density of population of age $a$ at time $t$ is represented by $n_1(t,a)$, and transition to phase~2. The density of individuals in phase~2 is represented by $n_2(t,a)$. The age $a$ corresponds to the age of an individual in its current phase. For phase~1, it is equal to the chronological age as individuals are directly born in the first phase. For phase~2 however, this corresponds to the time spent since transition and the age of an individual in phase~2 is set to 0 upon transition. From a biological perspective, the time spent since transition is key in understanding to what extent the second phase is a strong predictor of death. In particular, previous statistical studies of two-phase ageing consider the age since transition as the main variable affecting the probability of death \cite{Breuil2026.02.18.706552,tricoire_new_2015}. 

A schematic representation of the model is shown in Figure \ref{fig:schema}. 

\begin{figure}[h]

\begin{tikzpicture}[
    box/.style={draw, rectangle, minimum width=3cm, minimum height=1.6cm},
    arrow/.style={->, thick},
    node distance=5cm
]

\node[box] (P1) {phase~1};
\node[box, right=of P1] (P2) {phase~2};

\draw[arrow]
    (P1.west) to[out=220,in=160,looseness=3]
    node[left] { birth $b(a)$}
    (P1.north west);

\draw[arrow]
    (P1.east) -- ++(5,0.25)
    node[midway, above] {\shortstack{$k(a)$\\natural transition}};

\draw[arrow]
    (P1.east) -- ++(5,-0.25)
        node[midway, below] {\shortstack{$\eta\big(S_1(t),S_2(t)\big)$\\ competition transition}};

\draw[arrow]
    (P2.north) to[out=120,in=60,looseness=0.5]
    node[above] { birth $\tilde b(a)$}
    (P1.north);

\draw[arrow]
    (P1.south) -- ++(0,-2)
    node[midway, right] {$\gamo \big(S_1(t),S_2(t)\big)$}
    node[below] {\shortstack{accidental\\ death}};

\draw[arrow]
    (P2.south) -- ++(-1.5,-2)
    node[midway, left] {$\kd(a)$}
        node[below] {\shortstack{natural \\ death}};

\draw[arrow]
    (P2.south) -- ++(1.5,-2)
    node[midway, right] {$\gamt \big(S_1(t),S_2(t)\big)$}
        node[below] {\shortstack{accidental\\ death}};

\end{tikzpicture}

\caption{Schematic representation of the model.}
\label{fig:schema}
\end{figure}

The system we consider is the following
\begin{align}
\label{syst:edp}
    \begin{cases}
        &\frac{\partial \no(t,a) }{\partial t} + \frac{\partial \no(t,a) }{\partial a} = -(\ks(a) + \CS(S_2(t),S_1(t)) + \gamo(S_2(t),S_1(t)))\no(t,a)\\
        &\no(0,t) = \int_0^{+\infty}\no(t,a)b(a)+\nt(t,a)\tilde{b}(a) \dd a\\
        &\frac{\partial \nt(t,a) }{\partial t} + \frac{\partial \nt(t,a) }{\partial a} = -(\kd(a) +\gamt(S_2(t),S_1(t)) )\nt(t,a)\\
        &\nt(0,t) = \int_0^{+\infty}\no(t,a)(\ks(a)+\CS(S_2(t),S_1(t)))\dd a\\
        &\no(0,.) = n_{1,0}(.) \in L^1(\R_+, \R_+),\quad \nt(0,.) = n_{2,0}(.) \in L^1(\R_+,\R_+).
    \end{cases}
\end{align}
where there exist two competition kernels $\psi_1, \psi_2$ such that 
\begin{equation}
     S_1(t) = \int_0^{+\infty} \psi_1(a) \no(t,a) \dd a \quad \text{and} \quad  S_2(t) = \int_0^{+\infty} \psi_2(a) \nt(t,a) \dd a.
    \label{eq:def_2}
\end{equation}
In this model, 
\begin{itemize}
\item $\no(t,a)$ represents the density of individuals in phase 1 of chronological age $a$ at time $t$.
\item $\nt(t,a)$ represents the density of individuals in phase 2 of age since transition $a$ at time $t$ (when an individual transitions from phase~1 to 2, its age since transition is set to 0).
\item $\ks$ and $\kd$ are respectively the natural rates of transition from phase~1 to 2 and death, which we assume to be age-dependent.
\item $b$ and $\tilde{b}$ are respectively the reproduction rates of individuals in phase~1 and 2. Biological evidence suggests that $\tilde{b} \ll b$, and we exploit similar assumptions in this article.
\item $\CS$ represents the competition affecting the transition from phase~1 to 2, $\gamo$ the competition causing accidental death of individuals in phase~1 without going through phase~2 and $\gamt$ the accidental death of individuals in phase~2. All these competition functions are assumed to depend only on $S_1$ and $S_2$, which are environmental factors linked to the size of each population. 
\end{itemize}

For ease of notation, we introduce a notation for the total competition affecting phase~1
\begin{equation*}
    \gamtot(.,.) = \gamo(.,.) + \CS(.,.).
\end{equation*}
We begin by introducing the main assumptions, which ensure that the system makes biological sense and that we model a population in which everyone transitions from phase~1 to 2 and dies.  

\begin{assumption}\label{assump:bio}
\begin{align*}
     &\ks \geq 0, \kd \geq 0, b\geq 0 , \tilde{b} \geq 0, \\
     &\int_0^{+\infty}\ks(a) \dd a = +\infty \qquad \int_0^{+\infty}\kd(a) \dd a = +\infty.
    \end{align*}
\end{assumption}

\begin{assumption}\label{assump:compet_0}
\begin{align*} 
     &\gamt, \gamo, \eta \in L^{\infty}_{\text{loc}}(\R^2),\\
     & \gamt(0,0) =\gamo(0,0) =  \CS(0,0) = 0,\\
    &\forall S_2, S_1 \geq 0, \gamt(S_1, S_2)\geq 0, \gamo(S_1, S_2)\geq 0 , \CS(S_1, S_2) \geq 0.
    \end{align*}
\end{assumption}

\begin{assumption}\label{assump:bounds}
\begin{align}
&b, \tilde{b} \in L^{\infty}(\R_+) \label{bound:b} \tag{i},\\
& \ks \in  L^{\infty}(\R_+). \label{bound:ks} \tag{ii}
\end{align}
\end{assumption}
In the following, we denote by $||.||_{\infty}$ the $L^{\infty}$ norm and $||.||_1$ the $L^1$ norm.

Finally, we introduce the classical non-extinction condition. 

\begin{assumption} \label{ass:for_lambda_0}
    \begin{equation*}
         1< \int_0^{+\infty} \Big(b(a) e^{-\int_0^a \ks(u) \dd u} + \tilde{b}(a)e^{-\int_0^a \kd(u) \dd u} \Big) \dd a< +\infty.
    \end{equation*}
\end{assumption}

\subsection{Eigenvalue problem for the linearised system}

We introduce the following eigenvalue problem for the eigenvalue $\lambda_0$ corresponding to $S_1 = S_2 = 0$, 

\begin{equation}
\begin{cases}
\label{eq:nns_0}
   & \frac{\dd \no^{0}}{\dd a}(a) +(\ks(a) +\lambda_0)\no^{0}(a) =0\\
   &\no^{0}(0) =  \int_0^{+\infty} (b(a) \no^{0}(a)+ \tilde{b}(a) \nt^0(a)) \dd a = 1 \\
    &-\frac{\dd \nt^{0}}{\dd a}(a) +(\kd(a) +\lambda_0)\nt^{0}(a) =0 \\
    &\nt^{0}(0) = \int_0^{+\infty} \ks(a) \no^{0}(a) \dd a.
    \end{cases}
\end{equation}
The associated adjoint equation is
\begin{equation}
\begin{cases}
\label{eq:phi_0}
   & -\frac{\dd \phi_1^{0}}{\dd a}(a) +(\ks(a) +\lambda_0)\phi_1^{0}(a) =\phi_1^{0}(0) b(a) + \phi_2^{0}(0)\ks(a), \quad \phi_1^{0} \geq 0 \\
    &-\frac{\dd \phi_2^{0}}{\dd a}(a) +(\kd(a) +\lambda_0)\phi_2^{0}(a) =\phi_1^{0}(0) \tilde{b}(a), \quad \phi_2^{0} \geq 0\\ 
    &\int_0^{+\infty} \phi_1(a) n^0_1(a) + \phi_2(a) n^0_2(a) \dd a = 1.
    \end{cases}
\end{equation}

We begin with an existence and boundedness result inspired from \cite{Perthame_Tumulari2008}. 

\begin{prop}\label{prop:ex_phi}
  Suppose Assumptions \ref{assump:bio} and \ref{ass:for_lambda_0} are verified. Then, systems \eqref{eq:nns_0} and \eqref{eq:phi_0} have unique solutions with $\lambda_0 >0$ and the solutions to \eqref{eq:phi_0} are positive. Provided Assumption \ref{assump:bounds}\eqref{bound:b} is also verified, the solutions to \eqref{eq:phi_0} are bounded. 
  Finally, if $b$ and $\tilde{b}$ are bounded from below by strictly positive constants $\underline{b}$ and $\underline{\tilde{b}}$ and  $\kd$ and $\ks$ are bounded, the solutions to \eqref{eq:phi_0} are also bounded from below by strictly positive constants. 
\end{prop}
\begin{proof}
    We can explicitly solve \eqref{eq:nns_0} which yields for $a \geq 0$, 
\begin{align}
    &\no^0(a) = e^{-\int_0^a \ks(u) + \lambda_0 \dd u}\nonumber, \quad \nt^0(a) = \nt^0(0) e^{-\int_0^a \kd(u) + \lambda_0 \dd u}\nonumber, \quad \nt^0(0) = \int_0^{+\infty}\ks(a) \no^0(a) \dd a\nonumber,\\
    &1 = \int_0^{+\infty} b(a) e^{-\int_0^a (\ks(u) + \lambda_0) \dd u} +\Big(\int_0^{+\infty}\ks(u) e^{-\int_0^u (\ks(x) + \lambda_0) \dd x} \dd u\Big)\tilde{b}(a)e^{-\int_0^a (\kd(u) + \lambda_0) \dd u} \dd a\label{eq:lambda_0}.
\end{align}
The map $\lambda \mapsto  \int_0^{+\infty} b(a) e^{-\int_0^a (\ks(u) + \lambda) \dd u} +\Big(\int_0^{+\infty}\ks(u) e^{-\int_0^u (\ks(x) + \lambda) \dd x} \dd u\Big)\tilde{b}(a)e^{-\int_0^a (\kd(u) + \lambda) \dd u} \dd a$ is strictly decreasing in $\lambda$ and continuous on $\R_+$. 
With Assumption \ref{ass:for_lambda_0}, we have 
\begin{align*}
   & \lim_{\lambda \to 0}  \int_0^{+\infty} b(a) e^{-\int_0^a (\ks(u) + \lambda) \dd u} +\Big(\int_0^{+\infty}\ks(u) e^{-\int_0^u (\ks(x) + \lambda) \dd x} \dd u\Big)\tilde{b}(a)e^{-\int_0^a (\kd(u) + \lambda) \dd u} \dd a> 1,\\
    &  \lim_{\lambda \to +\infty } \! \int_0^{+\infty} b(a) e^{-\int_0^a (\ks(u) + \lambda) \dd u} \!+\! \Big(\int_0^{+\infty}\ks(u) e^{-\int_0^u (\ks(x) + \lambda) \dd x} \dd u\Big)\tilde{b}(a)e^{-\int_0^a (\kd(u) + \lambda) \dd u} \dd a \!= \! 0. 
\end{align*}
Hence, there is a unique strictly positive solution $\lambda_0$ to \eqref{eq:lambda_0}.
For the adjoint problem solution to system \eqref{eq:phi_0}, we have
\begin{align*}
    &\phi_1^{0}(a)\! =\!\Big( \phi_1^{0}(0)\int_a^{+\infty}\! b(u)e^{-\int_0^u (\ks(x)+\lambda_0) \dd x} \!\dd u \\
    & \qquad \qquad +\! \phi_2^{0}(0) \int_a^{+\infty} \ks(u) e^{-\int_0^u (\ks(x)+\lambda_0) \dd x} \dd u\Big)e^{\int_0^a (\ks(x)+ \lambda_0) \dd x} \leq \phi_1^{0}(0)\frac{||b||_{\infty}}{\lambda_0} + \phi_2^{0}(0) ,\\
    &\phi_2^{0}(a) = \phi_1^{0}(0) \int_a^{+\infty} \tilde{b}(u) e^{-\int_0^u (\kd(x)+\lambda_0) \dd x} \dd u \times e^{\int_0^a (\kd(x)+\lambda_0) \dd x} \leq \phi_1^{0}(0)\frac{||\tilde{b}||_{\infty}}{\lambda_0}. 
\end{align*}
With 
\begin{equation*}
    \phi_2^{0}(0) = \phi_1^{0}(0)\int_0^{+\infty} \tilde{b}(u)e^{-\int_0^u( \kd(x) +\lambda_0 )\dd x}  \dd u \leq \phi_1^{0}(0)\frac{||\tilde{b}||_{\infty}}{\lambda_0},
\end{equation*}
and $\phi^0_1(0)$ is fixed by the last condition in \eqref{eq:phi_0}. 
These solutions are therefore positive and bounded. Furthermore, under the assumption that $b$ and $\tilde{b}$ are bounded from below by strictly positive constants $\underline{b}$ and $\underline{\tilde{b}}$, and $\ks$ and $\kd$ are bounded, we have
\begin{align*}
    &\phi_1^{0}(a) \geq \frac{\underline{b}}{\lambda_0 + ||k||_{\infty}} \phi_1^{0}(0)\int_a^{+\infty} (\ks(u) + \lambda_0)e^{-\int_0^a (\ks(x) +\lambda_0)\dd x} \dd u = \frac{\underline{b}}{\lambda_0 + ||\ks||_{\infty}} \phi_1^{0}(0),\\
    & \phi_2^{0}(a) \geq  \frac{\underline{\tilde{b}}}{\lambda_0 + ||\kd||_{\infty}} \phi_1^{0}(0) \int_a^{+\infty} (\kd(x) + \lambda_0)e^{-\int_a^u( \kd(x)+\lambda_0) \dd x} \dd u \geq \frac{\underline{\tilde{b}}}{\lambda_0 + ||\kd||_{\infty}} \phi_1^{0}(0) . 
\end{align*}
 And the solutions to \eqref{eq:phi_0} are also bounded from below by strictly positive constants. 
\end{proof}

\subsection{Study of the linear system}\label{sec:gen_lin}

In order to study the existence of solutions to the non-linear system \eqref{syst:edp}, we need to first study the corresponding linear system. 
We follow the same steps as what is done in Chapter 3 of \cite{Perthame2007_transport} for the renewal equation. 

Consider the following system, where $S_1$ and $S_2$ are locally bounded functions of time, independent of $\no$ and $\nt$. 
\begin{align}
\label{syst:edp_lin}
    \begin{cases}
        &\frac{\partial \no(t,a) }{\partial t} + \frac{\partial \no(t,a) }{\partial a} = -(\ks(a) + \CS(S_1(t),S_2(t)) + \gamo(S_1(t),S_2(t)))\no(t,a)\\
        &\no(0,t) = \int_0^{+\infty}\no(t,a)b(a)+\nt(t,a)\tilde{b}(a) \dd a\\
        &\frac{\partial \nt(t,a) }{\partial t} + \frac{\partial \nt(t,a) }{\partial a} = -(\kd(a) +\gamt(S_1(t),S_2(t)) )\nt(t,a)\\
        &\nt(0,t) = \int_0^{+\infty}\no(t,a)(\ks(a)+\CS(S_1(t),S_2(t)))\dd a\\
        &\no(a,0) = n_{1,0}(a)  \quad \nt(a,0) = n_{2,0}(a), n_{1,0}, n_{2,0} \in (L^{1}(\R_+))^2.
    \end{cases}
\end{align}

\begin{definition}[Weak solution] \label{def:weak_2ol}
    A pair of functions $(\no, \nt) \in L^1_{\text{loc}}(\R_+ \times \R_+)^2$ satisfies system \eqref{syst:edp_lin} in a weak sense if $\int_0^{+\infty} b(a) |\no(t,a)| + \tilde{b}(a) |n(t,a)| \dd a$ and $ \int_0^{+\infty} (\ks(a) + \CS(S_1(t), S_2(t)))|\no(t,a)| \dd a \in L^1_{\text{loc}}(\R_+)$ and for any $T>0$, for all pairs of test functions $(\phi_1, \phi_2) \in C^1_{\text{comp}}([0,T] \times [0, \infty))^2$ such that $\phi_1(T,x) \equiv 0, \phi_2(T,x) \equiv 0$, we have
    \begin{align*}
        -\int_0^T &\int_0^{+\infty} \no(t,a) \Big( \frac{\partial \phi_1(t,a)}{\partial t} + \frac{\partial  \phi_1(t,a)}{\partial a} - (\ks(a) + \gamtot(S_1(t), S_2(t)))  \phi_1(t,a) \Big) \dd a \dd t \\
        & = \int_0^{+\infty}n_{1,0}(a) \phi_1(a,0) \dd a + \int_0^T  \phi_1(0,t) \int_0^{+\infty} (b(a) \no(t,a) + \tilde{b}(a) \nt(t,a)) \dd a. 
    \end{align*}
    And 
    \begin{align*}
        -\int_0^T &\int_0^{+\infty} \nt(t,a) \Big( \frac{\partial \phi_2(t,a)}{\partial t} + \frac{\partial  \phi_2(t,a)}{\partial a} - (\kd(a) + \gamt(S_1(t), S_2(t)))  \phi_2(t,a) \Big) \dd a \dd t \\
        & = \int_0^{+\infty}n_{2,0}(a) \phi_2(a,0) \dd a + \int_0^T  \phi_2(0,t) \int_0^{+\infty} (\ks(a) + \CS(S_1(t), S_2(t)))\no(t,a) \dd a. 
    \end{align*}
\end{definition}
We now prove the existence and positivity of solutions to \eqref{syst:edp_lin}, as well as a comparison principle, for which the proof is presented in the Appendix \ref{sec:app_lin}. 
\begin{thm}\label{thm:exis_lin}
  Assume \ref{assump:bio},  \ref{assump:compet_0} and \ref{assump:bounds} are verified and $S_1, S_2 \in L_{\text{loc}}^{\infty}(\R_+)$. Then, there is a unique weak solution $(\no, \nt) \in  C(\R_+; L^1(\R_+))^2$  to system \eqref{syst:edp_lin}. Furthermore if $(n_1, n_2)$ and $(\tilde{n}_1, \tilde{n}_1)$ are solutions to \eqref{syst:edp_lin} starting respectively from $(n_{1,0},n_{1,0}) $ and $(\tilde{n}_{1,0} , \tilde{n}_{2,0} )$, we have the following comparison principle 
   \begin{align*}
       n_{1,0} \leq  \tilde{n}_{1,0} \text{ and }n_{2,0} \leq  \tilde{n}_{2,0} \implies  \no(t,.) \leq  \tilde{\no}(t,.) \text{ and }\nt(t,.) \leq  \tilde{\nt}(t,.). 
   \end{align*}
   In particular, this means that the solutions to \eqref{syst:edp_lin} stay positive provided the initial conditions are positive.
   And 
   \begin{equation}
       \int_0^{+\infty}(|\no(t,a)| + |\nt(t,a)| )\dd a \leq e^{\max(||b||_{\infty},||\tilde{b}||_{\infty})t}\int_0^{+\infty}(|n_{1,0}(a)| + |n_{2,0}(a)|) \dd a. \label{eq:exp_bound_L1}
   \end{equation}
\end{thm}

\begin{remark}  
As $k$ plays the role of a "birth" rate for phase~2 in our model, we need to assume its boundedness for mathematical reasons. However, this does not seem likely from a biological perspective, as it is also a "death" rate for individuals in phase~1. Assuming a bounded birth rate is the most classical way to determine the existence of weak solutions \cite{gurtin1974non, Perthame2007_transport,pruss1981equilibrium}. However, it can easily be replaced by assuming that the initial condition has a compact support as in \cite{Huang_HIV_12} and that the rates are locally bounded. 
\end{remark}

Finally, we state the following result, which shows that the solutions to the system when $\gamtot = \gamt = 0$ i.e without competition, in the linear and homogeneous case, converge to a fixed age distribution and for which the proof is presented in the Appendix \ref{sec:app_lin}. This results relies on general relative entropy (GRE) methods \cite{MICHEL20051235,Perthame2007_transport}. These methods allow to get results on long time convergence of the system, mainly for linear systems although some results tackle the non-linear case \cite{Pakdaman_09,perthame2025stronglynonlinearagestructured,Torres_2022,TUMULURI20111420}. 

\begin{thm}[Long time convergence]\label{thm:asymp_linear_b}
Under Assumptions \ref{assump:bio}, \ref{assump:bounds} and \ref{ass:for_lambda_0}, if $\gamtot \equiv \gamt \equiv 0$ and if there exists $\mu > 0$ such that $\forall a \geq 0$,
\begin{equation*}
    b(a) \geq \mu \frac{\phi^0_1(a)}{\phi^0_1(0)} \quad \text{and}  \quad \tilde{b}(a) \geq \mu \frac{\phi^0_2(a)}{\phi^0_2(0)},
\end{equation*}
where $\phi_1^0$ and $\phi_2^0$ are defined as in Proposition \ref{prop:ex_phi}. 
Then the weak solutions $(\no, \nt) \in  C(\R_+; L^1(\R_+))^2$ to \eqref{syst:edp_lin} verify
  \begin{equation*}
    \int_0^{+\infty} \!\Big(|h_1(t,a)| \phi^0_1(a) \!+ \!|h_2(t,a)| \phi^0_2(a) \Big)\dd a  \! \leq \! e^{-\mu t} \! \Big(\int_0^{+\infty} \!\Big( |h_1(a,0)| \phi^0_1(a)\!+\! |h_2(a,0)| \phi^0_2(a) \Big) \! \dd a \!\Big),
\end{equation*}
where $m^0 = \int_0^{+\infty} \no(a,0) \phi^0_1(a) + \nt(a,0) \phi^0_2(a) \dd a$ and
\begin{align*}
   &h_1(t,a) = e^{-\lambda_0 t} \no(t,a) - m^0n^0_1(a), \qquad h_2(t,a) = e^{-\lambda_0 t} \nt(t,a) - m^0n^0_2(a).
\end{align*}
\end{thm}

\subsection{Existence and uniqueness of solutions in the general case}
\label{sec:ex_un}
In this section, we turn back to the full non-linear model \eqref{syst:edp} and rely on the results shown in Section \ref{sec:gen_lin} for the linear case to prove existence and uniqueness of solutions along with a priori bounds as is done in \cite{Perthame_Tumulari2008}. 
We introduce the following assumption. 
\begin{assumption}\label{assump:lip}
$\exists L>0, \forall S_{1}, S_{2},M_{1}, M_{2} \geq 0 ,$
    \begin{align*}
      &|\gamt(S_{1}, S_{2}) - \gamt(M_1, M_2)| \leq L(| S_{1}- M_{1}| +|S_{2} -M_{2}|), \\
    &|\gamtot(S_{1}, S_{2}) - \gamtot(M_1, M_2)| \leq L(| S_{1}- M_{1}| +|S_{2} -M_{2}|).
    \end{align*}
\end{assumption}
We begin with an existence result, for which the proof is postponed to the Appendix \ref{sec:app_ex}. 
\begin{thm}\label{them:ex}
 Suppose $\psi_1, \psi_2$ (as defined by \eqref{eq:def_2}) are positive and bounded. Under Assumptions \ref{assump:bio}, \ref{assump:compet_0}, \ref{assump:bounds} and \ref{assump:lip}, system \eqref{syst:edp} has a unique weak solution $(\no, \nt) \in  C(\R_+; L^1(\R_+))^2$.
\end{thm}

We can also deduce from the comparison principle of the linear system the positivity of the solutions. The proof is presented in the Appendix \ref{sec:app_ex}. 

\begin{prop}\label{prop:positivity}
Suppose $\psi_1, \psi_2$ are positive and bounded. Under Assumptions \ref{assump:bio}, \ref{assump:compet_0}, \ref{assump:bounds} and \ref{assump:lip}, the solutions to system \eqref{syst:edp} are positive.
\end{prop}

We introduce the following assumptions ensuring the fact that competition increases with the environmental factor (typically, a population size), in line with its biological intepretation. 
\begin{assumption} 
\label{assump:partial}
\begin{align}
    & \frac{\partial \gamtot}{\partial S_1} > 0, \frac{\partial \gamtot}{\partial S_2} \geq 0 \text{ and }\frac{\partial \gamt}{\partial S_1} \geq 0, \frac{\partial \gamt}{\partial S_2} > 0 \label{eq:partial_compet}\\
    & \forall S_2, S_1 \geq 0, \gamt(S_1, x) \xrightarrow[x \to + \infty]{}+\infty, \gamo(x, S_2)\xrightarrow[x \to + \infty]{}+\infty \label{eq:compet_infinity} 
\end{align}
\end{assumption}

\begin{assumption}\label{assump:psi_phi}
Let $\phi^0_1$ and $\phi^0_2$ be the adjoint eigenfunctions defined by \eqref{eq:phi_0}, 
\begin{align*}
    &\exists c_1^0, C_1^0 > 0, c_1^0\phi^0_1(a) \leq \psi_1(a) \leq C_1^0\phi^0_1(a) \\
     &\exists c_2^0, C_2^0 > 0, c_2^0\phi^0_2(a) \leq \psi_2(a) \leq C_2^0\phi^0_2(a). 
\end{align*}
\end{assumption}

\begin{assumption}\label{assump:psi_const}
\begin{align*}
    &\exists c_1, C_1 > 0, c_1 \leq \psi_1(a) \leq C_1\\
     &\exists c_2, C_2 > 0, c_2\leq \psi_2(a) \leq C_2. 
\end{align*}
\end{assumption}

 In the case where $\phi_1^0$ and  $\phi_2^0$ are bounded from above and below by a strictly positive constant, Assumption \ref{assump:psi_phi} is not a very strong requirement and is trivially verified for  competition kernels that are also bounded from above and below by a strictly positive constant. This is the case under the assumptions of Proposition \ref{prop:ex_phi}.

\begin{assumption}\label{assump:sub_tot_compet}
$\forall A >0, \exists \alpha_A,\beta_A >0,  \CS(A,.) \leq  \beta_A+ \alpha_A \gamt(0, .)$. 
\end{assumption}

We follow up with a result on the boundedness of $S_1$ and $S_2$. 

\begin{prop} \label{prop:bound_wcompker}
    Let $(S_1(t), S_2(t))$ be defined by \eqref{eq:def_2} where $\no,\nt \in  C(\R_+; L^1(\R_+))^2$ are solutions of \eqref{syst:edp}. Suppose $\int_0^{+\infty} \phi^0_1(a)\no(a,0)+ \phi^0_2(a)\nt(a,0)\dd a > 0$ and Assumptions \ref{assump:bio} to \ref{assump:sub_tot_compet} are verified. Then, 
    \begin{equation*}
        \exists m, M \geq0, \forall t \geq 0, m < S_1(t) + S_2(t) < M.
    \end{equation*}
\end{prop}
\begin{proof}
We introduce $N(t) := \int_0^{+\infty} n_1(t,a) \dd a +\int_0^{+\infty} n_2(t,a) \dd a := N_1(t) + N_2(t) $. We have by Assumptions \ref{assump:partial} and \ref{assump:psi_const}, 
\begin{align*}
    \frac{\dd N(t)}{\dd t} &\leq \max(||b||_{\infty},||\tilde{b}||_{\infty} ) N(t)  - \gamt(0, S_2(t)) N_2(t) - \gamo(S_1(t), 0) N_1(t)\\
    & \leq  \max(||b||_{\infty},||\tilde{b}||_{\infty} ) N(t)  - \gamt(0, c_2 N_2(t)) N_2(t) - \gamo(c_1 N_1(t), 0) N_1(t). 
\end{align*}
The right hand side of the previous equation goes to $-\infty$ when $N_0 \to \infty$, and has a unique strictly positive root (by Assumptions \ref{assump:compet_0} and \ref{assump:partial}). Thus, we can apply Lemma \ref{lemma:eq_diff} in the Appendix to show that  $N(t) \leq 2M$ where $M$ is such that $2\max(||b||_{\infty},||\tilde{b}||_{\infty} ) - \min(\gamt(0, c_2 M), \gamo(c_1 M, 0)) < 0$. Finally,  Assumption \ref{assump:psi_const} yields
\begin{equation*}
    S_1(t) + S_2(t) \leq \max(C_1, C_2) N(t) \leq  2\max(C_1, C_2)M. 
\end{equation*}

We now prove the lower bound.
This part of the proof relies on a technique based on the adjoint problem, also used in \cite{Perthame_Tumulari2008}.  
We begin by introducing
      \begin{equation*}
          N_{1,0}(t) = \int_0^{+\infty} \phi^0_1(a) \no(t,a) \dd a, \qquad N_{2,0} = \int_0^{+\infty} \phi^0_2(a) \nt(t,a) \dd a, \qquad N_0 = N_{1,0} + N_{2,0}, 
      \end{equation*}
where  $\phi^0_1$ and $\phi^0_2$ are defined by \eqref{eq:phi_0}. We have
      \begin{align*}
          \frac{\dd N_0(t)}{\dd t} 
           & \geq  \lambda_0 N_0(t) - \gamtot(S_1(t),S_2(t)) N_{1,0}(t) - \gamt(S_1(t),S_2(t)) N_{2,0}(t) \\
           & \geq \lambda_0 N_0(t) - (\gamtot(C_1^0 N_0(t),C_2^0 N_{0}(t)) + \gamt(C_1^0 N_0(t),C_2^0 N_{0}(t)))N_0(t). 
      \end{align*}
       By Assumptions \ref{assump:compet_0} and \ref{assump:partial}, the right hand side of the previous equation has a unique strictly positive root $\tilde{N_0}$ and is positive on $[0, \tilde{N_0}]$. Thus, we can apply Lemma \ref{lemma:eq_diff} in the Appendix, which yields
      \begin{equation*}
          N_0(t) \geq \min\Big(N_0(0), \Tilde{N}_0 \Big) > 0. 
      \end{equation*}
     Therefore finally,
      \begin{equation*}
          S_1(t) + S_2(t) \geq \frac{1}{\max(c^{0}_1, c^0_2 ) }N_0(t) \geq  \frac{\min\Big(N_0(0), \Tilde{N}_0 \Big)}{\max(c^{0}_1, c^0_2 )},
      \end{equation*}
which shows the expected lower bound. 
\end{proof}

The proof for the lower bound in the previous proposition requires the assumption $\tilde{b} \not\equiv 0$, else $\phi^0_2 \equiv 0$ and the upper bound for $\psi_2$ in \ref{assump:psi_phi} does not hold. As this assumption is not needed for the upper bound, it remains true even when $\tilde{b} \equiv 0$. 

 \subsection{Characterisation of steady states}
 
Now that we have obtained existence, uniqueness and positivity of solutions to \eqref{syst:edp}, we wish to study the existence of steady state solutions.
In the rest of the article, we assume that the competition kernels are identically one i.e
\begin{align*}
    &S_1(t) = N_1(t) := \int_0^{+\infty} \no(t,a) \dd a,\qquad S_2(t) = N_2(t) := \int_0^{+\infty} \nt(t,a) \dd a.
\end{align*}
Defined as such, $N_1$ and $N_2$ represent respectively the total population in phases 1 and 2. Furthermore, constant competition kernels verify Assumption \ref{assump:psi_const} as well as \ref{assump:psi_phi} if the solutions to the adjoint system \eqref{eq:phi_0} are bounded from above and below (which is true under the assumptions stated in Proposition \ref{prop:ex_phi}).

By solving \eqref{syst:edp} at steady state and plugging back into the birth condition, we find that a non-trivial steady state $(N_1^*,N_2^*)$ to \eqref{syst:edp} exists iff it verifies the following system

\begin{subnumcases}
    {} \int_0^{+\infty}\!\tilde{b}(a) e^{-\int_0^{a} \kd(u) \dd u - \gamt(N_1,N_2 )a} \dd a\! \cdot\! \int_0^{+\infty}\!(\ks(a) \! + \!\CS(N_1,N_2)) e^{-\int_0^{a} \ks(u) \dd u -\gamtot(N_1,N_2) a} \dd a \nonumber\\
    \qquad +  \int_0^{+\infty}b(a) e^{-\int_0^{a} \ks(u) \dd u-\gamtot(N_1,N_2)a} \dd a = 1 \label{eq:steady_full}\\
    N_1 \int_0^{+\infty} e^{-\int_0^{a} \kd(u) \dd u - \gamt(N_1,N_2)a} \dd a \cdot \int_0^{+\infty}(\ks(a) \! +\!\gamt(N_1,N_2)) e^{-\int_0^{a} \ks(u)  \dd u-\gamtot(N_1,N_2)a} \dd a \nonumber \\
   \qquad - N_2 \int_0^{+\infty}e^{-\int_0^{a} \ks(u)  \dd u-\gamtot(N_1,N_2)a} \dd a  = 0
\end{subnumcases}

We begin with a short proposition comparing the (potential) steady state population size of the 2-phase model with the population size of the same model obtained by setting $\nt \equiv  0$ in \eqref{syst:edp}.
We introduce the following system
\begin{equation}
\label{eq:syst_one_phase}
    \begin{cases}
         &\frac{\partial n(t,a) }{\partial t} + \frac{\partial n(t,a) }{\partial a} = -(\ks(a) + \gamtot(N(t),0))n(t,a)\\
        &n(0,t) = \int_0^{+\infty}b(a)n(t,a)\dd a\\
        &n(a,0) = n_{0}(a).
    \end{cases}
\end{equation}

Such models have been extensively studied (for example in  \cite{FARKAS2004771,Iannelli_17,Perthame2007_transport,pruss1981equilibrium,webb1985theory})
and under the assumption that $\int_0^{+\infty} b(a) e^{-\int_0^{a} \ks(u) \dd u } \dd a >1$, the model admits a strictly positive stable steady state for which the total population size verifies $ \gamtot(N^*, 0) = \lambda_0$ where $\lambda_0>0$ is such that  $\int_0^{+\infty} b(a) e^{-\int_0^{a} \ks(u) \dd u - \lambda_0 a} \dd a = 1$. If $\int_0^{+\infty} b(a) e^{-\int_0^{a} \ks(u) \dd u } \dd a < 1$, the only stable steady state is such that $\gamtot(N^*,0) = 0$ (see \cite{pruss1981equilibrium} Theorem 3 and Corollary 2).

\begin{prop}\label{prop:comp_2s}
Under Assumptions \ref{assump:bio}, \ref{assump:compet_0}, \ref{ass:for_lambda_0}, \ref{assump:partial} and if 
    \begin{equation*}
        \left \lVert \frac{\partial \gamtot}{\partial y} \right \rVert_{\infty} \leq \inf_{(x,y)\in \R_+^2}\left \{ \frac{\partial \gamtot}{\partial x}(x,y) \right \},
    \end{equation*}
    all the solutions $(N_1^*, N_2^*)$ to \eqref{eq:steady_full} are such that 
    \begin{equation*}
        N_1^* + N_2^* \geq N^*,
    \end{equation*}
where $N^*$ is the population size of the unique steady state of \eqref{eq:syst_one_phase}.    
\end{prop}

\begin{proof}
    We begin by showing that non-trivial solutions to \eqref{eq:steady_full} exist. We first introduce $f(x, y)$ such that \eqref{eq:steady_full} can be rewritten $f(N_1^*, N_2^*) = 1$. We notice that $f$ is continuous in $x$ and $y$ on $\R_+$ as the exponential functions under the integrals can be dominated by their value when $x = y = 0$ and the term in $\CS$ out of the exponential can be factorised out of the integral. 
    By Assumption \ref{ass:for_lambda_0}, $f(0,0) >1$ and $f(x,0) \xrightarrow[x \to + \infty]{} 0$. Thus there is a strictly positive value $N_1^*$ such that $f(0,N_1^*) =1$. 
    This shows the existence of non-trivial solutions to \eqref{eq:steady_full}.

Let $(N_1^*, N_2^*)$ be a non-trivial solution of \eqref{eq:steady_full}. 

\textbf{Case 1 : $\int_0^{+\infty} b(a) e^{-\int_0^{a} \ks(u) \dd u } \dd a >1$}.

Suppose $\gamtot( N_1^*, N_2^*) <\lambda _0$, then we have 
    \begin{equation*}
        \int_0^{+\infty}b(a) e^{-\int_0^{a} \ks(u)\dd u - \gamtot( N_1^*, N_2^*) a} \dd a > 1 
    \end{equation*}
    hence $(N_1^*, N_2^*)$ cannot be a solution of \eqref{eq:steady_full}. Thus, any solution of \eqref{eq:steady_full} is such that $\gamtot( N_1^*, N_2^*) \geq \lambda_0 = \gamtot(N^*, 0)$. By monotonicity by Assumption \eqref{eq:partial_compet}, there cannot be both $N_1^* < N^*$ and $N_2^* = 0$. If  $N_1^* \geq N^*$,  then $N_1^* + N_2^* \geq  N^*$ and the result is shown. If $N_1^* < N^*$ and $N_2^* > 0$, we have
\begin{align*}
     0 &\leq  \gamtot( N_1^*, N_2^*) - \gamtot(N^*, 0) \\
    & =  \gamtot(N_1^*, N_2^*) - \gamtot(N_1^*,0) + \gamtot(N_1^*,0)  - \gamtot(N^*, 0)\\
    & = \int_0^{N_2^*}\frac{\partial \gamtot}{\partial y}(N_1^*,y) \dd y - \int_{N_1^*}^{N^*}\frac{\partial \gamtot}{\partial x}(x,0) \dd x \\
    & \leq N_2^*   \left \lVert \frac{\partial \gamtot}{\partial y} \right \rVert_{\infty} - (N^* - N_1^*) \inf_{(x,y) \in \R_+^2}\left \{ \frac{\partial \gamtot}{\partial x}(x,y) \right \}\\
    & \leq (N_2^* + N_1^* - N^* )\inf_{(x,y) \in \R_+^2}\left \{ \frac{\partial \gamtot}{\partial x}(x,y) \right \}.
\end{align*}
    Thus $N_2^* + N_1^* \geq  N^*$ and the result follows. 
    
\textbf{Case 2 : $\int_0^{+\infty} b(a) e^{-\int_0^{a} \ks(u) \dd u } \dd a \leq 1$}.
In this case, the 1-dimensional system only admits $(0,0)$ as a positive steady state (see Theorem 3 in \cite{pruss1981equilibrium}) and thus $N_1^* + N_2^* \geq N^* = 0$. 
\end{proof}

Proposition \ref{prop:comp_2s} can be interpreted as follows: 
the presence of older diminished individuals that are less competitive in a population results in a larger population size at equilibrium, compared to the same population where everyone has the same competitivity.

The result shown in Proposition \ref{prop:comp_2s} is true for the solutions of \eqref{eq:steady_full} and therefore true for any steady state of the system, provided it exists. However, the proposition does not  prove the existence of such a steady state. 
Indeed, the presence of $\tilde{b} \neq 0$ as well as the different competition terms make this system impossible to study directly with monotonicity arguments as is the case in \cite{covei2023nonlinearpopulationmodel,TUMULURI20111420} for example.  

In the following, we assume that the competition functions are linear in population sizes i.e.

\begin{assumption}\label{assump:positive_compet}
\begin{equation}
\label{eq:lin_compet}
\begin{split}
    & \exists \CSS, \CSNS \in \R_+^2, \CS(N_1, N_2) = \CSS N_2 + \CSNS N_1\\
    & \exists \CDS, \CDNS \in \R_+^2, \gamo(N_1, N_2) = \CDS N_2 + \CDNS N_1\\
    & \exists \CDSS, \CDSNS \in \R_+^2, \gamt(N_1, N_2) = \CDSS N_2 + \CDSNS N_1. 
    \end{split}
\end{equation}
\end{assumption}
For ease of notation, we introduce 
\begin{align*}
    & \CtotNS  = \CSNS + \CDNS \text{ and } \CtotS  = \CSS + \CDS.
\end{align*}

\begin{remark} \label{rem:compet_less}
\begin{itemize}
    \item  Linear competition rates verify Assumptions \ref{assump:compet_0}, \ref{assump:lip} and \ref{assump:sub_tot_compet} as well as  \ref{assump:partial} provided $\CtotNS  > 0$ and $ \CDSS >0$.
    \item The biological motivation behind this work is that individuals in phase 2 are old, diminished and close to death. As such, it would be reasonable to assume that individuals in phase 1 are not very affected by the competition of individuals in phase 2, meaning that $\CSS \ll \CSNS$ and $\CDS \ll \CDNS$. 
\end{itemize}
   
\end{remark}

Proposition \ref{prop:comp_2s} can be directly applied to linear competition rates: 
\begin{corollary}\label{cor:comp_2s_lin}
    Under Assumption \ref{ass:for_lambda_0}, if $\CtotNS  > 0, \CDSS >0$ and $\CtotNS  \geq \CtotS $. Then, all the solutions $(N_1^*, N_2^*)$ to \eqref{eq:steady_full} are such that 
    \begin{equation*}
        N_1^* + N_2^* \geq N^*.
    \end{equation*}
\end{corollary}
In this section, we have shown general existence results for system \eqref{syst:edp}. We now turn to the study of simplified versions of the system, for which we can show results on the existence, uniqueness and stability of steady states.

\section{Long-time asymptotics in the ODE case}
\label{sec:ODE}
In this section, we show existence, uniqueness and stability of a strictly positive steady state for a simplified version of system \eqref{syst:edp} were all rates are constant (and positive) thus reducing the system to two ODEs. The system we consider is the following
\begin{align}
    \begin{cases}
        & \frac{\dd N_1(t)}{\dd t} = -(\ks + (\CSNS + \CDNS) N_1(t) + (\CSS + \CDS) N_2(t)) N_1(t)) + bN_1(t) + \tilde{b} N_2(t)\\
        &\frac{\dd N_2(t)}{\dd t} = -(\kd+ (\CDSNS N_1(t)+\CDSS N_2(t)) N_2(t)) + (\ks+\CSNS N_1(t) + \CSS N_2(t))N_1(t)\\
       &N_1(0) = N_{1,0} \geq 0, \quad N_2(0) = N_{2,0} \geq 0. \label{eq:syst_const}
    \end{cases}
\end{align}

To simplify the computations, and as biological evidence shows that $\tilde{b}$ is very small, we allow ourselves to take $\tilde{b}$ as small as necessary compared to the other constants in the system during the proofs.

\begin{thm} \label{thm:orbit}
     Under Assumptions \ref{assump:bio} and \ref{assump:positive_compet}, system \eqref{eq:syst_const} has no periodic orbit lying entirely in $\R_+^2$. 
\end{thm}

\begin{proof}
    To prove the result, we use Dulac's Theorem (see Proposition 7.2.6. in \cite{schaeffer2016ordinary}). The theorem states that if there is a $C^1$ function $g:(\R_+^*)^2 \to \R^2$ such that the divergence $\frac{\partial }{\partial N_1}(g(N_1,N_2) \frac{\dd N_1}{\dd t}) +\frac{\partial }{\partial N_2}(g(N_1,N_2) \frac{\dd N_2}{\dd t})$ is non-negative and not identically $0$ on any open subset of $(\R_+^*)^2$, then system \eqref{eq:syst_const} has no periodic solution lying entirely in $(\R_+^*)^2$. 

    We introduce $g(x,y) := \frac{1}{xy}$, $g$ is defined and $C^1$ on $(\R_+^*)^2$ and we have 
    \begin{align*}
        \frac{\partial}{\partial N_1}(g(N_1,N_2) \frac{\dd N_1}{\dd t}) &=  \frac{\partial}{\partial N_1}\Big(-\frac{\ks}{N_2} - \frac{\CtotNS N_1}{N_2} - \CtotS + \frac{b}{N_2} + \frac{\tilde{b}}{N_1}\Big)\\
        & = -\frac{\CtotNS}{N_2} - \frac{\tilde{b}}{N_1^2} < 0. 
    \end{align*}
Furthermore, since Assumption \ref{assump:bio} implies $k > 0$, 
 \begin{align*}
        \frac{\partial}{\partial N_2}(g(N_1,N_2) \frac{\dd N_2}{\dd t}) &=  \frac{\partial}{\partial N_2}\Big(-\frac{\kd}{N_1} - \frac{\CDSS N_2}{N_1} - \CDSNS + \frac{\ks}{N_2} + \frac{\CSNS}{N_1} + \CSS \Big)\\
        & = -\frac{\CDSS}{N_1} - \frac{\ks}{N_2^2} < 0. 
    \end{align*}

Thus $\frac{\partial d}{\partial N_1}(g(N_1,N_2) \frac{\dd N_1}{\dd t}) +\frac{\partial d}{\partial N_2}(g(N_1,N_2) \frac{\dd N_2}{\dd t}) < 0$ and system \eqref{eq:syst_const} has no periodic orbit in $(\R_+^*)^2$. Since $(0,0)$ is a steady state of \eqref{eq:syst_const}, it cannot be on a periodic orbit. Thus, there is no periodic orbit entirely within $\R_+^2$. 
\end{proof}

\begin{thm} \label{thm:ODE_pos}
     Assume $\CtotNS  > 0, \CDSS >0$,and  $b/\ks + \tilde{b}/\kd > 1 $, and $N_1(0) > 0$.
     Under Assumptions \ref{assump:bio} and \ref{assump:positive_compet} and for $\tilde{b}$ small enough, there is a single globally stable positive steady state $(N_1^*, N_2^*)$ to \eqref{eq:syst_const}.
     Furthermore, if $b-\ks < 0$, $N_1^* $ and $N_2^*$ are of order $O(\tilde{b})$ when $\tilde{b} \to 0$. 

\end{thm}
\begin{proof}
\textbf{Step 1: strict positivity of the steady state(s)}. Since $\frac{b}{\ks} + \frac{\tilde{b}}{\kd}>1$, Assumption \ref{ass:for_lambda_0} is verified and we can apply the result for the upper bound shown in Proposition \ref{prop:bound_wcompker} and by Proposition \ref{prop:positivity}, the solutions to \eqref{eq:syst_const} stay positive. Thus, the system is non-negative and cannot go to infinity. By the Poincaré-Bendinxson Theorem (see e.g. \cite{strogatz2024nonlinear} p.150) there is either at least one stable positive steady state or a periodic orbit. As the solutions to \eqref{eq:syst_const}  are positive and  periodic orbits in $\R_+^2$ were ruled out by Theorem \ref{thm:orbit}, there is at least one stable positive steady state. 

We first rule out $(0,0)$ as a stable steady state.
If $\tilde{b} \neq 0$ and $b \neq 0$, the strictly positive lower bound of Proposition \ref{prop:bound_wcompker} also holds and $(0,0)$ cannot be a stable steady state. 
If $\tilde{b} = 0$, the linearisation of the system around $(0,0)$ shows that the eigenvalues are $b-\ks$ and $\kd$, which are both positive (since $\tilde{b} = 0$, $\frac{b}{\ks} > 1$ by assumption) and thus $(0,0)$ is not stable. 
If $b = 0$, the eigenvalues of the linearised system at $(0,0)$ are $\frac{1}{2}( \pm \sqrt{(\ks - \kd)^2 +4\ks \tilde{b}} - (\ks + \kd))$, one of which is positive and $(0,0)$ is not a stable steady state.   

Pairs where either $N_1^* = 0$ or $N_2^* = 0$ are never steady states of our system. We are therefore looking for steady states $(N_1^*, N_2^*)$ where each component is strictly positive.\\ 

\textbf{Step 2: uniqueness of the steady states.}
By setting system \eqref{eq:syst_const} to steady state, we find that potential steady state(s) $(N_1^*, N_2^*)$ are characterised by the following polynomial system 
    \begin{subnumcases}
       {} \CtotNS  (N_1^*)^2 + N_1^*(\ks + \CtotS  N_2^* - b) -\tilde{b} N_2^* = 0     \label{eq:syst_2tead_const1}\\
       -\CSNS(N_1^*)^2 + N_1^*(N_2^*(\CDSNS - \CSS) - \ks) +(N_2^*)^2\CDSS + N_2^* \kd = 0.     \label{eq:syst_2tead_const2}
    \end{subnumcases}

We now study the existence of positive values $N_2^*$ such that equations \eqref{eq:syst_2tead_const1} and \eqref{eq:syst_2tead_const2} considered as polynomials in $N_1^*$ have a common root. To do so, we study the resultant of \eqref{eq:syst_2tead_const1} and \eqref{eq:syst_2tead_const2} (considered as polynomials in $N_1^*$). The resultant is a polynomial function of the coefficients, and thus a polynomial in $N_2^*$. The values of $N_2^*$ for which the resultant is $0$ are exactly the values of $N_2^*$ for which \eqref{eq:syst_2tead_const1} and \eqref{eq:syst_2tead_const2} have at least one common root $N_1^*$, positive or negative. There will then remain to determine for which positive values of $N_2^*$ such that the resultant cancels out the common root is positive. 

The resultant of the two polynomials is 
\begin{align*}
     &\begin{vmatrix}
         \CtotNS  &  0& -\CSNS &0\\
        \ks-b + \CtotS  N_2^* & \CtotNS & N_2^*(\CDSNS - \CSS) -\ks& -\CSNS\\
        -\tilde{b} N_2^* & \ks-b + \CtotS N_2^*& (N_2^*)^2 \CDSS + N_2^* \kd&N_2^*(\CDSNS - \CSS) -\ks \\
        0 & -\tilde{b}N_2^* & 0 & (N_2^*)^2 \CDSS + N_2^* \kd
    \end{vmatrix} \\
    & =: -N_2^* [A_3 (N_2^*)^3 + A_2(N_2^*)^2 + A_1 N_2^* + A_0]\\
    & =: -N_2^* P(N_2^*), 
\end{align*}
where 
\begin{align*}
    &A_3 = \CtotNS \CtotS  \CDSS(\CDSNS - \CSS) - (\CtotNS )^2 (\CDSS)^2+ \CSNS(\CtotS )^2\CDSS,\\
    &A_2 = 2\CtotNS \CSNS \tilde{b} \CDSS+ \CtotNS \CtotS \kd(\CDSNS - \CSS) - \CtotNS \CDSS((b-\ks)(\CDSNS - \CSS) + \ks\CtotS ) \\
    &\quad - 2 \kd (\CtotNS )^2 \CDSS + \tilde{b}\CtotNS (\CDSNS - \CSS)^2 +\tilde{b}(\CDSNS - \CSS)\CSNS\CtotS  
   + \CSNS \kd(\CtotS )^2 \\
   & \quad - 2(b-\ks)\CtotS \CSNS\CDSS, \\
    &A_1 = \CtotNS \tilde{b}\ks\CSNS -\CtotNS \kd((b-\ks)(\CDSNS - \CSS) + \ks \CtotS ) +  \CtotNS \CDSS \ks(b-\ks) \\
    & \quad - \kd^2 (\CtotNS )^2 - 2\tilde{b}\ks\CtotNS (\CDSNS - \CSS) -  \CSNS \tilde{b} ((b-\ks)(\CDSNS - \CSS) + \ks\CtotS ) \\
    & \quad + \CSNS \kd \tilde{b} \CtotNS  - 2 (b-\ks)\kd \CSNS\CtotS  + \CSNS\CDSS (b-\ks)^2 +(\tilde{b})^2(\CSNS)^2,\\
    & A_0 = \ks \kd\Big(\frac{\tilde{b}}{\kd}+ \frac{b}{\ks}-1\Big)(\CDNS\ks + b\CSNS).
\end{align*}
We notice that $A_0 > 0$ since $\frac{\tilde{b}}{\kd} + \frac{b}{\ks} > 1$ by assumption.

\textbf{Case 1 : $b-k > 0$}\\
\textbf{Subcase 1 : $A_3 > 0$.}
This condition implies $\CtotS  > 0$. By Descartes' rule of sign, if $A_0$ and $A_3$ have the same sign, there are either $0$ or $2$ positive roots to the equation $A_3 X^3 + A_2X^2 + A_1 X + A_0=0$ depending on the sign of the discriminant. However, as we know that there must be a positive steady state, there must be 2 roots to the resultant in that case. We denote these two roots $M_2 \leq  N_2$ and $M_1, N_1$ the corresponding positive solutions of \eqref{eq:syst_2tead_const1} (one can easily check that $M_1$ and $N_1$ are unique).

 We notice that $P$ evaluated at $Y = \frac{b-\ks}{\CtotS }$ is such that
\begin{align*}
    P\Big(\frac{b-\ks}{\CtotS }\Big) =&  -Y\Big( \CtotNS  (Y\CDSS + \kd) + \tilde{b}\CSNS \Big)^2 + \tilde{b}\CtotNS (Y(\CDSNS - \CSNS) - \ks)^2  
\end{align*}
which is strictly smaller than a negative constant for $\tilde{b}$ small enough. 
Hence, since we are considering the case $A_3 >0$ and $b-\ks > 0$, necessarily $\frac{b-\ks}{\CtotS} \in ]M_2, N_2[$. More precisely, we have two positive constants $\alpha, \beta$ such that 
\begin{equation*}
    M_2 < \alpha < \frac{b-\ks}{\CtotS } < \beta < N_2. 
\end{equation*}
We know that to $N_2$ and $M_2$ corresponds at least one common root of \eqref{eq:syst_2tead_const1} and \eqref{eq:syst_2tead_const2} seen as polynomials in $N_1^*$, but this common root is not necessarily positive. 
We now study the positive solutions $M_1$ and $N_1$ to \eqref{eq:syst_2tead_const1} corresponding respectively to $M_2$ and $N_2$, and check wether they also verify \eqref{eq:syst_2tead_const2}. We have, 
\begin{align}
    &M_1 = \frac{b-\ks- \CtotS  M_2}{2\CtotNS }\Big(1+\sqrt{1 + 4\frac{\CtotNS \tilde{b}M_2}{(\ks - b + \CtotS  M_2)^2}} \Big) \geq \frac{b - \ks - \CtotS  \alpha}{\CtotNS } \label{eq:M1},\\
    &N_1 =\frac{b - \ks - \CtotS  N_2}{2\CtotNS }\Big(1-\sqrt{1 + 4\frac{\CtotNS \tilde{b}N_2}{(\ks - b + \CtotS  N_2)^2}} \Big) \leq  \frac{2\tilde{b}N_2}{\ks - b + \CtotS  \beta}.
\end{align}
Plugging back $N_1$ into \eqref{eq:syst_2tead_const2} and using the previous upper bound yields
\begin{align*}
    0 \geq &N_2^2\Big(\CDSS - 4\CSNS(\tilde{b})^2 \frac{1}{(\ks-b +\CtotS \beta)^2} + \frac{4\tilde{b}}{\ks-b +\CtotS \beta}\min(0, (\CDSNS - \CSS)) \Big)\\
    & + N_2\Big(\kd - \frac{2\tilde{b}\ks}{\ks-b +\CtotS \beta}\Big). 
\end{align*}
We see that for $\tilde{b}$ small enough such that both coefficients are positive, the previous inequality cannot hold since $N_2 >0$ and thus $N_1$ cannot be a root of \eqref{eq:syst_2tead_const2} (the common root of both equations for that value of $N_2$ is therefore a negative root), hence $(N_1, N_2)$ cannot be a steady state. 

Once again invoking the Poincaré-Bendinxson Theorem,  $(M_1, M_2)$ has to be a steady state, and thus solution of \eqref{eq:syst_2tead_const1} and \eqref{eq:syst_2tead_const2}. Since there are no periodic orbits, it is globally asymptotically stable. 
\textbf{Subcase 2 :$A_3 < 0$.} We notice that
\begin{align*}
    A_2 &= \frac{\kd}{\CDSS}A_3 + \frac{\CDSS}{\kd}A_1 -\frac{(b-\ks)(\CDSS)^2}{\kd}(b\CSNS + \ks\CDNS)\\
    &  + ( \CDSS \CSNS \CtotNS  +(\CDSNS - \CSS)[\CtotNS (\CDSNS - \CSS) + \CtotS  \CSNS + 2\CtotNS \frac{\ks}{\kd}\CDSS])\tilde{b}\\
    &+ \CSNS\frac{\CDSS}{\kd} ((b-\ks)(\CDSNS - \CSS) + \ks\CtotS )  \tilde{b} -(\tilde{b})^2\frac{\CDSS}{\kd} (\CSNS)^2. 
\end{align*}
Hence for $\tilde{b}$ small enough, we cannot have $ A_3 \leq 0, A_1 \leq 0,$ and $A_2 \geq 0$ since $b-\ks > 0$. Thus either $A_1 \geq 0$ and $A_2 \leq 0$ or $A_2$ and $A_1$ have the same sign.

In any case, by Descarte's rule of sign, there is only one positive root $M_2$ to the resultant hence one positive steady state.

\textbf{Subcase 3 : $A_3 = 0$}. Since by assumption $ \CtotNS  \neq 0 , \CDSS \neq 0 \implies \CtotS  \neq 0$,  $A_3 = 0 \implies \CtotS  \neq 0$. Thus at first order in $\tilde{b}$, we can express $A_2$ as 
\begin{equation*}
    A_2 = \Big(\frac{\kd}{\CDSS}-\frac{(b-\ks)}{\CtotS } \Big)A_3 - (b-\ks) \CtotS  \CSNS \CDSS -\frac{(b-\ks)}{\CtotS }(\CDSS)^2(\CtotNS )^2-\CDSS\CtotNS (\kd \CtotNS + \ks\CtotS ).
\end{equation*}
Thus, $A_3 = 0 \implies A_2 < 0$ since $b-\ks >0$ and Descarte's rule of sign allows to conclude that there is still only one positive steady state, which is necessarily globally stable.\\

\textbf{Case 2 : $b-k \leq 0$.} In that case, we have by assumption $\ks-b \leq \ks\frac{\tilde{b}}{\kd}$. Hence we notice that $\tilde{b}$ and $\ks-b$ are of the same order when $\tilde{b} \to 0$. 
The positive solutions of \eqref{eq:syst_2tead_const1} corresponding to a value $N_2$ are such that 
\begin{align}
    N_1 =-\frac{\ks - b + \CtotS  N_2}{2\CtotNS }\Big(1-\sqrt{1 + 4\frac{\CtotNS \tilde{b}N_2}{(\ks - b + \CtotS  N_2)^2}} \Big)\label{eq:N1}.
\end{align}

Notice that with the expression of $N_1$ in \eqref{eq:N1}, if $N_2$ stays strictly greater than a positive constant when $\tilde{b}\to 0$ we have $\ks - b + \CtotS  N_2 \sim \CtotS  N_2 $ when $\tilde{b} \to 0$. Therefore at first order in $\tilde{b}$, we have 
\begin{equation*}
  N_1 \sim  \frac{\tilde{b}}{\CtotS }.
\end{equation*}
Plugging back into \eqref{eq:syst_2tead_const2}, at first order in $\tilde{b}$, this yields
\begin{equation*}
    (N_2)^2\CDSS + N_2\kd = 0
\end{equation*}
which is a contradiction with the fact that $N_2$ stays greater than a strictly positive constant and $(N_1,N_2)$ cannot be a steady state. 

We now show that as $\tilde{b} \to 0$, there can be at most one positive solution of $P$ that does not stay strictly greater than a positive constant independent of $\tilde{b}$, which will then allow us to conclude that there is only one positive solution to \eqref{eq:syst_2tead_const1} and \eqref{eq:syst_2tead_const2} with the argument above.

Indeed, if $M_2$ and $N_2$ are two positive roots of $P$ such that $N_2 > M_2$ there is a positive solution $S(\tilde{b})$ to the equation $P'(X) = 0$ in the interval $[M_2, N_2]$. We will show that $S(\tilde{b})$ stays greater than a positive constant independent of $\tilde{b}$ when $\tilde{b} \to 0$.
By differentiating $P$, we have
\begin{align*}
     3A_3S(\tilde{b})^2 + 2A_2S(\tilde{b}) + A_1 = 0.
\end{align*}

However,
\begin{align*}
   A_1 =&  -\kd \ks\CtotNS \CtotS  - \kd^2(\CtotNS )^2  \\
   & + \tilde{b}\Big(\CtotNS \ks\CSNS- 2\ks\CtotNS (\CDSNS - \CSS) -  \CSNS ((b-\ks)(\CDSNS - \CSS) + \ks\CtotS ) \CSNS \kd \CtotNS  +\tilde{b}(\CSNS)^2 \Big)\\
   & + (\ks-b) \Big(\CtotNS \kd(\CDSNS - \CSS) -\CtotNS \CDSS \ks + 2 \kd \CSNS\CtotS  +\CSNS\CDSS (\ks-b)\Big).
\end{align*}

Thus, we can consider $\tilde{b}$ small enough (and thus $k-b$ also) such that 
\begin{align*}
    &A_1 < - \frac{1}{2}(\kd \ks\CtotNS \CtotS  + \kd^2(\CtotNS )^2) := -D < 0 \text{ and }\\
    &|A_2| < 2 |\CtotNS \CtotS \kd(\CDSNS - \CSS) - \CtotNS \CDSS \ks\CtotS  - 2 \kd (\CtotNS )^2 \CDSS  + \CSNS \kd(\CtotS )^2 |:= C
\end{align*}
 where $C$ and $D$ are positive constants that do not depend on $\tilde{b}$. We have 
\begin{align*}
  &0 < D < -A_1 = 3A_3S(\tilde{b})^2 + 2A_2S(\tilde{b})  \leq  3|A_3|S(\tilde{b})^2 + 2|A_2| S(\tilde{b})  < 3|A_3|S(\tilde{b})^2 +2C S(\tilde{b}). 
\end{align*}
Let $S >0$ be such that $3|A_3|S^2 +2C S = D$, $S$ is strictly positive and does not depend on $\tilde{b}$ and we have $S(\tilde{b}) \geq S$ and thus $N_2 \geq S$.


Hence $(N_1,N_2)$ cannot be a solution of \eqref{eq:syst_2tead_const1} and \eqref{eq:syst_2tead_const2} and the (stable) steady state corresponds to the smallest positive root of $P$, $M_2$.

\textbf{Step 3: asymptotic behavior when $b-k \leq 0$.} We now prove that the positive steady state $(N_1^*, N_2^*)$ converges to $0$ at speed $O(\tilde{b})$  when $b-\ks \leq 0$. Indeed, we first suppose that $N_1^*$ does not converge to $0$. This means that 
\begin{equation*}
    \exists C >0, \forall \varepsilon, \exists \tilde{b}_{\varepsilon} \leq \varepsilon, N_1^*(\tilde{b}_{\varepsilon}) > C.
\end{equation*}
Equation \eqref{eq:syst_2tead_const1} implies that 
\begin{align*}
     \varepsilon N_2^*(\tilde{b}_{\epsilon}) \geq \tilde{b}N_2^*(\tilde{b}_{\epsilon}) + (b-\ks)N_1^*(\tilde{b}_{\epsilon}) = \CtotNS (N_1^*(\tilde{b}_{\epsilon}))^2 + \CtotS  N_1^*(\tilde{b}_{\epsilon})N_2^*(\tilde{b}_{\epsilon})  \geq C^2\CtotNS .
\end{align*}
It follows that $N_2^*(\tilde{b}_{\epsilon}) \xrightarrow[\varepsilon \to 0]{} + \infty$.
If $\CtotS \neq 0$ we have,
\begin{equation*}
    N_1^*(\tilde{b}_{\epsilon}) = -\frac{\ks - b + \CtotS  N_2(\tilde{b}_{\epsilon})}{2\CtotNS }\Big(1-\sqrt{1 + 4\frac{\CtotNS \tilde{b}_{\varepsilon }N_2(\tilde{b}_{\epsilon})}{(\ks - b + \CtotS  N_2(\tilde{b}_{\epsilon}))^2}} \Big)\sim \frac{\tilde{b}_{\epsilon}}{\CtotS }
\end{equation*}
which is a contradiction with $N_1^*(\tilde{b}_{\epsilon}) > C$.
If $\CtotS = 0$, 
\begin{equation*}
    N_1^*(\tilde{b}_{\epsilon}) \sim \sqrt{ \frac{\tilde{b}(k-b) N_2^*(\tilde{b}_{\varepsilon})}{\CtotNS}} = o(N_2^*(\tilde{b}_{\varepsilon}))
\end{equation*}
and \eqref{eq:syst_2tead_const2} yields $(N_2^*(\tilde{b}_{\varepsilon}))^2 \CDSS = 0$, which is also a contradiction. Hence $N_1^*(\tilde{b}) \xrightarrow[\tilde{b} \to 0]{} 0$.

With \eqref{eq:syst_2tead_const2}, we obtain 
\begin{equation*}
    N_2^* = -\frac{\kd + N_1^*(\CDSNS- \CSS)}{2 \CDSS}\Big( 1 - \sqrt{1 + 4\frac{\CDSS (\ks N_1^* + \CSNS(N_1^*)^2}{(\kd + N_1^*(\CDSNS- \CSS))^2 }} \Big)\sim \frac{\ks}{\kd}N_1^*.
\end{equation*}
Thus $N_2^*(\tilde{b}) \xrightarrow[\tilde{b} \to 0]{} 0$.
Plugging that into \eqref{eq:syst_2tead_const1}, we obtain
\begin{align*}
    (N_1^*)^2\Big(\CtotNS  + \CtotS   \frac{\ks}{\kd}\Big) \sim \CtotNS  (N_1^*)^2 + \CtotS  N_1^* N_2^* = \tilde{b} N_2^* + N_1^*(b-\ks) \sim (\tilde{b}\frac{\ks}{\kd} + b- \ks) N_1^*. 
\end{align*}
And since by assumption $\tilde{b}\frac{\ks}{\kd} + b- \ks > 0$ and $\CtotNS  + \CtotS   \frac{\ks}{\kd} > 0$, this means that 
\begin{equation*}
    N_1^* \sim \frac{\tilde{b}\ks + \kd(b- \ks) }{\CtotNS \kd + \CtotS   \ks } \qquad N_2^* \sim \frac{\ks}{\kd}\frac{\tilde{b}\ks + \kd(b- \ks) }{\CtotNS \kd + \CtotS   \ks }. 
\end{equation*}
Thus, the steady state value if $O(\tilde{b})$.

\end{proof}

\begin{remark}\label{rem:comp_condi_classical}
    In a classical one phase age-structured model, with no competition and with birth and death rates denoted respectively $b$ and $d$, the non-extinction condition is $\int_0^{+\infty} b(a) e^{-\int_0^{a} d(u) \dd u} \dd a > 1$ (i.e. $b \geq d$ if the parameters are constant), see e.g. \cite{Perthame_Tumulari2008}. In comparison here, the condition holds on all four coefficients $b, \ks, \tilde{b}, \kd$ and not just $b,\tilde{b}$ and $\kd$. The condition $b/\ks + \tilde{b}/\kd >1$ can be interpreted, as follows: the condition $b/\ks >1$ ensures non extinction of the population in phase 1 and  $\tilde{b}/\kd >1$ non-extinction of the population in phase 2, which are both sufficient conditions for non-extinction of the total population. 
  In particular, the presence of $\tilde{b}\neq 0$ allows non-extinction of the population even if $b < \ks$. 
\end{remark}

We now consider the case where the non-extinction assumption \ref{ass:for_lambda_0} is not verified and the population goes to extinction. 
\begin{thm}\label{thm:ODE_0}
     Assume $\CtotNS  > 0, \CDSS >0,b/\ks + \tilde{b}/\kd < 1 $.
     Under Assumptions \ref{assump:bio} and \ref{assump:positive_compet} and for $\tilde{b}$ small enough, $(0,0)$ is the unique steady state and it is stable. Furthermore, all trajectories converge to $0$. 
\end{thm}

\begin{proof}

We first rule out the presence of a steady state other than $(0,0)$. As done in the proof of Theorem \ref{thm:ODE_pos}, a non-trivial positive steady state $(N_1, N_2)$ would verify \eqref{eq:syst_2tead_const1} and \eqref{eq:syst_2tead_const2}. By solving \eqref{eq:syst_2tead_const1}, this yields 
\begin{align*}
     N_1 =-\frac{\ks - b + \CtotS  N_2}{2\CtotNS }\Big(1-\sqrt{1 + 4\frac{\CtotNS \tilde{b}N_2}{(\ks - b + \CtotS  N_2)^2}} \Big) \leq  \frac{2\tilde{b}N_2}{\ks - b}. 
\end{align*}
Plugging back $N_1$ into \eqref{eq:syst_2tead_const2} and using the upper bound yields as before
\begin{align*}
    0 \geq &N_2^2\Big(\CDSS - 4\CSNS(\tilde{b})^2 \frac{1}{(\ks-b)^2} + \frac{4\tilde{b}}{\ks-b}\min(0, (\CDSNS - \CSS)) \Big) + N_2\Big(\kd - \frac{2\tilde{b}\ks}{\ks-b}\Big), 
\end{align*}
which is a contradiction for $\tilde{b}$ sufficiently small. Hence $(0,0)$ is the unique steady state. Since there are no periodic orbits by Theorem \ref{thm:orbit}, the Poincaré-Hopf theorem (see Theorem 6.8.2 in \cite{strogatz2024nonlinear}) implies that any trajectory converges to $(0,0)$. 
\end{proof}

\begin{remark}
\begin{itemize}
    \item  When $\tilde{b} = 0$ and $b-k>0$, there is a single stable positive steady state $(N_1^*, N_2^*)$ (see Theorem \ref{thm:unique_2S}), which verifies $N_1^* = \frac{b-k-\CtotS N_2^*}{\CtotNS}$. The relationship verified by the solutions in the case $\tilde{b} \neq 0$ and $b-k >0$ given by \eqref{eq:N1} can be seen as a perturbation of the previous equation. Similarly when $\tilde{b} \neq 0$ and $b-k \leq 0$, the vanishing property of the solutions can be interpreted as a perturbation of the system when $\tilde{b} = 0$ and $b-k \leq 0$ where $(0,0)$ is the unique (stable) steady state (see Section \ref{sec:PDE_ODE}). 
\item In the case of purely competitive systems (i.e. systems where the right-hand side of the equations of each compartment is decreasing with respect to the density of individuals in any other compartment), there cannot be sustained oscillations and any solution must either converge to a fixed point or diverge to infinity (see Theorem 3.4.1 in \cite{hofbauer1998evolutionary}). The result remains true for cooperative systems where the converse is true. As our model is neither competitive nor cooperative, such results cannot be used. 
\end{itemize}
\end{remark}

\section{Local and global stability for a PDE-ODE model}
\label{sec:PDE_ODE}

In this section, we introduce a simplified PDE-ODE version of system \eqref{syst:edp}, with assumptions different than the ones made in Section \ref{sec:ODE}. We first show the existence of steady states, and then study their local and global stability. 

\subsection{Steady state analysis}
In this section, we study the solutions of the following system which can be obtained from system \eqref{syst:edp} by assuming $\tilde{b} \equiv 0$, a constant natural death rate $\kd$ and linear competition rates as defined by equation \eqref{eq:lin_compet}. These assumptions are relevant from a biological perspective as experimental evidence shows that the Smurf individuals, which inspired phase 2 of this study, have a very reduced fertility rate ($\tilde{b}$) and a constant death rate (see \cite{tricoire_new_2015}). With these new assumptions, system \eqref{syst:edp} can be rewritten as:
\begin{equation}
    \begin{cases}
        &\frac{\partial \no(t,a) }{\partial t} + \frac{\partial \no(t,a) }{\partial a} = -(\ks(a)+ \CtotNS N_1(t) +\CtotS  N_2(t))\no(t,a) \\
        & \no(0,t) = \int_0^{+\infty}b(a) \no(t,a) \dd a\\
        &\frac{\partial \nt(t,a) }{\partial t} + \frac{\partial \nt(t,a) }{\partial a} = -(\kd+ \CDSNS N_1(t) +\CDSS N_2(t))\nt(t,a) \\
        & \nt(0,t) = \int_0^{+\infty}(\ks(a) + \CSNS N_1(t) + \CSS N_2(t))\no(t,a) \dd a\\
        & \no(a,0) = n_{1,0}(a) \in L^1(\R_+) \\
        & \nt(a,0) = n_{2,0}(a) \in L^1(\R_+)  .\label{eq:syst_linear}
    \end{cases}
\end{equation}
Where we recall that 
\begin{equation}
    N_1(t) = \int_0^{+\infty}\no(t,a) \dd a \text{ and }N_2(t) = \int_0^{+\infty}\nt(t,a) \dd a. \label{eq:def_N}
\end{equation}
In fact, the PDE on $\nt$ in the previous system can be reduced to an ODE by integration and the system can be rewritten as  

\begin{equation}
    \begin{cases}
        &\frac{\partial \no(t,a) }{\partial t} + \frac{\partial \no(t,a) }{\partial a} = -(\ks(a)+ \CtotNS N_1(t) +\CtotS  N_2(t))\no(t,a) \\
       & \no(0,t) = \int_0^{+\infty}b(a) \no(t,a)  \dd a\\
        &\frac{\dd N_2(t) }{\dd t}= -(\kd+ \CDSNS N_1(t) +\CDSS N_2(t))N_2(t) +\int_0^{+\infty}(\ks(a) + \CSNS N_1(t) + \CSS N_2(t))\no(t,a) \dd a \\
        & \no(a,0) = n_{1,0}(a) \in L^1(\R_+)\\
        & N_2(0) = N_{2,0} \geq 0  .\label{eq:syst_linear_PDEODE}
    \end{cases}
\end{equation}

We recall the definition of $\lambda_0$ as the unique solution to 
\begin{equation}
    \int_0^{+\infty} b(a)e^{-\int_0^{a} \ks(u) + \lambda_0 \dd u} \dd a = 1. \label{eq:lambda} 
\end{equation}

\begin{thm}\label{thm:unique_2S}
   Under Assumptions \ref{assump:bio}, \ref{assump:bounds}, \ref{ass:for_lambda_0}, \ref{assump:positive_compet} and if $\CtotNS  > 0$, then  \eqref{eq:syst_linear_PDEODE} has a unique non-trivial positive steady state $(\no^*, N_2^*)$. Furthermore, if $\kd >0 $ or $  \CDSS> 0$, $N_1^* = \int_0^{+\infty} \no^*(a) \dd a > 0$ and if $\ks \not \equiv 0$ or $\CSNS >0$, $N_2^* >0$. 
\end{thm}
\begin{proof}
 Let $(\no^*, N_2^*)$ denote a non-trivial steady state of \eqref{eq:syst_linear_PDEODE}. We denote $N_1^* = \int_0^{+\infty} \no^*(a) \dd a$. By solving for $\no^*$ in \eqref{eq:syst_linear_PDEODE} and plugging into the renewal condition, we obtain the condition
\begin{equation*}
     \int_0^{+\infty}b(a) e^{-\int_0^{a} \ks(u) \dd u - (\CtotNS  N_1^* + \CtotS  N_2^*) a} \dd a  = 1. 
\end{equation*}
 Hence by Assumption \ref{ass:for_lambda_0},  $\CtotNS N^*_1 + \CtotS  N^*_2 = \lambda_0$ where $\lambda_0$ is defined by \eqref{eq:lambda}. By integrating the equation on $\no^*$, this entails that
 \begin{align}
   \int_0^{+\infty} \ks(a) \no^*(a) \dd a  = \int_0^{+\infty} b(a) \no^*(a) \dd a - \lambda_0 N_1^* = \no^*(0) - \lambda_0 N_1^* = N_1^*(\frac{1}{I} - \lambda_0), \label{eq:kn1}
 \end{align}
 where $I := \frac{N_1^*}{\no^*(0)} = \int_0^{+\infty} e^{-\int_0^{a} (\ks(u) + \lambda_0) \dd u} \dd a < \frac{1}{\lambda_0}$. The strict inequality comes from Assumption \ref{assump:bio}, which ensures that $\ks$ is not identically $0$.
 As $\no^*(a) =\frac{N_1^*}{I}e^{-\int_0^{a} (\ks(u) + \lambda_0 )\dd u} $, there is a bijection between the solutions $(\no^*,N_2^*)$ and $(N_1^*, N_2^*)$. We therefore study the system verified by $(N_1^*, N_2^*)$. 
Using \eqref{eq:kn1}, we see that $(N_1^*, N_2^*)$ is a solution of the following system in $(X,Y)$:
\begin{equation}
\label{eq:steady_pop}
    \begin{cases}
        & \CtotNS  X + \CtotS  Y = \lambda_0  \\
        &-\kd Y  - \CDSNS XY - \CDSS Y^2 +\Big(\frac{1}{I} - \lambda_0\Big) X + \CSNS X^2 + \CSS XY = 0. 
    \end{cases}
\end{equation}
Notice that $(0,0)$ is always a solution and that $(0,Y)$ with $Y > 0$ is never a solution since $\kd > 0$ by Assumption \ref{assump:bio}. Similarly, $(X,0)$ with $X>0$ is never a solution. 
We now study the existence of non-trivial positive solutions of \eqref{eq:steady_pop}. 
With the first equation, we get $X = \frac{1}{\CtotNS }( \lambda_0 - \CtotS Y)$. Plugging into the second equation, this yields 

\begin{align}
 & P(Y) := Y^2( \CSNS(\frac{\CtotS }{\CtotNS })^2  +(\CDSNS-\CSS)\frac{\CtotS }{\CtotNS } - \CDSS) \nonumber\\
  & \quad+ Y \Big(-\big(\frac{1}{I} - \lambda_0\big) \frac{\CtotS }{\CtotNS }- 2\lambda_0\CSNS \frac{\CtotS }{(\CtotNS )^2} -(\CDSNS-\CSS)\frac{\lambda_0}{\CtotNS } -\kd\Big)+ \frac{\lambda_0}{\CtotNS }\Big(\big(\frac{1}{I} - \lambda_0\big)+\CSNS \frac{\lambda_0}{\CtotNS }\Big) \nonumber\\
  & =: AY^2 +BY+C = 0. \label{eq:poly_Y}
\end{align}
Where $C>0$.
Firstly, if $\CtotS  = 0$,  we have $A < 0$. Thus $A$ and $C$ have opposite signs and there is  a single positive root to $P$. When $\CtotS  = 0$, $X$ is defined by the first equation of \eqref{eq:steady_pop} independently of $Y$ and is therefore strictly positive for any value of $Y$, hence there exists a single non-trivial positive steady state.

In the rest of the proof, we consider $\CtotS  > 0$. We express $B$ in terms of $A$ and $C$ as follows,

\begin{equation}
   \CtotS  B = -\lambda_0 A -\lambda_0\CDSS - \frac{(\CtotS)^2}{\lambda_0}C -  \CtotS \kd. \label{eq:ABC}
\end{equation}
Thus, the discriminant of $P$ is
\begin{align*}
    \Delta &= B^2 - 4AC = (\frac{\lambda_0}{\CtotS } A)^2 + (\frac{\lambda_0}{\CtotS }\CDSS)^2 +(\frac{\CtotS}{\lambda_0}C)^2 +  (\CtotS)^2 \kd^2 \\
    & -2AC + 2\frac{\lambda_0}{\CtotS } A(\frac{\lambda_0}{\CtotS }\CDSS  +  \CtotS \kd) + 2\frac{\CtotS}{\lambda_0}C(\frac{\lambda_0}{\CtotS }\CDSS  +  \CtotS \kd) + 2\kd\CtotS \CDSS \\
    & = (-\frac{\lambda_0}{\CtotS } A -\frac{\lambda_0}{\CtotS }\CDSS + \frac{\CtotS}{\lambda_0}C - \CtotS \kd)^2 + 4\frac{\CtotS}{\lambda_0}C (\frac{\lambda_0}{\CtotS }\CDSS  + \CtotS \kd) > 0. 
\end{align*}
Therefore $P$ always has $2$ -potentially equal- real roots. 
There remains to verify wether there exists (a) positive solution(s) to \eqref{eq:poly_Y} and if the corresponding value of $X$ obtained with the first equation of \eqref{eq:steady_pop} is strictly positive. 

Since $X = \frac{1}{\CtotNS} (\lambda_0 - \CtotS Y)$, to determine wether the value of $X$ corresponding to a value of $Y$ is positive or negative, we have to determine wether $Y$ is greater or smaller than $\frac{\lambda_0}{\CtotS } =: Y_0$. 
$(0,Y_0)$ is solution of the first equation of \eqref{eq:steady_pop}. Plugging into \eqref{eq:poly_Y}, we have $P(Y_0)= -\kd Y_0 - \CDSS Y_0^2 <0$.  
\begin{figure}[htbp]
\centering
\begin{minipage}{0.48\textwidth}
\centering
\begin{tikzpicture}[scale=1.3]

\draw[->] (-0.5,0) -- (4,0) node[right] {$y$};
\draw[->] (0,-1) -- (0,3) node[above] {};

\draw[blue,thick,domain=0.3:3.5,smooth,variable=\x]
  plot ({\x},{1.5*(\x-1.2)*(\x-2.6)});

\node[blue,right] at (3.2,2) {\(P(y)\)};

\draw[blue,fill=blue] (1.2,0) circle (0.03);
\draw[blue,fill=blue] (2.6,0) circle (0.03);
\draw[red,fill=red] (2,0) circle (0.03);

\node[blue,below] at (1.2,0) {\(Y_1^* = N_2^*\)};
\node[blue,below] at (2.6,0) {\(Y_2^*\)};
\node[red,above] at (2,0) {\(Y_0\)};

\draw[red,thick,domain=-1:3.5,smooth,variable=\x]
  plot ({\x},{(3 - 1.5*\x)/2});

\node[red,right] at (3.4,-1) {\(\displaystyle \frac{\lambda_0 - \CtotS y}{2 \CtotNS}\)};
\end{tikzpicture}
\caption{Case $A > 0$. }
\label{fig:Case_1}
\end{minipage}
\hfill
\begin{minipage}{0.48\textwidth}
\centering
\begin{tikzpicture}[scale=1.3]

\draw[->] (-0.5,0) -- (4,0) node[right] {$y$};
\draw[->] (0,-1) -- (0,3) node[above] {};

\draw[blue,thick,domain=-0.5:1.9,smooth,variable=\x]
  plot ({\x},{-0.8*(\x)*(2.6*\x-2.7)+2});

\node[blue,right] at (1,2) {\(P(y)\)};

\draw[blue,fill=blue] (1.62,0) circle (0.03);
\draw[red,fill=red] (2,0) circle (0.03);
\node[red,above] at (2,0) {\(Y_0\)};

\node[blue,below] at (1.46,0) {\(Y^* = N_2^*\)};

\draw[red,thick,domain=-1:3.5,smooth,variable=\x]
  plot ({\x},{(3 - 1.5*\x)/2});

\node[red,right] at  (3.4,-1) {\(\displaystyle \frac{\lambda_0 - \CtotS y}{2 \CtotNS}\)};

\end{tikzpicture}
\caption{Case $A \leq  0$}
\label{fig:Case_2}

\end{minipage}
\end{figure}

\textbf{Case 1: \eqref{eq:poly_Y} has two positive solutions $Y^*_1 < Y^*_2$}.
A schematic representation of this case if shown on Figure \ref{fig:Case_1}.
This case arises when $A > 0$. With equation \eqref{eq:ABC}, we notice that $A>0 \implies B <0$. Since $C > 0$, we have $(P(Y_0) < 0 \text{ and } Y_0 > 0) \implies Y_0 \in [Y^*_1,Y^*_2]$ thus the value $X^*_2$ corresponding to $Y^*_2$ is negative, and the only positive solution to \eqref{eq:steady_pop} is $(X_1^*, Y_1^*)$.

\textbf{Case 2: \eqref{eq:poly_Y} has a unique positive solution $Y^*$}
A schematic representation of this case if shown on Figure \ref{fig:Case_2}.
This case arises when $A \leq 0 $. Then, we have $(P(Y_0) < 0 \text{ and } Y_0 > 0) \implies Y_0 \geq Y^*$ thus the value $X^*$ corresponding to $Y^*$ is positive.

In any case, there is one strictly positive solution to \eqref{eq:steady_pop} and a unique non-trivial positive steady state. 
 
\end{proof}

\begin{prop} \label{prop:exp_ss}
    Assume the conditions of Theorem \ref{thm:unique_2S} are verified. The unique positive non-trivial steady state $(n_1^*(a),N_2^*)$ of \eqref{eq:syst_linear} is such that $n_1^*$ verifies
    \begin{align*}
        &n_1^*(a) = n_1^*(0) e^{-\int_0^{a} (\ks (u)+ \lambda_0) \dd u},
    \end{align*}
    where $\lambda_0$ is defined in \eqref{eq:lambda}. 
\end{prop}

\begin{proof}
    The proof follows directly by noticing that the population sizes at equilibrium $N_1^* = \int_0^{+\infty} n_1^*(a) \dd a $ and $N_2^*$ verify $ \CtotNS N_1^* + \CtotS N_2^* = \lambda_0$. 
\end{proof}
We now state and prove a qualitative result on the impact of a change in $\lambda_0$ on the proportion of individuals in phase~2. $\lambda_0$ quantifies the natural growth rate of the population and is increasing in $b$ and decreasing in $\ks$. 
\begin{prop}
    Assume the conditions of Theorem \ref{thm:unique_2S} are verified and $\CSNS(\frac{\CtotS }{\CtotNS })^2  +(\CDSNS-\CSS)\frac{\CtotS }{\CtotNS } - \CDSS \neq 0$. Let $\lambda_0$ be defined by \eqref{eq:lambda}, $A,B,C$ by \eqref{eq:poly_Y} and $(N_1^*, N_2^*)$ be the population sizes of the unique positive non-trivial steady state. We have 
    \begin{align*}
          \frac{\partial}{\partial \lambda_0}\! \big( \! \frac{N_2^*}{N_1^*\! +\! N_2^*}\big) \!=\! \frac{\lambda_0}{\CtotNS(N_1^*\! +\! N_2^*)^2}  \Bigg(\frac{1}{2A} \Big(\frac{\kd}{\lambda_0}\! +\! R  \frac{\CtotS }{\CtotNS }\Big) \Big(\!-\!1 \!-\!  \frac{B}{ \sqrt{B^2\!-\!4AC}} \Big)  \! -\!  \frac{\lambda_0 R }{ \CtotNS \sqrt{B^2\!-\!4AC}} \Bigg).
    \end{align*}
    where 
    \begin{equation*}
      I =  \int_0^{+\infty} e^{-\int_0^{a} \ks(y) + \lambda_0 \dd u} \dd a \text{ and } R =  \frac{\int_0^{+\infty} a \ks(a) e^{-\int_0^{a} \ks(y) + \lambda_0 \dd u} \dd a}{\lambda_0 I^2}. 
    \end{equation*}
In particular, $\frac{1}{\lambda_0 I} \geq R$ and 
\begin{align*}
     &\frac{\partial}{\partial \lambda_0} \big(\frac{N_2^*}{N_1^* + N_2^*}\big)  \sim_{\CSNS \to + \infty} \frac{\lambda_0^2}{\kd (\CtotNS)^2 (N_1^* + N_2^*)^2}   (\frac{1}{\lambda_0 I} - R) > 0.
\end{align*}
\end{prop}

\begin{proof}
    We have 
    \begin{align}
        \frac{\partial}{\partial \lambda_0} \big(\frac{N_2^*}{N_1^* + N_2^*}\big) &= \frac{\frac{\partial N_2^*}{\partial \lambda_0}(N_1^* + N_2^*) - N_2^*(\frac{\partial N_1^*}{\partial \lambda_0} + \frac{\partial N_2^*}{\partial \lambda_0})  }{(N_1^* + N_2^*)^2} \nonumber\\
        &= \frac{\frac{\partial N_2^*}{\partial \lambda_0}(\frac{\lambda_0}{\CtotNS} - \frac{\CtotS N_2^*}{\CtotNS} + N_2^*) - N_2^*(\frac{1}{\CtotNS} - \frac{\CtotS}{\CtotNS} \frac{\partial N_2^*}{\partial \lambda_0} + \frac{\partial N_2^*}{\partial \lambda_0})  }{(N_1^* + N_2^*)^2} \nonumber\\
        &=  \frac{\lambda_0 \frac{\partial N_2^*}{\partial \lambda_0} -N_2^* }{\CtotNS (N_1^* + N_2^*)^2}. \label{eq:diff_ratio}
    \end{align}
We now study the sign of $\lambda_0 \frac{\partial N_2^*}{\partial \lambda_0} -N_2^*$. denote $P(Y) = AY^2 + BY + C$ the polynomial such that  $P(N_2^*) = 0$ as introduced in the proof of Theorem \ref{thm:unique_2S} (see equation \eqref{eq:poly_Y}). It can be seen from Figures \ref{fig:Case_1} and \ref{fig:Case_2} that for any value of $A$, we have 
\begin{align*}
    N_2^* = \frac{1}{2A} (-B - \sqrt{B^2-4AC}). 
\end{align*}
Furthermore, for $I = \int_0^{+\infty} e^{-\int_0^{a} \ks(y) + \lambda_0 \dd u} \dd a$, we have by integration by parts
\begin{align*}
    \frac{\partial }{\partial \lambda _0} \big(\frac{1}{I}\big) &= \frac{-1}{I^2} \frac{\partial I}{\partial \lambda_0} =  \frac{1}{I^2} \int_0^{+\infty} a e^{-\int_0^{a} \ks(y) + \lambda_0 \dd u} \dd a\\
    & =  \frac{1}{I^2} \Bigg(\big[-\frac{1}{\lambda_0}  a e^{-\int_0^{a} \ks(y) + \lambda_0 \dd u} \big]_0^{+\infty} + \frac{1}{\lambda_0} \int_0^{+\infty}  e^{-\int_0^{a} \ks(y) + \lambda_0 \dd u} \dd a \\
    & \qquad \qquad \qquad \qquad - \frac{1}{\lambda_0}\int_0^{+\infty} a \ks(a) e^{-\int_0^{a} \ks(y) + \lambda_0 \dd u} \dd a\Bigg)\\
    &= \frac{1}{I \lambda_0 } -\frac{\int_0^{+\infty} a \ks(a) e^{-\int_0^{a} \ks(y) + \lambda_0 \dd u} \dd a}{\lambda_0 I^2}\\
    & =: \frac{1}{I \lambda_0 }  - R. 
 \end{align*}
Using the previous identity, we have 
\begin{align*}
    \frac{\partial B}{\partial \lambda_0} &= -\big(\frac{1}{I \lambda_0 }  - R - 1\big) \frac{\CtotS }{\CtotNS }- 2\CSNS \frac{\CtotS }{(\CtotNS )^2} -(\CDSNS-\CSS)\frac{1}{\CtotNS } = \frac{(B+\kd)}{\lambda_0} + R  \frac{\CtotS }{\CtotNS }. 
\end{align*}
Similarly, 
\begin{align*}
    \frac{\partial C}{\partial \lambda_0} & = \frac{1}{\CtotNS }\Big(\big(\frac{1}{I} - \lambda_0\big)+\CSNS \frac{\lambda_0}{\CtotNS }\Big) + \frac{\lambda_0}{\CtotNS }\Big(\big(\frac{1}{\lambda_0 I} - R - 1\big)+\frac{\CSNS}{\CtotNS }\Big) = \frac{2 C}{\lambda_0} - \frac{\lambda_0 R}{\CtotNS}. 
\end{align*}
Hence, since $A$ does not depend on $\lambda_0$, we have 
\begin{align*}
    &\frac{\partial N_2^*}{\partial \lambda_0} = \frac{1}{2A}\Bigg(- \frac{\partial B}{\partial \lambda_0} - \frac{1}{2 \sqrt{B^2-4AC}} \Big( 2 B \frac{\partial B}{\partial \lambda _0} - 4A \frac{\partial C}{\partial \lambda_0}\Big) \Bigg)\\
    & =  \frac{1}{2A} \Bigg(-\frac{B}{\lambda_0} - \frac{\kd}{\lambda_0} - R  \frac{\CtotS }{\CtotNS } -  \frac{1}{ \sqrt{B^2-4AC}} \Big(  \frac{B^2}{\lambda_0} + B\frac{\kd}{\lambda_0} + B R  \frac{\CtotS }{\CtotNS } - \frac{4AC}{\lambda_0} + 2A \frac{\lambda_0 R}{\CtotNS}\Big) \Bigg) \\
    & = \frac{1}{2A \lambda_0 }\Bigg(-B - \frac{B^2-4AC}{\sqrt{B^2-4AC}} \Bigg) + \frac{1}{2A}\Bigg( -\frac{\kd}{\lambda_0} - R  \frac{\CtotS }{\CtotNS } \\
    & \qquad \qquad \qquad \qquad \qquad -  \frac{1}{ \sqrt{B^2-4AC}}\Big( B\frac{\kd}{\lambda_0} + BR \frac{\CtotS }{\CtotNS } + 2A \frac{\lambda_0 R}{\CtotNS}\Big) \Bigg)\\
    & = \frac{N_2^*}{\lambda_0} + \frac{1}{2A}\Bigg( \Big(\frac{\kd}{\lambda_0} + R  \frac{\CtotS }{\CtotNS }\Big) \Big(-1 -  \frac{B}{ \sqrt{B^2-4AC}} \Big)   -  \frac{1}{\sqrt{B^2-4AC}}\Big(2A \frac{\lambda_0 R}{\CtotNS}\Big) \Bigg). 
\end{align*}
Thus, by equation \eqref{eq:diff_ratio}, to determine the sign of $ \frac{\partial}{\partial \lambda_0} \big(\frac{N_2^*}{N_1^* + N_2^*}\big) $, we need to determine the sign of 
\begin{equation*}
     \frac{1}{2A} \Big(\frac{\kd}{\lambda_0} + R  \frac{\CtotS }{\CtotNS }\Big) \Big(-1 -  \frac{B}{ \sqrt{B^2-4AC}} \Big)   -  \frac{\lambda_0 R }{ \CtotNS \sqrt{B^2-4AC}}. 
\end{equation*}
Since $A>0 \implies B \leq 0$, the first term is always positive, and the second term is always negative. Although we cannot determine the sign of this expression without further assumptions, we can find its sign when $\CSNS \to +\infty$ 
\begin{align*}
     &\frac{\partial N_2^*}{\partial \lambda_0} -\frac{N_2^*}{\lambda_0} \sim_{\CSNS \to + \infty}\frac{\lambda_0}{\kd \CtotNS} (\frac{1}{\lambda_0 I} - R) > 0.
\end{align*}
\end{proof}

\begin{remark}
    The previous result allows to see that there is no generality in the way an increase of the natural growth rate of the population, quantified by $\lambda_0$, affects the proportion of individuals in phase~2. For different parameter values, the impact can be either positive or negative. However, when the competition accelerating transition is high, an increase in $\lambda_0$ results in an increase of the proportion of individuals in phase~2. 
\end{remark}

We now study the impact of competition on the transition from phases 1 to 2 by comparing the population sizes at equilibrium of the system with and without competition affecting the transition rate.

  We introduce the following system, which is obtained from \eqref{eq:syst_linear_PDEODE} by setting $\eta_1 = \eta_2 = 0$. 
    \begin{equation}
    \begin{cases}
        &\frac{\partial \no(t,a) }{\partial t} + \frac{\partial \no(t,a) }{\partial a} = -(\ks(a)+ \CtotNS N_1(t) +\CtotS N_2(t))\no(t,a) \\
        & \no(0,t) = \int_0^{+\infty}b(a) \no(t,a) \dd a\\
        &\frac{\partial \nt(t,a) }{\partial t} + \frac{\partial \nt(t,a) }{\partial a} = -(\kd+ \CDSNS N_1(t) +\CDSS N_2 (t))\nt(t,a) \\
        & \nt(0,t) = \int_0^{+\infty}\ks(a)\no(t,a) \dd a\\
        & \no(a,0) = n_{1,0}(a) \in L^1(\R_+,\R_+) \\
        & \nt(a,0) = n_{2,0}(a) \in L^1(\R_+,\R_+).\label{eq:syst_linear_nosmurf}
    \end{cases}
\end{equation}
By Theorem \ref{thm:unique_2S}, \eqref{eq:syst_linear_nosmurf} has a unique non-trivial positive steady state. 

\begin{prop}\label{prop:compet_compare}
Suppose Assumptions \ref{assump:bio}, \ref{assump:bounds}\eqref{bound:b}, \ref{ass:for_lambda_0} and \ref{assump:positive_compet} are verified and $\CSNS > 0$.
Let $(N_1^*, N_2^*)$ and $(N_1^{**}, N_2^{**})$ be respectively the population sizes of the unique non-trivial positive steady states of \eqref{eq:syst_linear_PDEODE} and \eqref{eq:syst_linear_nosmurf}. 
   We have
\begin{equation*}
    N_1^{**} \geq N_1^* \text{ and } N_2^{**} \leq N_2^*.  
 \end{equation*}
 Furthermore,
 \begin{equation*}
   \CtotS \leq \CtotNS \iff  N_1^{**} + N_2^{**} \leq N_1^* + N_2^*.  
 \end{equation*}
\end{prop}
\begin{proof}
We begin by recalling that the uniqueness of a non-trivial positive steady state for both \eqref{eq:syst_linear_PDEODE} and \eqref{eq:syst_linear_nosmurf} is given by Theorem \ref{thm:unique_2S}. 
   We denote by $P(Y)$ the polynomial such that $P(N_2^*) = 0$ as in the proof of Theorem \ref{thm:unique_2S}, equation \eqref{eq:poly_Y}. We denote by $P_2$ the polynomial such that $P_2(N_2^{**}) = 0$. The expression of $P_2$ can be deduced from the expression of $P$ by setting $\CSNS = \CSS = 0$ and $\CDNS = \CtotNS, \CDS = \CtotS$. 
   We show that $N_2^{**} \leq N_2^*$ by determining the sign of $P(N_2^{**})$. 
    We have 
    \begin{align*}
        P(Y) - P_2(Y) &= \frac{\CtotS}{\CtotNS}(\CSNS  \frac{\CtotS}{\CtotNS} - \CSS) Y^2 + \frac{\lambda_0}{\CtotNS}(\CSS - 2\CSNS\frac{\CtotS}{\CtotNS}) Y + \CSNS \Big(\frac{\lambda_0}{\CtotNS}\Big)^2 \\
        & = \frac{\CtotS}{(\CtotNS)^2}\Big(Y - \frac{\lambda_0}{\CtotS}\Big) \Big( (\CSNS \CtotS - \CSS \CtotNS) Y- \lambda_0 \CSNS \Big). 
        \end{align*}
    Since $\lambda_0 = \CtotS N_2^{**} + \CtotNS N_1^{**}$, we have $N_2^{**}-  \frac{\lambda_0}{\CtotS} \leq 0$. Furthermore, 
    \begin{align*}
        & (\CSNS \CtotS - \CSS \CtotNS) N_2^{**} -  \lambda_0 \CSNS  = \CSNS( \CtotS N_2^{**} - \lambda_0 ) -  \CSS \CtotNS N_2^{**} \leq 0.
    \end{align*}
    Thus, we have $P(N_2^{**}) \geq 0$. 
    
   Once again referring to the notations introduced in the proof of Theorem \ref{thm:unique_2S} where $P(Y) = AY^2+BY+C$ is the polynomial such that $P(N_2^*) = 0$, we can distinguish two cases. 
      
   \textbf{Case 1 :  $A > 0$}, $P$ has two positive roots $N_2^*$ and $Y_2$ such that $N_2^* < Y_2$ and $\lambda_0 - \CtotS Y_2 < 0$. Since by definition $\CtotNS N_1^{**} =  \lambda_0 - \CtotS N_2^** \geq  0$, we have necessarily $N_2^{**} \leq Y_2$ and since $P(N_2^{**}) \geq 0$, $N_2^{**} \leq N_2^*$ ( in Figure \ref{fig:Case_1}, $N_2^{**}$ is such that $P(N_2^{**}) \geq 0$ and is under the red line).  
   
   \textbf{Case 2 :  $A \leq 0$} , $0 \leq N_2^{**} \leq N_2^*$ since $N_2^*$ is the only positive root of $P(Y) = 0$ (see Figure \ref{fig:Case_2} for a schematic representation) .

    In any case, we have $N_2^{**} \leq N_2^*$. Since  $\lambda_0 = \CtotS N_2^{**} + \CtotNS N_1^{**} =\CtotS N_2^{*} + \CtotNS N_1^{*} $, this entails that $N_1^{**} \geq N_1^*$. 
    
    Finally, we have 
    \begin{align*}
       & N_1^*  + N_2^* = \frac{\lambda_0}{\CtotNS} + (1-\frac{\CtotS}{\CtotNS})N_2^* \\
        &N_1^{**}  + N_2^{**} = \frac{\lambda_0}{\CtotNS} + (1-\frac{\CtotS}{\CtotNS})N_2^{**}.
    \end{align*}
Which entails $\CtotS \leq \CtotNS \iff  N_1^{**} + N_2^{**} \leq N_1^* + N_2^*$. 
\end{proof}

\begin{remark}
  As mentioned in Remark \ref{rem:compet_less}, the assumption $\frac{\CtotS}{\CtotNS} \leq 1$ is biologically realistic. Under this condition, not having competition accelerating the transition between phases 1 and 2 entails a smaller population size at equilibrium. As Proposition \ref{prop:compet_compare} assumes that both compared systems have the same total competition affecting phase 1, this means that in the presence of competition accelerating the transition from phases 1 to 2 (model \eqref{eq:syst_linear_PDEODE}), there is less accidental death in phase~1, which is why it results in a larger total population size at equilibrium. 
\end{remark}

\subsection{Local stability}
In this section, we study the stability of the system around its positive equilibria. We present the main result, for which the proof is detailed in the Appendix \ref{sec:app_thm}.

\begin{thm}\label{thm:stab}
    Suppose Assumptions \ref{assump:bio}, \ref{assump:bounds} and \ref{assump:positive_compet} are verified, $\CtotNS >0$ and $\kd~\cdot~ \CDSS \neq 0$. Then, if  $1 < \int_0^{+\infty} b(a) e^{-\int_0^{+\infty }\ks(u) \dd u } \dd a$, the unique non-trivial strictly positive stationary state $(\no^*, N_2^*)$ of system \eqref{eq:syst_linear} is locally asymptotically stable. If  $1 > \int_0^{+\infty} b(a) e^{-\int_0^{+\infty }\ks(u) \dd u } \dd a$, the stationary state $(0,0)$ is locally asymptotically stable. 
\end{thm}

In order to prove these results, we proceed in the following way :
\begin{itemize}
    \item Express system \eqref{eq:syst_linear_PDEODE} as a theoretical Cauchy problem. 
    \item Prove that the linearised operator around a steady state generates a $C_0$-continuous semi-group.
    \item Study the spectrum of the linearised operator to conclude on stability and prove Theorem \ref{thm:stab}. 
\end{itemize}

To reformulate our system as a Cauchy problem, we first introduce a series of notations (see \cite{Implication_Perasso_19} for more details).
Let $L$ be a differential operator and $\{T_L(t)\}_{t \geq 0}$ its semi-group. We denote by $\rho(L)$ the resolvent set of $L$ and define the spectrum of $L$, $\sigma(L)$, its point spectrum $\sigma_p(L)$ and the spectral bound $s(L)$ by 
\begin{align*}
    &\sigma(L) = \mathbb{C}\backslash \rho(L), \\
    &\sigma_p(L) = \{\lambda \in \mathbb{C}, \lambda I - L \text{ is not injective} \}, \\
    &s(L) = \sup \{Re(\lambda), \lambda \in \sigma(L) \}.
\end{align*}
We let $ X =  L^1((0, +\infty), \R) \times \R^2 $ endowed with the product norm denoted $||.||$. Define the linear operator $\mathcal{A}$ on $D(\mathcal{A}) =  W^{1,1}((0, +\infty), \R) \times \R\times \{0\} \subset X$ by
\begin{equation*}
\mathcal{A} \begin{pmatrix}
x \\ y \\ 0 
\end{pmatrix}
= \begin{pmatrix}
-x'(a) -\ks(a) x(a)\\
-\kd y  \\
 -x(0)
\end{pmatrix}
\end{equation*}
where $x'$ denotes the differential of $x$.
We also define the operator $\mathcal{F}: L^1((0, +\infty), \R) \times \R\times \{0\} \to X$ corresponding to the nonlinear part of system \eqref{eq:syst_linear_PDEODE} (including the birth term for ease of computation later on), 
\begin{align*}
    &\mathcal{F} \begin{pmatrix}
x \\ y \\ 0 
\end{pmatrix}
=\\
&\begin{pmatrix}
-(\CtotNS  \int_0^{+\infty} x(a) \dd a + \CtotS  y) x(a) \\
-(\CDSNS \int_0^{+\infty} x(a) \dd a + \CDSS y)y + \int_0^{+\infty} \ks(a) x(a) \dd a + (\CSNS \int_0^{+\infty} x(a) \dd a + \CSS y) \int_0^{+\infty} x(a) \dd a\\
 \int_0^{+\infty} b(a) x(a) \dd a \nonumber
\end{pmatrix}.
\end{align*}
Defined as such, if we let $u(t) = (\no(t,.), N_2(t), 0)^{T}$ and $u_0 = (n_{1,0}(.), N_{2,0}, 0)^{T}$, then \eqref{eq:syst_linear_PDEODE} can be expressed as the following Cauchy problem,
\begin{equation}
    \begin{cases}
        &\frac{\dd}{\dd t} u(t) = \mathcal{A}u(t) + \mathcal{F}(u(t)), \, t \geq 0\\
        &u(0) = u_0.
    \end{cases}
    \label{eq:cauchy_abstract}
\end{equation}

We now compute the linearisation of the operator $\mathcal{A} + \mathcal{F}$ around a (generic) steady state.
Let $u^* = (\no^*(.), N_2^*, 0)$ be a steady state of \eqref{eq:cauchy_abstract}. The linearised system around $u^*$ is  $\mathcal{A} + D\mathcal{F}(u^*) $ with 
\begin{equation}
\label{eq:DFU}
\begin{split}
       &D\mathcal{F}(u^*)\begin{pmatrix}
        x(t,a)\\ y(t) \\ 0
    \end{pmatrix}
    =  \begin{pmatrix}
        -(\CtotNS N_1^* + \CtotS N_2^* ) x(t,a) \\
       0 \\
        0
    \end{pmatrix} +\\
   & \begin{pmatrix}
        - (\CtotNS X(t) + \CtotS y(t) )\no^*(a) \\
       (\CSNS N_1^*\! -\! \CDSNS  N_2^*) X(t) \!+\!  \int_0^{+\infty} (\ks(a) \!+\! \CSNS N_1^* + \CSS N_2^* ) x(t,a) \dd a  \! -\!((\CDSNS - \CSS)  N_1^* + 2\CDSS N_2^* ) y(t) \\
        \int_0^{+\infty} b(a) x(t,a) \dd a
    \end{pmatrix}\\
    & := D\mathcal{F}(u^*)_1 + D\mathcal{F}(u^*)_2, 
\end{split}
\end{equation}
where $X(t) = \int_0^{+\infty}x(t,a) \dd a$ and $N_1^* = \int_0^{+\infty}\no^*(a) \dd a$. 

We now study the operator $\mathcal{A} + D\mathcal{F}(u^*) $ and prove that it generates a $C_0$-semi-group. 
We recall the following definition (see e.g. \cite{Engel_Nagel_01} p.87), where $||.||_{\mathcal{L}(X)}$ denotes the operator norm on $\mathcal{L}(X)$ the set of bounded linear operators on $X$.
\begin{definition}[Hille-Yosida operator]
    Let $L : D(L) \subset X \to X$ be a linear operator. If there exist real constants $M \geq 1$ and $\omega \in \R$ such that $(\omega, + \infty) \subset \rho(L)$ and 
    \begin{equation*}
        ||(\lambda- L)^{-n} ||_{\mathcal{L}(X)} \leq \frac{M}{(\lambda - \omega)^n} \text{ for all } n \in \mathbb{N}_+ \text{ , and all } \lambda > \omega. 
    \end{equation*}
    Then the linear operator $(L, D(L))$ is called a Hille-Yosida operator. 
\end{definition}
This definition allows us to enunciate the Hille-Yosida theorem (Corollary 3.21 p.87 in \cite{Engel_Nagel_01}).

\begin{thm}[Hille-Yosida Theorem]
    Let $(L,D(L))$ be a linear operator on a Banach space $X$. If $(L, D(L))$ is a Hille-Yosida operator, then the part of $L$ on $\overline{D(L)}$ generates a $C_0$-continuous semi-group. 
    \label{thm:hille_yosida}
\end{thm}
We now apply Theorem \ref{thm:hille_yosida} to the operator $\mathcal{A}$. 
\begin{prop}\label{prop:Hille_Yosida}
Under Assumptions \ref{assump:bio}, \ref{assump:bounds} and \ref{assump:positive_compet},  the operator $(\mathcal{A}, D(\mathcal{A}))$ is a Hille-Yosida operator. 
\end{prop}

\begin{proof}
We introduce $\delta = \min(\inf(\ks), \kd) > - \infty$ (potentially $0$). 
Let $(x,y,z) ~\in ~X$, $(g_1,g_2,0)~ \in~ D(\mathcal{A})$ and $\lambda \in \mathbb{R}$ such that $\lambda > -\delta$. We have 
 \begin{align*}
     (\lambda - \mathcal{A})^{-1} \begin{pmatrix}
         x \\ y \\z \end{pmatrix}  = 
         \begin{pmatrix}
         g_1 \\ g_2 \\0 \end{pmatrix} 
         \iff 
         \begin{cases}
           g_1'(a) + \lambda g_1(a) +\ks(a) g_1(a) = x(a)\\
             (\lambda + \kd) g_2 = y \\
            g_1(0) = z.
         \end{cases}
\\ \implies 
    \begin{cases}
        g_1(a) = e^{-\int_0^{a} (\ks(u) + \lambda )\dd u} (z + \int_0^{a}x(u) e^{\int_0^{u} (\ks(x) + \lambda )\dd x} \dd u)\\
        g_2 = \frac{y}{\lambda + \kd}. 
    \end{cases}
\end{align*}
We have the following inequalities,
\begin{align*}
    \int_0^{+\infty} \int_0^{a} x(u) e^{-\int_u^{a} \ks(x) + \lambda \dd x} \dd u \dd a & \leq  \int_0^{+\infty} \int_0^{a} x(u) e^{-(a-u)(\delta + \lambda)} \dd u \dd a\\
    & \leq \int_0^{+\infty} x(u) \int_u^{+\infty}  e^{-(a-u)(\delta + \lambda)} \dd a \dd u\\
    & \leq \frac{1}{\lambda + \delta} ||x||_1. 
\end{align*}

By integrating $g_1$ and using the bound above, we obtain,
\begin{equation*}
    ||g_1||_1 + |g_2| \leq \frac{1}{\lambda + \delta}( ||x||_1 + |y| + |z|). 
\end{equation*}
Thus $||(\lambda - \mathcal{A})^{-1}||_{\mathcal{L}(X)} \leq \frac{1}{\lambda + \delta}$ and the inequality with $||(\lambda - \mathcal{A})^{-n}||_{\mathcal{L}(X)}$ follows directly. 
Finally, $(\mathcal{A}, D(\mathcal{A}))$ is a Hille-Yosida operator. 
\end{proof}

\begin{prop}
   Under Assumptions \ref{assump:bio}, \ref{assump:bounds} and \ref{assump:positive_compet}, $\mathcal{A} + D\mathcal{F}(u^*) $ generates a $C_0$-continuous semi-group on $\overline{D(\mathcal{A})}$.
\end{prop}
\begin{proof}

By Proposition \ref{prop:Hille_Yosida},  $\mathcal{A}$ is a Hille-Yosida operator thus it generates a $C_0-$continuous semi-group on $\overline{D(\mathcal{A})}$ by Theorem \ref{thm:hille_yosida}. Furthermore, since $b$ and $\ks$ are positive and bounded by Assumptions \ref{assump:bio} and \ref{assump:bounds}, $D\mathcal{F}(u^*)$ is a bounded linear operator. It follows with Proposition 4.14 in \cite{webb1985theory} that $\mathcal{A} + D\mathcal{F}(u^*) $ generates a $C_0$-continuous semi-group on $\overline{D(\mathcal{A})}$.
\end{proof}

\begin{remark}
    Formulated as such, the fact that $\mathcal{A}$ generates a $C_0$-continuous semi-group cannot be determined by the Lumer-Phillips theorem as $\mathcal{D}(\mathcal{A})$ is not dense in $X$ (see \cite{Hopf_YAN_19}). Another formulation of the problem where the initial condition is directly embedded in the definition space of $\mathcal{A}$ would allow to use the Lumer-Phillips theorem as is done in \cite{Implication_Perasso_19}. 
\end{remark}

Now that we have proved that $\mathcal{A} + D\mathcal{F}(u^*) $ generates a $C_0$-continuous semi-group, we study the asymptotic stability of that semi-group using spectral arguments and following similar steps as in \cite{Implication_Perasso_19}. We begin by introducing some notations (see e.g. \cite{Engel_Nagel_01}, IV.2.). 

Let $\mathcal{L}(X)$ denote the set of linear operators on $X$ and $\mathcal{K}(X)$ the subset of compact operators, we define the essential norm $||T||_{ess}$ for $T\in \mathcal{L}(X)$ by 
\begin{equation*}
    ||T||_{ess} = \inf_{K \in \mathcal{K}(X)}||T-K||_{\mathcal{L}(X)}. 
\end{equation*}
We define the essential growth bound $\omega_{ess}(L) \in [-\infty, +\infty)$, 
\begin{equation}
    \omega_{ess}(L) = \lim_{t\to \infty} \frac{1}{t}\log(||T_{L}(t)||_{ess}). 
\end{equation}
And finally, we set 
\begin{equation}
    \omega_{0}(L) = \max\{\omega_{ess}(L), s(L) \}.\label{eq:wo}
\end{equation}
With these notations given, we can now state the following stability result,
\begin{thm}[Theorem E in \cite{pruss1981equilibrium}] \label{thm:eigen_2table}
    Let $L$ be a generator of a $C_0$-semi-group in a Banach space $X$ and $F: X \to X$ be Lipschitz differentiable at $\phi$ where $\phi$ is such that $L \phi + F(\phi) = 0$. Put $L_{\phi} = L + F'(\phi)$. Then, we have
    \begin{itemize}
        \item $\omega_0(L_{\phi}) < 0 $ implies local asymptotic stability of $\phi$
        \item $\omega_0(L_{\phi}) > 0 $ and $\omega_{ess}(L_{\phi}) \leq 0$ imply instability of $\phi$. 
    \end{itemize}
\end{thm}
Furthermore, by Corollary 2.11 in \cite{Engel_Nagel_01}, we have 
\begin{equation}
    \forall \omega > \omega_{\text{ess}}, \sigma(L) \cap \{\lambda \in \mathbb{C}, Re(\lambda) \geq \omega \} \subset \sigma_p(L). \label{eq:wo_w}
\end{equation}
In particular, this means that if $\omega_{\text{ess}}(L) < 0$, it suffices to study $\sigma_p(L)$ to determine if the spectral radius of $L$ is strictly positive.

We now apply these results to the operator at hand by studying its essential growth bound and spectral radius. To that effect, we first prove that $D\mathcal{F}(u^*)_2$ as defined by equation \eqref{eq:DFU} is a compact operator, for which the contribution to the essential growth bound is null. 

\begin{prop}\label{prop:d2_compact}
    Under Assumptions \ref{assump:bio}, \ref{assump:bounds}\eqref{bound:b} and \ref{assump:positive_compet}, $D\mathcal{F}(u^*)_2$ is a compact operator. 
\end{prop}
\begin{proof}
    We write 
    \begin{equation*}
        D\mathcal{F}(u^*)_2 = \begin{pmatrix}
            G_1\\
            G_2\\
            G_3
        \end{pmatrix}
    \end{equation*}
    where $(G_1,G_2,G_3)$ are the components of $D\mathcal{F}(u^*)$ given by \eqref{eq:DFU}. 
    As the range of $G_2$ and $G_3$ is finite-dimensional, $G_2$ and $G_3$ are compact. For $G_1$, we apply the  Riesz-Fréchet-Kolmogorov criterion (see e.g. \cite{Brezis_analysis} p.72) in $L^1(\R_+)$ as done in the proof of Lemma 3.4 in \cite{Implication_Perasso_19}. 
    Let $S$ be a bounded subset of $X$ and $h \in \R_+$. We denote $\tau_h(f) = f(.+h)$ the translation operator in $L^1(\R_+)$.  For $(x,y,0)\in S$, we have 
    \begin{align*}
        ||\tau_h(G_1(x,y)) &- G_1(x,y)||_1 \leq |\CtotNS \int_0^{+\infty}x(a) \dd a + \CtotS  y | \cdot \int_0^{+\infty}| n^*_1(a+h) - n^*_1(a)| \dd a \\
        & \leq \max(\CtotNS ,\CtotS )\sup_{(x,y,0) \in S}||(x,y,0)|| \int_0^{+\infty}| n^*_1(a+h) - n^*_1(a)| \dd a \xrightarrow[h \to 0]{} 0.
    \end{align*}
Thus 
\begin{equation*}
   \sup_{(x,y,0) \in S} ||\tau_h(G_1(x,y)) - G_1(x,y)||_1 \xrightarrow[h \to 0]{} 0.
\end{equation*}
We also have 
\begin{align*}
    \int_r^{+\infty} |G_1(x,y,0)(a)| \dd a \leq \Big|\CtotNS \int_0^{+\infty}x(a) \dd a + \CtotS  y \Big| \cdot \int_r^{+\infty} n^*_1(a) \dd a \xrightarrow[r \to +\infty]{} 0,
\end{align*}
 as $\no^* \in L_1(\R_+)$. Since $G_1(S)$ is bounded, the Riesz-Fréchet-Kolmogorov criterion allows to conclude that $G_1(S)$ is relatively compact in $L_1(\R_+)$ and $G_1$ is a compact operator.
Finally, $D\mathcal{F}(u^*)_2$ is a compact operator. 

\end{proof}

We now present the following result on the essential growth bound of $\mathcal{A} + D\mathcal{F}(u^*)$. 
\begin{prop} \label{thm:omega_ess}
    Under Assumptions \ref{assump:bio}, \ref{assump:bounds} and \ref{assump:positive_compet}, $\omega_{ess}(\mathcal{A} + D\mathcal{F}(u^*) )\leq  0$. Furthermore, if $(\CtotNS N_1^* + \CtotS N_2^*) > 0$, $\omega_{ess}(\mathcal{A} + D\mathcal{F}(u^*) )<   0$. 
\end{prop}
\begin{proof}
    We recall that $\mathcal{A} + D\mathcal{F}(u^*) = \mathcal{A} + D\mathcal{F}(u^*)_1 + D\mathcal{F}(u^*)_2$. Since $D\mathcal{F}(u^*)_2$ is compact by Proposition \ref{prop:d2_compact}, $||D\mathcal{F}(u^*)_2||_{ess} = 0 $. Thus, $\omega_{ess}(\mathcal{A} + D\mathcal{F}(u^*) ) = \omega_{ess}(\mathcal{A} + D\mathcal{F}(u^*)_1)$. 
    We can express the semi-group generated by $\mathcal{A} + D\mathcal{F}(u^*)_1$ as follows 
    \begin{align*}
        T_{\mathcal{A} + D\mathcal{F}(u^*)_1}(t) \begin{pmatrix}
            x_0\\y_0
        \end{pmatrix}
        &= 
        \begin{pmatrix}
            x_0(a-t)e^{-\int_{a-t}^{a}( \ks(u) + \CtotNS N_1^* + \CtotS N_2^* ) \dd u} \one_{a \geq t} \\0
        \end{pmatrix}\\
        & + \begin{pmatrix}
             x_0(t-a)e^{-\int_0^{a}(\ks(u) +  \CtotNS N_1^* + \CtotS N_2^* )  \dd u }\one_{a < t}\\
            0
        \end{pmatrix} \\
        & := T_1 + T_2.
        \end{align*}
We have,
\begin{equation*}
    ||T_1 (x_0,y_0)^T||_{\mathcal{L}(X)} \leq e^{-  (\CtotNS N_1^* + \CtotS N_2^*)  t} \int_0^{+\infty}x_0(u) \dd u.
\end{equation*}
Hence 
\begin{equation*}
   ||T_1 (t)||_{ess} \leq  ||T_1 (t)||_{\mathcal{L}(X)} \leq e^{- (\CtotNS N_1^* + \CtotS N_2^*)t}.  
\end{equation*}
Furthermore, we notice that our operator $T_2$ is exactly of the form of the operator $T_2$ in the proof of Theorem 3.3 in \cite{Implication_Perasso_19}. Therefore,  $T_2$ is compact and $||T_2 (t)||_{ess} = 0$. Hence
\begin{equation*}
     ||T_{\mathcal{A} + D\mathcal{F}(u^*)_1}(t) ||_{ess} \leq e^{-(\CtotNS N_1^* + \CtotS N_2^*)t}. 
\end{equation*}
And finally,
\begin{equation*}
    \omega_{ess}(\mathcal{A} + D\mathcal{F}(u^*)_1(t)) \leq -(\CtotNS N_1^* + \CtotS N_2^*). 
\end{equation*}
\end{proof}

As the essential growth bound is negative, by Theorem \ref{thm:eigen_2table} and \eqref{eq:wo} and \eqref{eq:wo_w}, the linearised system around a steady state is stable iff all of its eigenvalues have strictly negative real parts, which is what we show to prove Theorem \ref{thm:stab} in the Appendix \ref{sec:app_thm}.

\subsection{Global long-time convergence}
\label{sec:global_PDE}
We have already shown that under the non-extinction condition, the unique positive steady state of our system is locally asymptotically stable, but that does not imply that any trajectory will converge to that steady state. We now prove a stronger result of global attractiveness, but which necessitates stronger technical assumptions on the different rates. Global convergence results are difficult to obtain for non-linear systems. Oftentimes, such results are shown by assuming weak or infinite non-linearities as in \cite{Pakdaman_09,perthame2025stronglynonlinearagestructured,Torres_2022}. Here, we present a global convergence result that does not require such strict conditions on the non-linearities.

\begin{assumption}
\label{assump:bounds_rates}
$\exists \underline{k}>0$ such that $\forall a \geq 0$, 
\begin{align}
    &\underline{k} \leq \ks(a)\label{eq:low_k}.  
\end{align}
\end{assumption}

We introduce the following quantities, 
\begin{align}
    &\bar{N_1}:=  \max\Big(N_1(0), \frac{||b||_{\infty}- \underline{k}}{\CtotNS }\Big) \label{eq:N1bar}, \\
    &\bar{N}_2 :=  \max\Big(N_2(0), \frac{||k||_{\infty} + \CSNS\bar{N_1}^2 }{\kd - \CSS \bar{N_1}}\Big), \label{eq:N2bar}\\
    & \underline{N_1} :=  \min\Big( \frac{\lambda_0}{M(||b||_{\infty} - \underline{\ks})} , \frac{\lambda_0/M -\CtotS \bar{N_2}}{\CtotNS}\Big) \label{eq:N1under}. 
\end{align}

Finally, we introduce the following assumption on the rates, 
\begin{assumption} \label{assump:rates}
\begin{align}
&\kd \geq \CSNS \bar{N_1},\label{assump:death} \tag{i} \\
&  \frac{\CtotNS}{\CtotS} \geq \frac{||k||_{\infty} + \CDSNS \bar{N_2}+ \CSS N_2^* + \CSNS  N_1^* }{2 \CDSNS\underline{N_1}}+ \frac{\CSNS}{2\CDSNS},\label{assump:ratio} \tag{ii} \\
& \kd \geq  - \frac{1}{2}\Big((||k||_{\infty}+ 2\CSS + \CSNS) \bar{N_1}+ \CDSNS \bar{N_2} -2  \CDSS \underline{N_2} + (\CSS-2\CDSS) N_2^*  +(\CSNS -\CDSNS) N_1^* \Big) \label{assump:death2} \tag{iii} \\
& \exists M >1 , \bar{N_2} < \frac{\lambda_0}{ \CtotS M } \label{assump:M},\tag{iv} 
\end{align}
 where $\bar{N_1}, \bar{N_2}$ and $\underline{N_1}$ are defined by \eqref{eq:N1bar}, \eqref{eq:N2bar} and \eqref{eq:N1under} and $\lambda_0$ is the unique solution to $\int_0^{+\infty} b(a) e^{-\int_0^{a} \ks(u) + \lambda \dd u } \dd a = 1$ as defined in \eqref{eq:lambda}.

\end{assumption}
In the following, $(\no^*, N_2^*)$ denotes the unique non-trivial positive steady state of \eqref{eq:syst_linear_PDEODE}, which exists under the assumptions stated in Theorem \ref{thm:unique_2S} and $N_1^*$ is defined by \eqref{eq:def_N}. 
The main result of this section is as follows, 
\begin{thm}\label{thm:conv_global}
Suppose Assumptions \ref{assump:bio}, \ref{assump:bounds}, \ref{ass:for_lambda_0},  \ref{assump:positive_compet} ,\ref{assump:bounds_rates} and \ref{assump:rates} are verified and $\CtotNS, \CDSNS  >0$.  
 Let $(\no,N_2) \in C(\R_+; L^1(\R_+))\times C(\R_+)$ be a positive solution of \eqref{eq:syst_linear_PDEODE} with $N_1$ defined by \eqref{eq:def_N}. 

Then, we have, for any $a \geq 0$,
\begin{align*}
  \no(t,a) \xrightarrow[t \to + \infty]{}\no^*(a) \\
  N_2(t) \xrightarrow[t \to + \infty]{}N_2^*. 
\end{align*} 
Furthermore, the convergence of $\no$ is uniform in $a$ on bounded intervals. 
\end{thm}

We (re)introduce the following notations : 
\begin{align}
    &\lambda(\no,t) := \int_0^{+\infty} (b(a)-\ks(a))\frac{\no(t,a)}{N_1(t)} \dd a \label{def:lambda}\\
    &\lambda_0 = \int_0^{+\infty} (b(a)-\ks(a))\frac{\no^*(a)}{N_1^*} \dd a = \lambda(\no^*, t) \label{def:lambda0}\\
    &\kappa(\no,t) := \int_0^{+\infty} \ks(a) \frac{\no(t,a)}{N_1(t)} \dd a \label{def:kappa}\\
    & I := \int_0^{+\infty} e^{-\int_0^{a}\ks(u) + \lambda_0 \dd u}\dd a\\
    &\kappa_0 := \int_0^{+\infty} \ks(a) \frac{\no^*(a)}{N_1^*} \dd a = \kappa(n_1^*,t) = \frac{1}{I}-\lambda_0 \label{def:kappa0}. 
\end{align}

Where $ \lambda_0 $ is the unique solution to $\int_0^{+\infty} b(a) e^{-\int_0^{a} \ks(u) + \lambda_0 \dd u } \dd a = 1$, as defined by equation \eqref{eq:lambda}, for which the existence under Assumptions \ref{assump:bio} and \ref{ass:for_lambda_0} is given by Proposition \ref{prop:ex_phi}. Equality \eqref{def:lambda0} is shown in the proof of Theorem \ref{thm:unique_2S}.

We first present an intermediate convergence result on the population density in phase~1, $\no$. We study a renormalisation of $\no$ which verifies a linear PDE and converges to a stable age distribution to conclude that $\no$ also converges to a stable age distribution.

\begin{lemma}\label{lemma:conv_nns}
Suppose Assumptions \ref{assump:bio}, \ref{assump:bounds},\ref{ass:for_lambda_0},  \ref{assump:positive_compet} and  \ref{assump:bounds_rates} are verified and $\CtotNS  >0$. 
Let $(\no,N_2) \in  C(\R_+; L^1(\R_+))\times C(\R_+)$ be a positive solution of \eqref{eq:syst_linear_PDEODE} and $N_1$ be defined by \eqref{eq:def_N}. We have, $\forall a \geq 0$, 
$ \frac{\no(t,a)}{N_1(t)} \xrightarrow[t \to + \infty]{}\frac{\no^*(a)}{N_1^*}$ and the convergence is uniform in $a$ on bounded intervals. In particular, $\lambda(\no,t) \xrightarrow[t \to + \infty]{} \lambda_0$ and $\kappa(\no,t) \xrightarrow[t \to + \infty]{} \kappa_0$ with $\kappa, \lambda, \kappa_0$ and $\lambda_0$ defined by equations \eqref{def:lambda} to \eqref{def:kappa0}.
\end{lemma}

\begin{proof}
This proof is inspired by the proof of Proposition 3.3 in \cite{meleard_roget_bd} and Appendix A in \cite{Tran2006ModlesPS}.
We introduce 
\begin{equation*}
    v_1(t,a) = \no(t,a) e^{\CtotNS  \int_0^t N_1(u) \dd u + \CtotS  \int_0^t N_2(u) \dd u}.
\end{equation*}
Defined as such, $v_1$ is a solution of the following linear renewal system
\begin{equation}\label{eq:v_ns}
    \begin{cases}
        &\frac{\partial v_1(t,a)}{\partial a} + \frac{\partial v_1(t,a)}{\partial a} = -\ks(a)v_1(t,a)\\
        &v_1(0,t) = \int_0^{+\infty}b(a) v_1(t,a) \dd a\\
        & v_1(a,0) = v_1^0(a) \in L_1(\R_+). 
    \end{cases}
\end{equation}
This system has been extensively studied in the litterature (see e.g. \cite{ Iannelli_17,Perthame2007_transport, pruss1981equilibrium,webb1985theory}) and it can be rewritten as the following system
\begin{align*}
\begin{cases}
    &v_1(t,a) = K(t-a) e^{-\int_0^{a} \ks(u) \dd u}\\
    &K(t) =  \one_{t\geq 0}\int_0^{t} b(a)K(t-a) e^{-\int_0^{a} \ks(u) \dd u} \dd a + \one_{t<0} \int_0^{+\infty}b(a+t) n_{1,0}(a)e^{-\int_{a+t}^{a} \ks(u) \dd u}  \dd a.
 \end{cases}
\end{align*}
We set $K_0(t) = \one_{\{t\geq 0\}} \int_0^{+\infty}b(a+t) n_{1,0}(a)e^{-\int_a^{a+t} \ks(u) \dd u}  \dd a$. By Theorem 2 in \cite{athreya2004branching}, we have, 
\begin{equation*}
    K(t)e^{-\lambda_0 t} \xrightarrow[t \to \infty]{} \frac{\int_0^{+\infty} K_0(u) e^{-\lambda_0 u} \dd u}{\int_0^{+\infty} a b(a) e^{-\lambda_0 a -\int_0^{a}\ks(u) \dd u} \dd a} =: c. 
\end{equation*}
Hence, for any fixed $a >0$,  and as $n_1^*(a) = n_1^*(0) e^{-\int_0^{a} (\ks(u) + \lambda_0) \dd u}$ by Proposition \ref{prop:exp_ss}, 
\begin{align*}
    e^{-\lambda_0 t}v_1(t,a) = e^{-\lambda_0 (t-a)}K(t-a) e^{-\lambda_0 a}e^{-\int_0^{a} \ks(u) \dd u} \xrightarrow[t \to \infty]{} c e^{-\lambda_0 a}e^{-\int_0^{a} \ks(u) \dd u} = cn^*_1(a) 
\end{align*}
 uniformly in $a$ on bounded intervals.
 
Le us now show that $||v_1(t,.) e^{-\lambda_0 t}||_1 \to c N_1^*$. In order to prove that, we first show that $\int_A^{+\infty} v_1(t,a) \dd a$ converges to $0$ when $A \to +\infty$ uniformly in $t$ for $t$ large enough. 

By Theorem 3.1 in \cite{Perthame2007_transport}, there exists a constant $C$ such that 
\begin{equation}
    \forall t \geq 0, \int_0^{+\infty} v_1(t,a) \dd a \leq C e^{\lambda_0 t}. \label{eq:unif_N}
\end{equation}

Furthermore, by the method of characteristics, we can write 
\begin{align}
    v_1(t,a) = \one_{t \geq a} v_1(0,t-a) e^{-\int_0^{a} \ks(u) \dd u} + \one_{t < a} v_1^0(a-t) e^{-\int_{a-t}^{a} \ks(u) \dd u} \label{eq:charac}. 
\end{align}
Let $\varepsilon > 0$. For $A > 0$ to be fixed later and $t \geq A$. Using \eqref{eq:charac}, \eqref{eq:unif_N} and Assumptions \ref{assump:bounds} and \ref{assump:bounds_rates}, we have
\begin{align*}
    &\int_A^{+\infty} v_1(t,a) \dd a = \int_A^{t}  v_1(0,t-a) e^{-\int_0^{a} \ks(u) \dd u} \dd a +  \int_t^{+\infty} v_1^0(a-t) e^{-\int_{a-t}^{a} \ks(u) \dd u} \dd a \\
    &\leq \int_A^{t}  \int_0^{+\infty} b(u) v_1(u,t-a) \dd u \cdot e^{-\int_0^{a} \ks(u) \dd u} \dd a + e^{-\underline{k} t}||v_1^0||_1 \\
    & \leq ||b||_{\infty} C \int_A^{t}  e^{\lambda_0 (t-a)}  e^{-\int_0^{a} \ks(u) \dd u} \dd a + e^{-\underline{k} t}||v_1^0||_1 \\
    &\leq ||b||_{\infty} e^{\lambda_0 t} C \int_A^{+\infty}  e^{-\lambda_0  a}  e^{-\int_0^{a} \ks(u) \dd u} \dd a  + e^{-\underline{k} t}||v_1^0||_1 \\
    & \leq \varepsilon e^{\lambda_0 t} + e^{-\underline{k} t}||v_1^0||_1, 
\end{align*}
for $A$ large enough. 
Hence, we can choose $t_0 > 0$ large enough such that for $t \geq t_0$
\begin{align*}
    \int_A^{+\infty} e^{-\lambda_0 t} v_1(t,a) \dd a \leq \varepsilon  + e^{-(\underline{k} +\lambda_0)t}||v_1^0||_1 \leq \varepsilon  + e^{-(\underline{k} +\lambda_0)t_0}||v_1^0||_1 \leq 2 \varepsilon. 
\end{align*}
Since the convergence of $e^{-\lambda_0 t} v_1(t,a)$ is uniform in $a$ on bounded intervals, we have for $A$ large enough and  $t \geq t_0$, 
 \begin{align*}
     \Big|\int_0^{+\infty} e^{-\lambda_0 t} v_1(t,a) \dd a - c \int_0^{+\infty} n^*_1(a) \dd a \Big|  &\leq \Big|\int_0^{A} e^{-\lambda_0 t} v_1(t,a) \dd a - c \int_0^{A} n^*_1(a) \dd a \Big| + 3 \varepsilon\\
     & \leq 4 \varepsilon. 
 \end{align*}
Hence $\int_0^{+\infty} e^{-\lambda_0 t} v_1(t,a) \dd a \to c N_1^*$.  Thus, we have, 
\begin{align}
    \frac{\no(t,a)}{N_1(t)} &= \frac{ e^{-\CtotNS  \int_0^t N_1(u) \dd u - \CtotS  \int_0^t N_2(u) \dd u} v_1(t,a)}{ e^{-\CtotNS  \int_0^t N_1(u) \dd u - \CtotS  \int_0^t N_2(u) \dd u} \int_0^{+\infty}v_1(u,t) \dd u} \nonumber\\
    & = \frac{e^{-\lambda_0 t} v_1(t,a)}{e^{-\lambda_0 t} \int_0^{+\infty}v_1(u,t) \dd u} \xrightarrow[t \to + \infty]{} \frac{\no^*(a)}{N_1^*}. \label{eq:conv_NS_ratio}
\end{align}
Similarly, since $b$ is bounded and $\ks$ positive by Assumptions \ref{assump:bio} and \ref{assump:bounds}, we have for $A$ large enough and some $t_0 > 0$ such that for $t \geq t_0$
\begin{align*}
   \Big| \int_A^{+\infty} e^{-\lambda_0 t} (b(a)-k(a)) v_1(t,a) \dd a \Big| \leq ||b||_{\infty}  \int_A^{+\infty} e^{-\lambda_0 t} v_1(t,a) \dd a \leq 2 \varepsilon. 
\end{align*}
Hence $\int_0^{+\infty} e^{-\lambda_0 t} (b(a)-k(a)) v_1(t,a) \dd a  \xrightarrow[t \to + \infty]{} c \int_0^{+\infty} (b(a)-k(a)) n^*_1(a) \dd a $. Since $\ks$ is bounded by Assumption \ref{assump:bio}, we also have $  \int_0^{+\infty} e^{-\lambda_0 t} k(a) v_1(t,a) \dd a  \xrightarrow[t \to + \infty]{} c \int_0^{+\infty}k(a) n^*_1(a) \dd a $. And finally, 
\begin{align*}
    \lambda(n_1,t) = \frac{\int_0^{+\infty} (b(a)-k(a)) n_1(t,a) \dd a }{\int_0^{+\infty}  n_1(t,a) \dd a} = \frac{\int_0^{+\infty} e^{-\lambda_0 t} (b(a)-k(a)) v_1(t,a) \dd a  }{\int_0^{+\infty} e^{-\lambda_0 t} v_1(t,a) \dd a } \xrightarrow[t \to + \infty]{} \lambda_0.
\end{align*}
Thus $\lambda(\no,t) \xrightarrow[t \to + \infty]{} \lambda_0$ and similarly, $\kappa(\no,t) \xrightarrow[t \to + \infty]{} \kappa_0$, where $\lambda, \kappa, \lambda_0$ and $\kappa_0$ are defined in \eqref{def:lambda} to \eqref{def:kappa0}. 
\end{proof}

The proof strategy deployed above to prove Lemma \ref{lemma:conv_nns} relies heavily on the fact that the equation on $\no$ can be rescaled with an exponential to an autonomous equation, which only works if $\tilde{b} = 0$.

We now state and prove a boundedness result. Indeed, the lower bound presented in Proposition \ref{prop:bound_wcompker} uses the assumption that $\tilde{b} >0$, which is not the case for system \eqref{eq:syst_linear_PDEODE}. We explicit the lower and upper bounds for the population sizes, which is useful to prove the convergence later on. 

\begin{prop}\label{prop:bound_Pop_2ize}
     Suppose Assumptions \ref{assump:bio}, \ref{assump:bounds}, \ref{ass:for_lambda_0}, \ref{assump:positive_compet} and  \ref{assump:bounds_rates} are verified and $\CtotNS  >0$. 
    Let  $(\no, N_2) \in C(\R_+; L^1(\R_+))\times C(\R_+)$ be a positive solution of \eqref{eq:syst_linear_PDEODE} with $N_1$ defined by \eqref{eq:def_N}. Then
    \begin{equation*}
      \sup_{t \geq 0}(N_1(t)) \leq \max\Big(N_1(0), \frac{||b||_{\infty}- \underline{k}}{\CtotNS }\Big) = \bar{N_1}. 
    \end{equation*}
    Under the further condition that Assumption \ref{assump:rates} \eqref{assump:death} is verified, we have 
    \begin{align*}
        \sup_{t \geq 0}(N_2(t)) \leq \max\Big(N_2(0), \frac{||k||_{\infty} + \CSNS\bar{N_1}^2 }{\kd - \CSS \bar{N_1}}\Big) = \bar{N}_2.
    \end{align*}
    Finally, if Assumption \ref{assump:rates} \eqref{assump:M} is also verified, we have for some $t_0 \geq 0$ and $M$ defined in \ref{assump:rates} \eqref{assump:M},  
    \begin{equation*}
         \inf_{t \geq t_0}(N_1(t)) \geq  \min\Big( \frac{\lambda_0}{M(||b||_{\infty} - \underline{\ks})} , \frac{\lambda_0/M -\CtotS \bar{N_2}}{\CtotNS}\Big)  = \underline{N_1} > 0.  
    \end{equation*}
\end{prop}
\begin{proof}

We start by proving the upper bound of $N_1$. By integrating the first equation of \eqref{eq:syst_linear_PDEODE} in age and performing an integration by parts, we have $\forall t \geq 0$,
\begin{align*}
    \frac{\dd N_1(t)}{\dd t} &= \int_0^{+\infty} b(a) \no(t,a) \dd a - \int_0^{+\infty} \ks(a)  \no(t,a) \dd a - (\CtotNS N_1(t) + \CtotS N_2(t)) N_1(t)\\
    & = \lambda(n_1,t) N_1(t) - (\CtotNS N_1(t) + \CtotS N_2(t)) N_1(t)\\
    & \leq ||b||_{\infty}N_1(t)  - (\underline{k} + \CtotNS N_1(t)) N_1(t), 
\end{align*}
with $\lambda(n_1,t)$ defined in equation \eqref{def:lambda} and using the positivity of $N_1$ and $N_2$. 
By Assumption \ref{ass:for_lambda_0}, we have 
\begin{equation}
    1 < \int_0^{+\infty} b(a) e^{-\int_0^a \ks(u) \dd u} \dd a \leq \frac{||b||_{\infty}}{\underline{\ks}}. \label{eq:binf_k}
\end{equation}
Hence for all $t \geq 0$   $||b||_\infty > \underline{k}$ and we have by Lemma \ref{lemma:eq_diff} in the Appendix, 
\begin{equation*}
    \sup_{t \geq 0}(N_1(t)) \leq \max\Big(N_1(0), \frac{||b||_{\infty}- \underline{k}}{\CtotNS }\Big) = \bar{N_1}.
\end{equation*}
Similarly for $N_2$, and using the fact that $N_1,N_2 \geq 0$, we have
\begin{align*}
    \frac{\dd N_2(t)}{\dd t} &= - (\kd + \CDSNS N_1(t)+\CDSS N_2(t) ) N_2(t)+ \int_0^{+\infty}(\ks(a) +\CSNS N_1(t) + \CSS N_2(t) )N_1(t,a) \dd a \\
    &\leq - \kd  N_2(t)+ (||k||_{\infty}   + \CSNS \bar{N_1} + \CSS N_2(t)) \bar{N}_1. 
\end{align*}
As by Assumption \ref{assump:rates} \eqref{assump:death}, $\kd > \CSS \bar{N_1}$, we obtain by Lemma \ref{lemma:eq_diff} in the Appendix,
\begin{align*}
    \sup_{t \geq 0}(N_2(t))  &\leq \max\Big(N_2(0), \frac{||k||_{\infty} + \CSNS\bar{N_1}^2 }{\kd - \CSS \bar{N_1}}\Big):= \bar{N}_2. 
\end{align*}
    By Lemma \ref{lemma:conv_nns}, we know that $\lambda(n_1,t) \xrightarrow[t \to \infty]{} \lambda_0$. By assumption, $\exists M >1, \CtotS \bar{N_2} < \frac{\lambda_0}{M}$. Since $M >1$, we can find $t_0 > 0$ such that $\forall t \geq t_0, \lambda(n_1,t) \geq \lambda_0/M $. For $t \geq t_0$, we have 
    \begin{align*}
        \frac{\dd N_1(t)}{\dd t} \geq N_1(t) \Big(\frac{\lambda_0}{M} - \CtotNS N_1(t)  - \CtotS \bar{N_2}\Big). 
    \end{align*}
  Since by Assumption  \ref{assump:rates} \eqref{assump:M}, $\bar{N_2} < \frac{\lambda_0}{\CtotS  M}$, the previous equation allows us to conclude with Lemma \ref{lemma:eq_diff} that for $t \geq t_0$, $N_1(t) \geq \min \Big( N_1(t_0), \frac{\lambda_0/M -\CtotS \bar{N_2}}{\CtotNS} \Big)$. \\
  As $0 < \frac{\lambda_0}{M} \leq \lambda(n_1,t_0) \leq (||b||_{\infty} - \underline{\ks}) N_1(t_0)$ and $||b||_\infty > \underline{k}$ by \eqref{eq:binf_k}, we have $N_1(t_0) >\frac{\lambda_0}{M(||b||_{\infty} - \underline{\ks})}  $.
  Thus finally, 
  \begin{equation*}
     \inf_{t\geq t_0} (N_1(t)) \geq \min \Big( \frac{\lambda_0}{M(||b||_{\infty} - \underline{\ks})}, \frac{\lambda_0/M -\CtotS \bar{N_2}}{\CtotNS} \Big).
  \end{equation*}
\end{proof}

We now present a result on the convergence of total population sizes, which assumes several bounds on the rates to ensure a Lyapunov-type convergence result.

\begin{prop}\label{prop:conv_pop_2ize}
Suppose the assumptions of Theorem \ref{thm:conv_global} are verified. Let $(\no,N_2) \in C(\R_+; L^1(\R_+))\times C(\R_+)$ be a positive solution of \eqref{eq:syst_linear_PDEODE} with $N_1$ defined by \eqref{eq:def_N}.
Then we have 
    \begin{align*}
     N_1(t) \xrightarrow[t \to \infty]{} N_1^* \text{ and } N_2(t) \xrightarrow[t \to \infty]{} N_2^*,
    \end{align*}
with $N_1^*$ and $N_2^*$ the population sizes at the unique strictly positive equilibrium by Theorem \ref{thm:unique_2S}. 
\end{prop}

\begin{proof} 
We have $\forall t \geq 0, \forall n \in \mathcal{C}(\R_+; L^1(\R_+)) \geq 0, \lambda(n,t) \geq 0$, where $\lambda$ is defined by \eqref{def:lambda}. 
Using the fact that $(N_1^*, N_2^*)$ are the population sizes at equilibrium, and with $ \lambda_0 = \int_0^{+\infty} (b(a)-\ks(a))\frac{\no^*(a)}{N_1^*} \dd a = \lambda(\no^*, t)$, we have
\begin{align*}
    &\frac{1}{2}\frac{\dd (N_1(t) -N_1^*)^2}{\dd t } =  (N_1(t) -N_1^*) \frac{d(N_1(t) - N_1^*)}{\dd t} \\
    & = (N_1(t)\! -\!N_1^*)\Big(\lambda(n_1,t) N_1(t) \!-\!\CtotNS N_1(t)^2 \!-\! \CtotS N_1(t)N_2(t) \!-\! \lambda_0 N_1^* \!+\! \CtotNS (N_1^*)^2\! + \! \CtotS N_1^*N_2^*\Big)\\
    &= (N_1(t) - N_1^*)^2 \Big(\lambda_0 - \CtotNS (N_1(t) + N_1^*) - \CtotS N_2(t) \Big)  + N_1(t) (N_1(t) - N_1^*) (\lambda(n_1,t) - \lambda_0)  \\
    &\qquad +  (N_1(t) - N_1^*)\CtotS N_1^*(N_2(t)-N_2^*). 
\end{align*}

Similarly, with $\kappa$ and $\kappa_0$ defined in \eqref{def:kappa} and \eqref{def:kappa0}, we have
\begin{align*}
   \frac{1}{2} &\frac{\dd (N_2(t) - N_2^*)^2}{\dd t} = \Big(-\kd - \CDSNS  N_1^*  - \CDSS (N_2(t) + N_2^*) + \CSS N_1(t) \Big) (N_2(t) - N_2^*)^2 \\
    &+ \Big(\kappa(\no,t) + \CDSNS N_2(t) + \CSS N_2^* + \CSNS (N_1(t) + N_1^*)\Big) (N_2(t)-N_2^*) (N_1(t) - N_1^*)\\
    & + (\kappa(\no,t) - \kappa_0)N_1^*(N_2(t) - N_2^*) . 
\end{align*}
Adding both differential equations with a rescaling parameter on $\frac{\dd (N_2(t) - N_2^*)^2}{\dd t}$ yields
\begin{equation}\label{eq:diff_N2N1}
\begin{aligned}
    \frac{\CtotS}{2\CDSNS}\cdot &\frac{\dd (N_2(t) - N_2^*)^2}{\dd t} + \frac{1}{2} \frac{\dd (N_1(t) -N_1^*)^2}{\dd t }  = A(t)  (N_1(t) -N_1^*)^2 + B(t)(N_2(t) - N_2^*)^2 \\
    &+ C(t) (N_2(t)-N_2^*) (N_1(t) - N_1^*) +  \frac{\CtotS}{2\CDSNS}(\kappa(\no,t) - \kappa_0)N_1^*(N_2(t) - N_2^*) \\
    &+  N_1(t) (\lambda(n_1,t) - \lambda_0) (N_1(t) - N_1^*).
    \end{aligned}
\end{equation}
Where 
\begin{align*}
    &A(t) =  \lambda_0 - \CtotNS (N_1(t) + N_1^*) - \CtotS N_2(t) \\
    &B(t) = -\frac{\CtotS}{\CDSNS}\kd - N_1^* \CtotS - \frac{\CtotS}{\CDSNS}\CDSS (N_2(t) + N_2^*) + \frac{\CtotS}{\CDSNS}\CSS N_1(t) \\
    & C(t) = \CtotS N_1^* + \frac{\CtotS}{\CDSNS}\Big(\kappa(\no,t) + \CDSNS N_2(t) + \CSS N_2^* + \CSNS (N_1(t) + N_1^*)\Big). 
\end{align*}
As $C(t) \geq 0$, we can rewrite equation \eqref{eq:diff_N2N1} as follows 
\begin{equation}\label{eq:diff_N2N1}
\begin{aligned}
    &\frac{\CtotS}{2\CDSNS}\cdot \frac{\dd (N_2(t) - N_2^*)^2}{\dd t} + \frac{1}{2} \frac{\dd (N_1(t) -N_1^*)^2}{\dd t }  = (A(t)+\frac{C(t)}{2} ) (N_1(t) -N_1^*)^2 \\
    &+ (B(t) + \frac{C(t)}{2}) (N_2(t) - N_2^*)^2 - \frac{C(t)}{2} (N_2(t)-N_2^* +  N_1(t) - N_1^*)^2\\
    &+  \frac{\CtotS}{2\CDSNS} (\kappa(\no,t) - \kappa_0)N_1^*(N_2(t) - N_2^*) +  N_1(t) (\lambda(n_1,t) - \lambda_0) (N_1(t) - N_1^*).
    \end{aligned}
\end{equation}
As $N_1(t)$ and $N_2(t)$ are bounded and $\kappa(n_1,t)-\kappa_0 \to 0$ and $\lambda(n_1,t) - \lambda_0
\to 0$ by Lemma \ref{lemma:conv_nns}, the last two terms of equation \eqref{eq:diff_N2N1} go to $0$. 
Thus, if we can show that there is a positive constant $\alpha$ such that $A(t) + \frac{C(t)}{2} <- \alpha$ and $B(t) + \frac{C(t)}{2} < -\alpha$, the function $(N_1(t) - N_1^*)^2 + \frac{\CtotS}{2\CDSNS} (N_2(t) - N_2^*)^2$ will verify a differential inequality as in the conditions of Lemma \ref{lemma:eq_diff_0} which will allow us to conclude that $(N_1(t) - N_1^*)^2 + (N_2(t) - N_2^*)^2 \to 0 $ and thus that $N_1(t) \to N_1^*$ and $N_2(t) \to N_2^*$. 

We write the following
\begin{align*}
    &A(t) + \frac{C(t)}{2} = \lambda_0 - \CtotNS (N_1(t) + N_1^*) - \CtotS N_2(t) +  \frac{1}{2} \CtotS N_1^* \\
    &+ \frac{1}{2}\frac{\CtotS}{\CDSNS}\Big( \kappa(\no,t) + \CDSNS N_2(t) + \CSS N_2^* + \CSNS (N_1(t) + N_1^*)\Big)
\end{align*}

Since $\CtotNS  N_1^* + \CtotS  N_2^* = \lambda_0$, we have for $t \geq t_0$ where $t_0$ is such that $ \lambda(n_1,t) \geq \frac{\lambda_0}{M}$ for $t \geq t_0$, 
\begin{align*}
    &A(t) \!+ \!\frac{C(t)}{2} \! \leq \! \CtotS  N_2^* - \CtotNS N_1(t)  +  \frac{1}{2} \CtotS N_1^* + \frac{1}{2}\frac{\CtotS}{\CDSNS}\Big( ||\ks||_{\infty}+ \CDSNS N_2(t) + \CSS N_2^* + \CSNS (N_1(t) + N_1^*)\Big)
\end{align*}

Using the upper and lower bounds for $N_1$ and $N_2$ proved in Proposition \ref{prop:bound_Pop_2ize}, $A(t) + \frac{C(t)}{2}$ is strictly negative provided 
\begin{align*}
    \frac{\CtotNS}{\CtotS} \geq \frac{\frac{1}{2\CDSNS}\Big(||k||_{\infty}+ \CDSNS \bar{N_2} + \CSS N_2^* + \CSNS  N_1^*\Big) }{\underline{N_1}}+ \frac{1}{2\CDSNS} \CSNS.
\end{align*}
where $\bar{N_1}, \bar{N_2}, \underline{N_1}$ are defined in equations \eqref{eq:N1bar}, \eqref{eq:N2bar} and \eqref{eq:N1under}. The above inequality is verified by Assumption  \ref{assump:rates} \eqref{assump:ratio}.

On the other hand, we have 
\begin{align*}
    &\frac{\CDSNS}{\CtotS}\Big(B(t) + \frac{C(t)}{2}\Big) = -\kd - N_1^* \CDSNS - \CDSS (N_2(t) + N_2^*) +\CSS N_1(t)  \\
    & + \frac{1}{2} \CDSNS N_1^* + \frac{1}{2}\Big(\kappa(n_1, t)+ \CDSNS N_2(t) + \CSS N_2^* + \CSNS (N_1(t) + N_1^*)\Big) \\
    & \leq -\kd - \frac{1}{2}N_1^* \CDSNS - \CDSS (\underline{N_2} + N_2^*) +\CSS \bar{N_1}   + \frac{1}{2}\Big(||k||_{\infty} \bar{N_1}+\CDSNS \bar{N_2} + \CSS N_2^* + \CSNS (\bar{N_1}+ N_1^*)\Big). 
\end{align*}
The upper bound is strictly negative provided 
\begin{align*}
    \kd \geq \frac{1}{2}N_1^* \CDSNS + \CDSS (\underline{N_2} + N_2^*) -\CSS \bar{N_1}  - \frac{1}{2}\Big(||k||_{\infty} \bar{N_1}+ \CDSNS \bar{N_2} + \CSS N_2^* + \CSNS ( \bar{N_1} + N_1^*)\Big). 
\end{align*}
This is once again true by Assumption  \ref{assump:rates}\eqref{assump:death2}. 
We can therefore conclude with Lemma \ref{lemma:eq_diff_0}, and we have $N_1(t) \xrightarrow[t \to \infty]{} N_1^*$ and $N_2(t) \xrightarrow[t \to \infty]{} N_2^*$. 
\end{proof}

The proof of Theorem \ref{thm:conv_global} follows directly from the previous results. 

\begin{proof}[Proof of Theorem \ref{thm:conv_global}] 
Since $ \frac{\no(t,a)}{N_1(t)} \xrightarrow[t \to + \infty]{}\frac{\no^*(a)}{N_1^*}$ uniformly in $a$ on bounded intervals by Lemma \ref{lemma:conv_nns} and $N_1(t) \xrightarrow[t \to + \infty]{} N_1^* $ by Proposition \ref{prop:conv_pop_2ize}, we have the convergence of $\no$. The convergence of $N_2$ is given by Proposition \ref{prop:conv_pop_2ize}.

\end{proof}

\section*{Conclusion}
In this article, we study a novel system of two coupled age-structured partial differential equations inspired by the biological observation of two phase ageing \cite{rera_intestinal_2012}. One of the main novelties introduced is that the transition between both phases occurs at an age-dependent rate and is density dependent. 
Furthermore, the presence of non-linearities and the fact that both equations are highly coupled through transition, competition and birth terms imply several theoretical difficulties. Our main results are existence, uniqueness and local and/or global stability of steady states for two simplified systems. The first in an ODE system, which we study using tools from bifurcation theory and stability analysis. The second simplification is a PDE-ODE system. We study this system in a general setting using semi-group theory, and show local asymptotic stability of the strictly positive steady state. Under further assumptions, and by studying a renormalised equation and a Lyapunov-type functional, we show that this steady state attracts all trajectories, thus ruling out the presence of sustained oscillations. Notably, this result does not require the assumption of weak or infinite non-linearities. 

The uniqueness of steady states and absence of oscillations show that the proportion of individuals in each phase at equilibrium is a unique feature of the model. This is encouraging for ecological applications. Indeed, the experimental measure of such a proportion, which is realistically feasible in practice, could help gain some insight on the health of a wild population. Such insights could not be obtained otherwise as it is almost impossible to measure the chronological age of individuals in the wild.

\section*{Funding and acknowledgments}
The author would like to thank Marie Doumic, Sarah Kaakaï and Michaël Rera for their help and supervision and Benoit Perthame for his fruitful discussions. The author was funded by ANR project  ANR-25-CE45-5355. The author would like to thank the Chaire Modélisation Mathématique and Biodiversité Veolia Environnement - Ecole Polytechnique - Muséum National d’Histoire Naturelle - Fondation X.

\bibliographystyle{abbrv}
\bibliography{Biblio_these.bib}

\appendix
\section{Appendix}

\subsection{Proofs of Section \ref{sec:gen_lin}}
\label{sec:app_lin}
\begin{proof}[Proof of Theorem \ref{thm:exis_lin}]
    This proof is inspired by the proof of Theorem 3.1 in \cite{Perthame2007_transport}. 
    The existence result is obtained by using the Banach fixed point theorem in the Banach space $X^2$ where $X = C([0,T]; L^1(\R_+))$. We equip $X^2$ with the norm $||(n_1,n_2)||_{X^2} = \sup_{0\leq t \leq T} ||n_1(t)||_1 + \sup_{0\leq t \leq T} ||n_2(t)||_1$, where $T$ is chosen such that 
    \begin{align*}
        &T \max(||b||_{\infty},||\tilde{b}||_{\infty}) \leq \frac{1}{2} \quad  \text{and} \quad  T(||\ks||_{\infty}+ \sup_{0\leq s \leq T}(\CS(S_1(s), S_2(s))) \leq \frac{1}{4}.
    \end{align*}
    Where we use the boundedness assumption \ref{assump:bounds} for $b, \tilde{b}$ and $\ks$ as well as Assumption \ref{assump:compet_0} ($\CS \in L_{\text{loc}}^{\infty}(\R^2)$)  and the fact that $S_1$ and $S_2$ are locally bounded functions of time. 
    
    For $(m_1, m_2) \in X$, we define the operator $\mathcal{S}$ such that $(\no, \nt) =: \mathcal{S}[m_1,m_2]$ is defined by the following system
    \begin{align}
\label{syst:edp_lin_aux}
    \begin{cases}
        &\frac{\partial \no(t,a) }{\partial t} + \frac{\partial \no(t,a) }{\partial a} = -(\ks(a) + \CS(S_1(t),S_2(t)) + \gamo(S_1(t),S_2(t)))\no(t,a)\\
        &\no(0,t) = \int_0^{+\infty}(m_1(t,a)b(a)+m_2(t,a)\tilde{b}(a) )\dd a\\
        &\frac{\partial \nt(t,a) }{\partial t} + \frac{\partial \nt(t,a) }{\partial a} = -(\kd(a) +\gamt(S_1(t),S_2(t)) )\nt(t,a)\\
        &\nt(0,t) = \int_0^{+\infty}\no(t,a)(\ks(a)+\CS(S_1(t),S_2(t)))\dd a\\
        &\no(a,0) = n_{1,0}(a) \quad \nt(a,0) = n_{2,0}(a) \qquad n_{1,0}, n_{2,0} \in (L^{1}(\R_+))^2.
    \end{cases}
\end{align}
For two pairs of functions $(m_1, m_2)$ and $(\tilde{m}_1, \tilde{m}_2)$ in $X^2$ and $(\no, \nt), (\tilde{\no}, \tilde{\nt})$ the corresponding images by $\mathcal{S}$, we define $u_1 := \no - \tilde{\no}, u_2 := \nt - \tilde{\nt}$ and $v_1 := m_1 - \tilde{m}_1, v_2 := m_2 - \tilde{m}_2$ such that  
\begin{align*}
    \begin{cases}
        &\frac{\partial |u_1(t,a)| }{\partial t} + \frac{\partial |u_1(t,a)| }{\partial a} = -(\ks(a) + \CS(S_1(t),S_2(t)) + \gamo(S_1(t),S_2(t)))|u_1(t,a)|\\
        &|u_1(0,t)| = \Big|\int_0^{+\infty}(v_1(t,a)b(a)+v_2(t,a)\tilde{b}(a)) \dd a \Big|\\
        &\frac{\partial| u_2(t,a)| }{\partial t} + \frac{\partial |u_2(t,a)| }{\partial a} = -(\kd(a) +\gamt(S_1(t),S_2(t)) )|u_2(t,a)|\\
        &|u_2(0,t)| = \Big|\int_0^{+\infty}u_1(t,a)(\ks(a)+\CS(S_1(t),S_2(t)))\dd a \Big|\\
        &|u_1(a,0)| = 0 \quad |u_2(a,0)| = 0.
    \end{cases}
\end{align*}
The differential equation on the absolute value is obtained by multiplying the initial PDE by $\text{sign}(u_1)$, and stays true in a distributional sense, even though $|u_1|$ is not differentiable in the classical sense when $u_1 = 0$. As our estimates only require integral inequalities, a distributional equality is sufficient. 
By integrating the equation on $|u_1|$ on $\R_+$ in $a$ and on $[0,t]$ in time, we obtain 
\begin{align*}
    ||u_1(t)||_1 &\leq  \int_0^t|u_1(0,s)| \dd s \\
   & = \int_0^t  \Big|\int_0^{+\infty}(v_1(a,s)b(a)+v_2(a,s)\tilde{b}(a)) \dd a \Big| \dd s\\
   & \leq t (||b||_{\infty} \sup_{0\leq s \leq t}||v_1(s,.)||_1 + ||\tilde{b}||_{\infty} \sup_{0\leq s \leq t}||v_2(s,.)||_1)\\
   & \leq t \max(||b||_{\infty},||\tilde{b}||_{\infty}) \sup_{0\leq s \leq t}(||v_1(s,.)||_1  + ||v_2(s,.)||_1). 
\end{align*}
By taking the supremum for $0\leq t \leq T$, and by the choice of $T$, this yields 
\begin{equation*}
    ||u_1||_X \leq \frac{1}{2}( ||v_1||_X  + ||v_2||_X) . 
\end{equation*}
Similarly by integrating the equation on $|u_2|$,
\begin{align*}
    ||u_2(t)||_1 &\leq  \int_0^t|u_2(0,s)| \dd s \\
   & = \int_0^t  \Big|\int_0^{+\infty}u_1(a,s)(\ks(a)+\CS(S_1(t),S_2(t)) ) \dd a \Big| \dd s\\
   & \leq t (||\ks||_{\infty}+ \sup_{0\leq s \leq t}(\CS(S_1(s), S_2(s))) )\sup_{0\leq s \leq t}||u_1(s)||_1\\
   & \leq t (||\ks||_{\infty}\!+ \!\sup_{0\leq s \leq t}(\CS(S_1(s), S_2(s)))\max(||b||_{\infty},||\tilde{b}||_{\infty}) \!\sup_{0\leq s \leq t}(||v_1(s)||_X\! +\! ||v_2(s)||_X). 
\end{align*}
And by once again taking the supremum on $0 \leq t \leq T$, we obtain
\begin{equation*}
    ||u_2||_X \leq \frac{1}{4}( ||v_1||_X  + ||v_2||_X). 
\end{equation*}
Hence, we have 
\begin{equation*}
    ||(u_1,u_2)||_{X^2}\leq \frac{3}{4} ||(v_1,v_2||_{X^2}. 
\end{equation*}
Thus, the operator $\mathcal{S}$ is a strict contraction in the Banach space $X^2$ and the Banach fixed point theorem (see e.g. Theorem 1 in 9.2.1 in \cite{evans2010partial}) ensures that there is existence of a unique fixed point. As this proof can be iterated over intervals of the form $[iT, (i+1)T]$ for any integer $i$, this means that we can build a unique solution on $C(\R_+, L^1(\R_+))^2$.

Furthermore, for two initial condition sets $(n_{1,0}^1,n_{2,0}^1) $ and $(n_{1,0}^2,n_{2,0}^2) $ such that we have $n_{1,0}^1 ~\leq ~n_{1,0}^2$ and $n_{2,0}^1 ~\leq~ n_{2,0}^2$ and their respective operators $\mathcal{S}^1, \mathcal{S}^2$, we have for all $m_1, m_2$, $\mathcal{S}^1( m_1, m_2) ~\leq ~\mathcal{S}^2( m_1, m_2)$ (this can be easily seen by solving the equation verified by $\mathcal{S}^2( m_1, m_2) ~ - ~\mathcal{S}^1( m_1, m_2)$ and noticing that the solutions are always positive). Thus by iteration, this comparison principle is also verified by the solutions of \eqref{syst:edp_lin}. In particular, this ensures that they are always positive if the initial conditions are positive. 

Finally, the a priori bound \eqref{eq:exp_bound_L1} can be shown using the fact that 
\begin{align*}
    \frac{\partial }{\partial t }\int_0^{+\infty} |\no(t,a)| + |\nt(t,a)| \dd a &= \int_0^{+\infty} \Big(\frac{\partial }{\partial t }|\no(t,a)| + \frac{\partial }{\partial t }|\nt(t,a)|\Big) \dd a \\
    & \leq ||b||_{\infty}\int_0^{+\infty} |\no(t,a)| \dd a +  ||\tilde{b}||_{\infty}\int_0^{+\infty} |\nt(t,a)| \dd a 
\end{align*}
and concluding with Grönwall's Lemma. 
\end{proof}

\begin{proof}[Proof of Theorem \ref{thm:asymp_linear_b}]
    The proof is similar to the proof of Theorem 3.5 p.66 in \cite{Perthame2007_transport}, to which we refer the reader for more details. The result is shown by differentiating
    \begin{equation*}
        \int_0^{+\infty} \Big(|h_1(t,a)| \phi^0_1(a) + |h_2(t,a)| \phi^0_2(a)\Big) \dd a, 
    \end{equation*}
    and using Grönwall's Lemma. 
\end{proof}
\subsection{Proofs of Section \ref{sec:ex_un} }
\label{sec:app_ex}
\begin{proof}[Proof of Theorem \ref{them:ex} ]
This proof is inspired by the proof of Theorem 4 in \cite{Perthame_Tumulari2008} and relies once again on the Banach fixed point theorem (see \cite{evans2010partial}) as in the proof of Theorem \ref{thm:exis_lin}. We introduce $X^2 = C([0,T])^2$ for a positive time $T$ to be chosen later and $X_+$ the set of non-negative continuous functions on $[0,T]$. We equip $X_+^2$ with the distance $\mathcal{D}$ defined as 
$$\forall (u_1,u_2), (v_1,v_2) \in( X_+^2)^2\mathcal{D}((u_1,u_2), (v_1,v_2)) = \sup_{0\leq t \leq T}\Big(| u_1(t) - v_1(t) | +|u_2(t) - v_2(t) | \Big). $$

For some given functions $S_2(t), S_1(t) \in C([0,T])^2 \subset L^{\infty}_{\text{loc}}(\R_+)$, we have by Theorem \ref{thm:exis_lin} the existence of $(\no, \nt) \in C([0,T]; L^1(\R_+))^2$ solving \eqref{syst:edp} on $[0,T]$. 
Similarly as in the proof of Theorem \ref{thm:exis_lin}, this allows us to define the following map in $(X_+)^2\to (X_+)^2$,
    \begin{equation*}
       \mathcal{T}: (S_1(t),S_2(t)) \mapsto \Big(\int_0^{+\infty} \psi_1(a) \no(t,a) \dd a, \int_0^{+\infty} \psi_2(a) \nt(t,a) \dd a\Big).
    \end{equation*}
    We now show that  $\mathcal{T}$ is a contraction map, which will show the expected result by the Banach fixed point theorem. 
    Let us consider two solutions of \eqref{syst:edp}, $(\no,\nt)$ and $(\tilde{\no},\tilde{\nt})$ with their respective competition factors denoted $(S_1,S_2)$ and $(\tilde{S}_1, \tilde{S}_2)$ and defined by \eqref{eq:def_2}. Denote $h_1 = \no - \tilde{\no} $ and  $h_2 = \nt - \tilde{\nt} $.
   By multiplying the PDEs on $h_1$ and $h_2$ by $\text{sign}(h_1)$ and $\text{sign}(h_2)$ respectively, we obtain in a distributional sense,
    \begin{align*}
        &\begin{cases}
            &\frac{\partial}{\partial t}|h_1(t,a)| +\frac{\partial}{\partial a}|h_1(t,a)| + \gamtot(S_1(t), S_2(t))|h_1(t,a)| +\ks(a)|h_1(t,a)| \\
            & \qquad \leq |  \gamtot(S_1(t), S_2(t)) -  \gamtot(\tilde{S}_1(t), \tilde{S}_2(t))|\tilde{\no}(t,a)\\
            &|h_1(0,t)| \leq \int_0^{+\infty}b(a)|h_1(t,a)| \dd a + \int_0^{+\infty}\tilde{b}(a)|h_2(t,a)| \dd a\\
            & h_1(a,0) = 0.
        \end{cases}\\
        &\begin{cases}
            &\frac{\partial}{\partial t}|h_2(t,a)| +\frac{\partial}{\partial a}|h_2(t,a)| + \gamt(S_1(t), S_2(t))|h_2(t,a)| +\kd(a)|h_2(t,a)| \\
            & \qquad \leq |  \gamt(S_1(t), S_2(t)) -  \gamt(\tilde{S}_1(t), \tilde{S}_2(t))|\tilde{\nt}(t,a)\\
            &|h_2(0,t)| \leq \int_0^{+\infty}(\ks(a)+\CS(S_1(t), S_2(t)))|h_1(t,a)| \\
            & \qquad \qquad \qquad +|  \gamt(S_1(t), S_2(t)) -  \gamt(\tilde{S}_1(t), \tilde{S}_2(t))|\tilde{\no}(t,a)\dd a\\
            & h_2(a,0) = 0.
        \end{cases}
    \end{align*}
    Hence by integrating in age, with Assumptions \ref{assump:bounds}, \ref{assump:lip} and using the exponential bound given by equation \eqref{eq:exp_bound_L1}, we have 
    \begin{align*}
        &\frac{\dd}{\dd t }\int_0^{+\infty} \big(|h_1(t,a)|+ |h_2(t,a)|\big)\dd a \leq \int_0^{+\infty} \big(- \frac{\partial}{\partial a}|h_1| - \frac{\partial}{\partial a}|h_2| - \gamtot(S_1(t), S_2(t))|h_1(t,a)| \\
        &\qquad -\ks(a) |h_1(t,a)| +|  \gamtot(S_1(t), S_2(t)) -  \gamtot(\tilde{S}_1(t), \tilde{S}_2(t))|\tilde{\no}(t,a)\\
        & \qquad  - \gamt(S_1(t), S_2(t))|h_2(t,a)| + |  \gamt(S_1(t), S_2(t)) -  \gamt(\tilde{S}_1(t), \tilde{S}_2(t))|\nt^2(t,a) \\
        & \qquad - \kd(a) |h_2(t,a)| \dd a\\
        & \leq \int_0^{+\infty}b(a)|h_1(t,a)| \dd a + \int_0^{+\infty}\tilde{b}(a)|h_2(t,a)| \dd a - \int_0^{+\infty}\kd(a)|h_2(t,a)| \big)\dd a \\
        & \qquad + L(|S_1(t) - \tilde{S}_1(t)| + |S_2(t) - \tilde{S}_2(t)|)\int_0^{+\infty} \big(2\tilde{\no}(t,a) + \tilde{\nt}(t,a)\big) \dd a \\
        & \qquad - \int_0^{+\infty} \big(\gamt(S_1(t), S_2(t))|h_2(t,a)| + \gamt(S_1(t), S_2(t))|h_1(t,a)| \big)\dd a\\
        &\leq \max(||b||_{\infty},||\tilde{b}||_{\infty})\int_0^{+\infty} \big(|h_1(t,a)|+ |h_2(t,a)|\big)\dd a \\
        &\qquad + L\sup_{0\leq t\leq T}(|S_1(t) - \tilde{S}_1(t)| + |S_2(t) - \tilde{S}_2(t)|) 2e^{\max(||b||_{\infty},||\tilde{b}||_{\infty}) t}(||\no(0,.)||_{1}+||\nt(0,.)||_{1}).
    \end{align*}
    Then, by Grönwall's Lemma, we have 
    \begin{align*}
       \int_0^{+\infty} \big( |h_1(t,a)|+ |h_2(t,a)|\big)\dd a \leq& Lt 2e^{2\max(||b||_{\infty},||\tilde{b}||_{\infty}) t}(||\no(0,.)||_{1}+||\nt(0,.)||_{1})\\
       & \times \sup_{0\leq t\leq T}(|S_1(t) - \tilde{S}_1(t)| + |S_2(t) - \tilde{S}_2(t)|). 
    \end{align*}
Hence, when taking the supremum on $[0,T]$, 
    \begin{align*}
       \sup_{0\leq t \leq T}& \int_0^{+\infty} \big(|h_1(t,a)|+ |h_2(t,a)| \big)\dd a \\
       &\leq e^{||b||_{\infty}T}LT 2e^{\max(||b||_{\infty},||\tilde{b}||_{\infty}) T}(||\no(0,.)||_{1}+||\nt(0,.)||_{1}) \\
       & \qquad \qquad \qquad \times \sup_{0\leq t\leq T}(|S_1(t) - \tilde{S}_1(t)| + |S_2(t) - \tilde{S}_2(t)|).
    \end{align*}
    Finally, we have 
    \begin{align*}
       D(\mathcal{T}(S^1)&, \mathcal{T}(S^2)) = \sup_{0\leq t \leq T} \Big(\Big| \int_0^{+\infty}\psi_1(a) (\no^1(t,a) - \no^2(t,a))\dd a \Big| \\
       & \qquad \qquad \qquad  \qquad \qquad + \Big| \int_0^{+\infty}\psi_2(a) (\nt^1(t,a) - \nt^2(t,a))\dd a \Big| \Big)\\
        & \leq ||\psi_2 + \psi_1||_{\infty} \sup_{0\leq t \leq T} \Big(\int_0^{+\infty}|h_1(t,a)| \dd a + \int_0^{+\infty}|h_2(t,a)|\dd a \Big)\\
        & \leq ||\psi_2 + \psi_1||_{\infty} e^{\max(||b||_{\infty},||\tilde{b}||_{\infty})T} LT 2e^{\max(||b||_{\infty},||\tilde{b}||_{\infty}) T}(||\no(0,.)||_{1}+||\nt(0,.)||_{1})\\
        & \qquad \times \sup_{0\leq t\leq T}(|S_1(t) - \tilde{S}_1(t)| + |S_2(t) - \tilde{S}_2(t)|).
    \end{align*}
    Thus, we can choose $T$ small enough such that $\mathcal{T}$ is a contraction map, which allows us to use the Banach fixed point theorem to conclude. By iterating on intervals of the form $[iT, (i+1)T]$ for $i$ an integer, this proves the existence of solutions on $C(\R_+, L^1(\R_+))^2$
\end{proof}

\begin{proof}[Proof of Proposition \ref{prop:positivity}.]
Let $(n_1, n_2)$ be a solution to \eqref{syst:edp}, and denote  $S_1, S_2$ the corresponding environmental factors as defined by \eqref{eq:def_2}.$S_1$ and $ S_2$ are continuous on $\R_+^*$ and have a finite value at $t=0$ and are therefore locally bounded on $\R_+$. Now consider the following linear system where $S_1, S_2$ are considered as given functions of time. 
\begin{align*}
    \begin{cases}
        &\frac{\partial u_1(t,a) }{\partial t} + \frac{\partial u_2(t,a) }{\partial a} = -(\ks(a) + \CS(S_1(t),S_2(t)) + \gamo(S_1(t),S_2(t))u_1(t,a)\\
        &u_1(0,t) = \int_0^{+\infty}b(a)u_1(t,a)+\tilde{b}(a)u_2(t,a) \dd a \\
        &\frac{\partial u_2(t,a) }{\partial t} + \frac{\partial u_2(t,a) }{\partial a} = -(\kd(a) +\gamt(S_1(t),S_2(t)) )u_2(t,a)\\
        &u_2(0,t) = \int_0^{+\infty}u_1(t,a)(\ks(a)+\CS(S_1(t),S_2(t)))\dd a \\
        &u_1(a,0)= 0 \quad u_2(a,0) = 0.
    \end{cases}
\end{align*}
As the initial conditions are $0$, the solution to this system is $u_1 \equiv u_2 \equiv 0$. 
Thus, we can apply the comparison principle enunciated for the linear system in Theorem \ref{thm:exis_lin} and conclude that the solutions to \eqref{syst:edp} stay positive provided the initial conditions are positive. 

\end{proof}
\subsection{Proof of Theorem \ref{thm:stab} }
\label{sec:app_thm}
\paragraph{Case 1 $\int_0^{+\infty} b(a) e^{-\int_0^{a} \ks(u) \dd u} \dd a > 1$.}
The existence of a unique non-trivial positive steady state is given by Theorem \ref{thm:unique_2S}. We now prove its stability. Since $N_1^* = \int_0^{+\infty} n_1^*(a) \dd a > 0$ by Theorem \ref{thm:unique_2S}, we have $\omega_{ess}(\mathcal{A} + D\mathcal{F}(u^*) ) <  0$ by Theorem \ref{thm:omega_ess}. There only remains to study the eigenvalues of the linearised system around $u^* = (\no^*(.), N_2^*, 0)$.  
For ease of notation, for any $f$ in $L_1(\R_+)$ we denote $I(f) = \int_0^{+\infty} f(u) \dd u$. 
    Let $(x,y,0) \in X$ and $\lambda \in \mathbb{C}$ be such that $(\mathcal{A} + D\mathcal{F}(u^*))e^{\lambda t} (x(a),y,0)^T = 0$.  We have 
\begin{align}
   & \qquad \qquad  (\mathcal{A} + D\mathcal{F}(u^*))e^{\lambda t} \begin{pmatrix}
        x(a)\\ y \\ 0
    \end{pmatrix} = 0 \nonumber\\
    &\iff
     e^{\lambda t} \begin{pmatrix}
        -(\ks(a) + \lambda + \CtotNS N_1^* + \CtotS N_2^* ) x(a) - x'(a)\\
        -(\CDSNS N_1^* + \CDSS N_2^* + \CDSS N_2^* + \kd + \lambda) y   + \CSS  N_1^* y \\
        -x(0)
    \end{pmatrix} \nonumber \\
  + e^{\lambda t} &\begin{pmatrix}
        - (\CtotNS X + \CtotS y )\no^*(a) \\
        - \CDSNS I(x) N_2^* +  \int_0^{+\infty} (\ks(a) + \CSNS N_1^* + \CSS N_2^* ) x(a) \dd a + \CSNS I(x) N_1^*\\
        \int_0^{+\infty} b(a) x(a) \dd a
    \end{pmatrix}=0. \label{eq:diff_lambda}
\end{align}
Recall $\lambda_0 = \CtotNS N_1^* + \CtotS N_2^*$ as shown in the proof of Theorem \ref{thm:unique_2S}. 
We introduce 
\begin{align*}
&G_b(\lambda) = \int_0^{+\infty}b(a) e^{-\int_0^{a}(\ks(u) + \lambda +\lambda_0 )\dd u } \dd a \quad \text{and} \quad 
G(\lambda) = \int_0^{+\infty} e^{-\int_0^{a}(\ks(u) + \lambda + \lambda_0 )\dd u } \dd a .
\end{align*}
By solving \eqref{eq:diff_lambda}, the system verified by $\lambda$ and $(x,y)$ is
\begin{equation}\label{syst:xyI}
\begin{cases}
     &x(a) = e^{-\int_0^{a} (\ks(t) + \lambda_0 + \lambda) \dd u } (x(0) - (\CtotNS  I(x) + \CtotS y) \frac{\no^*(0) (e^{\lambda a}-1)}{\lambda})\\
    &y =  \frac{ \int_0^{+\infty} \ks(a) x(a) \dd a + 2\CSNS I(x) N_1^* +  \CSS N_2^*I(x) -\CDSNS I(x) N_2^* }{\CDSNS N_1^* + 2\CDSS N_2^*  + \kd + \lambda- \CSS N_1^* }\\
    &x(0) = x(0)G_b(\lambda) - \no^*(0)\frac{\CtotNS I(x) + \CtotS y}{\lambda}(1 - G_b(\lambda))\\
    & I(x)(1 + \frac{\CtotNS }{\lambda}(N_1^* - \no^*(0)G(\lambda))) = x(0)G(\lambda) - \frac{\CtotS  y}{\lambda}(N_1^* - \no^*(0)G(\lambda)).
    \end{cases}
\end{equation}
Furthermore, following the first equation in \eqref{syst:xyI} and by integration by parts, we have
\begin{align*}
     &\int_0^{+\infty} \ks(a) x(a) \dd a  = x(0) \int_0^{+\infty}\ks(a)  e^{-\int_0^{a}( \ks(t) + \lambda_0 + \lambda) \dd u } \dd a \\
     &   - (\CtotNS  I(x) + \CtotS y) \frac{\no^*(0) }{\lambda} \int_0^{+\infty} \!\ks(a) \Big( e^{-\int_0^{a}( \ks(t) + \lambda_0 )\dd u } \dd a \! -\!  \int_0^{+\infty} \!\ks(a)  e^{-\int_0^{a} (\ks(t) + \lambda_0 +\lambda )\dd u } \dd a \Big)\\
     & = x(0) \Big(1 - (\lambda_0 + \lambda) G(\lambda)\Big) - (\CtotNS  I(x) + \CtotS y) \frac{\no^*(0) }{\lambda} \Big(1 - \lambda_0 \frac{N_1^*}{n_1^*(0)} - (1 - (\lambda + \lambda_0) G(\lambda))\Big)\\
     & =  x(0) \Big(1 - (\lambda_0 + \lambda) G(\lambda)\Big) - (\CtotNS  I(x) + \CtotS y) \frac{1}{\lambda} \Big( - \lambda_0 N_1^*  + n_1^*(0)(\lambda + \lambda_0) G(\lambda))\Big). 
\end{align*}

Thus system \eqref{syst:xyI} can be rewritten only in terms of the three variables  $(x(0), y , I(x))$ in the following way, 
\begin{align*}
\begin{cases}
 & (1-G_b(\lambda))\Big(x(0) + \no^*(0)\frac{\CtotNS I(x) + \CtotS y}{\lambda} \Big) = 0\\
&y(\CDSNS N_1^* + 2\CDSS N_2^*  + \kd + \lambda- \CSS  N_1^*  + \frac{\CtotS}{\lambda} ((\lambda_0 +\lambda) \no^*(0)G(\lambda)- \lambda_0 N_1^*)) \\
& \qquad  - I(x)\Big(2\CSNS  N_1^* +  \CSS N_2^* -\CDSNS  N_2^* +\frac{\CtotNS}{\lambda}  ((\lambda_0 +\lambda) \no^*(0)G(\lambda)- \lambda_0 N_1^*)\Big)\\
& \qquad -  x(0) \Big(1- (\lambda+\lambda_0) G(\lambda)\Big) = 0 \\
& I(x)\Big(1 + \frac{\CtotNS }{\lambda}(N_1^* - \no^*(0)G(\lambda))\Big) - x(0)G(\lambda) + y\frac{\CtotS}{\lambda}\Big(N_1^* - \no^*(0)G(\lambda)\Big) = 0.
\end{cases}
\end{align*}

A non-trivial solution to the previous system exists iff the determinant $D$ associated to the linear system solved by $(x(0), y, I(x))$ is 0.

The determinant $D(\lambda)$ of the previous system is
\begin{align}
 D(\lambda) &= (1-G_b(\lambda))
    \begin{vmatrix}
         1 &  a_2& a_3 \\
        b_1 & b_2 & b_3\\
        c_1 & c_2 & c_3
    \end{vmatrix}
   \nonumber \\ \label{eq:det}
    & 
    = (1-G_b(\lambda)) (b_2c_3 + b_1c_2a_3+c_1a_2b_3 - a_3b_2c_1 - b_3c_2 - c_3a_2b_1).
\end{align}
Introducing the notations $b_2^*$ and $b_3^*$, the coefficients of \eqref{eq:det} can be explicited as follows,  
\begin{align*}
&a_2 := \frac{\CtotS \no^*(0)}{\lambda} \\
&a_3 := \frac{\CtotNS  \no^*(0)}{\lambda}\\
&b_1 := -1 + (\lambda_0 +\lambda) G(\lambda) \\
&b_2 := \CDSNS N_1^* + 2\CDSS N_2^*  + \kd - \CSS  N_1^* + \lambda +\frac{\CtotS }{\lambda}((\lambda_0 +\lambda) \no^*(0)G(\lambda)- \lambda_0 N_1^*)\\
& \quad = b_2^* +\lambda + (\lambda_0 +\lambda )G(\lambda) a_2 -\frac{\lambda_0 N_1^* \CtotS }{\lambda}\\
&b_3 := \CDSNS N_2^* - 2 \CSNS N_1^* - \CSS N_2^* + \frac{ \CtotNS }{\lambda}((\lambda_0 +\lambda) \no^*(0)G(\lambda)- \lambda_0 N_1^*)\\
& \quad = b_3^* +(\lambda_0+\lambda) G(\lambda) a_3 -\frac{\lambda_0 N_1^* \CtotNS }{\lambda}\\
& c_1 := -G(\lambda) \\
& c_2 := \frac{\CtotS }{\lambda}( N_1^* - \no^*(0)G(\lambda)) = \frac{\CtotS N_1^*}{\lambda} - G(\lambda)a_2\\
& c_3 := 1 + \frac{\CtotNS }{\lambda}(N_1^* - \no^*(0)G(\lambda)) = 1 + \frac{\CtotNS N_1^*}{\lambda} - G(\lambda)a_3.
\end{align*}

By expanding \eqref{eq:det} and expressing $b_2,b_3,c_2,c_3$ in terms of the other coefficients as done above, all the terms in $G(\lambda)$ and $1/\lambda^2$ cancel out and we are left with 
\begin{align*}
   D(\lambda) =  \frac{1}{\lambda}(1-G_b(\lambda))(\lambda^2 \!+ \!\lambda (b_2^* \!+ \!\CtotNS N_1^*) \!+ \!b_2^*\CtotNS  N_1^* -b_3^*\CtotS  N_1^* \!+ \!\CtotS \no^*(0) -\lambda_0\CtotS N_1^*).
\end{align*}
As $1 < \int_0^{+\infty} b(a) e^{-\int_0^{+\infty }\ks(u) \dd u } \dd a$ by Assumption \ref{ass:for_lambda_0}, $1-G_b(\lambda) \neq 0$ when $Re(\lambda) >0$, and there remains to check that the roots of the right factor second-degree polynomial in the expression of $D(\lambda)$ have negative real parts. By the Routh-Hurwitz criterion, this happens iff the three coefficients of the polynomial are strictly positive. The square coefficient is $1$. 
Furthermore, by using the properties of the steady states, we notice that 
\begin{align*}
   & b_2^* = \CDSNS N_1^* + 2\CDSS N_2^*  + \kd  - \CSS N_1^* = \frac{1}{N_2^*} \Big( \int_0^{+\infty}\ks(a)\no^*(a) \dd a  + \CSNS (N_1^*)^2 +  \CDSS (N_2^*)^2\Big),\\
    &b_3^* = \CDSNS N_2^* - 2 \CSNS N_1^* - \CSS N_2^* = \frac{1}{N_1^*}\Big(\int_0^{+\infty}\ks(a)\no^*(a) \dd a -\CDSS (N_2^*)^2 -\kd N_2^* -\CSNS(N_1^*)^2\Big).
\end{align*}
Thus, we have $b_2^* >0$, hence $b_2^* + \CtotNS N_1^* > 0$ and the linear coefficient is strictly positive. Finally, for the constant coefficient we have
\begin{align*}
    -b_3^*&\CtotS  N_1^* \!+ \!\CtotS \no^*(0) -\lambda_0\CtotS N_1^* \\
    &= -\CtotS \int_0^{+\infty}\ks(a)\no^*(a) \dd a + \CtotS  \CDSS (N_2^*)^2 + \CtotS \kd N_2^* +\CtotS  \CSNS(N_1^*)^2 \\
    & \qquad +\!\CtotS \no^*(0) -\lambda_0\CtotS N_1^*\\
    & =  \CtotS  \CDSS (N_2^*)^2 + \CtotS \kd N_2^* +\CtotS  \CSNS(N_1^*)^2 \\
    & \qquad +\!\CtotS \no^*(0)\Big(1 - \int_0^{+\infty}( \lambda_0+\ks(a))e^{-\int_0^{a}\ks(u) \dd u - \lambda_0 a} \dd a\Big)\\
    & =  \CtotS  \CDSS (N_2^*)^2 + \CtotS \kd N_2^* +\CtotS  \CSNS(N_1^*)^2  \geq  0. 
\end{align*}
Hence, the constant coefficient is strictly positive. 

Thus, all the roots have strictly negative real parts, and the positive stationary state is stable. 

\paragraph{Case 2 $\int_0^{+\infty} b(a) e^{-\int_0^{a} \ks(u) \dd u} \dd a < 1$.}

We state the equation verified by a potential eigenvalue of $\mathcal{A} + D\mathcal{F}(0) $,  $\lambda$. 
    We have for $(x,y,0) \in X$ and $\lambda \in \mathbb{C}$, 
    \begin{equation*}
        (\mathcal{A} + D\mathcal{F}(0))\begin{pmatrix}
            x(a) e^{\lambda t}\\
            y e^{\lambda t}\\
            0 
        \end{pmatrix} = e^{\lambda t} \begin{pmatrix}
            -(\ks(a) + \lambda)x(a) -x'(a) \\
            -(\kd + \lambda)y +  \int_0^{+\infty}\ks(a) x(a) \dd a\\
            -x(0)+ \int_0^{+\infty}b(a) x(a) \dd a. 
        \end{pmatrix}
    \end{equation*}
Thus, the equation $(\mathcal{A} + D\mathcal{F}(0))Y= 0, Y \in X $ has a non trivial solution $(x(.)e^{\lambda t},ye^{\lambda t} ,0)$ iff 
\begin{equation*}
    \begin{cases}
        &x(a) = x(0)e^{-\int_0^{a}(\ks(u) + \lambda) \dd u}\\
        &y = \frac{x(0)}{\kd + \lambda} (1- \lambda \int_0^{+\infty} e^{-\int_0^{a}(\ks(u) + \lambda) \dd u} \dd a)\\
        & 1 =  \int_0^{+\infty} b(a) e^{-\int_0^{a}(\ks(u) + \lambda )\dd u} \dd a. 
    \end{cases}
\end{equation*}
If $1 < \int_0^{+\infty} b(a) e^{-\int_0^{+\infty }\ks(u) \dd u } \dd a$ and $Re(\lambda) \geq  0$, the last condition yields  
\begin{align*}
    1 = \Big|\int_0^{+\infty} b(a) e^{-\int_0^{a}(\ks(u) + Re(\lambda) )\dd u} cos(Im(\lambda) a) \dd a\Big| \leq \int_0^{+\infty} b(a) e^{-\int_0^{a}(\ks(u) + Re(\lambda)) \dd u} \dd a < 1
\end{align*}
 which is a contradiction, hence $Re(\lambda) < 0$.    
 Similarly, $1 > \int_0^{+\infty} b(a) e^{-\int_0^{+\infty }\ks(u) \dd u } \dd a$ implies $Re(\lambda) > 0$. Finally, Theorems \ref{thm:eigen_2table} and \ref{thm:omega_ess} allow us to conclude.

\subsection{Technical lemmas}

We present here two technical results on ordinary differential equations. .
\begin{lemma} \label{lemma:eq_diff_0}
     Let $f:\R_+ \to \R_+$ be a function such that $\exists \varepsilon : \R_+ \to \R, \forall t \geq 0$, 
    \begin{equation*}
        \frac{\dd f(t)}{\dd t} \leq -\alpha f(t) + \varepsilon(t)
    \end{equation*}
    where $\alpha >0$ and $\varepsilon(t) \to 0$ as $t \to + \infty$. 
    Then, $f(t) \to 0$ as $t \to \infty$. 
\end{lemma}
\begin{proof}
Let $f^0$ be the function such that,
\begin{equation*}
\begin{cases}
        &\frac{\dd f^0(t)}{\dd t} =  -\alpha f^0(t) + \varepsilon(t) \qquad  \forall t \geq 0\\
        &f^0(0) = f(0).
        \end{cases}
\end{equation*}       
We have for $t \geq 0$, 
\begin{align*}
    f^0(t) &= e^{-\alpha t}\Big( f(0) + \int_0^t \varepsilon(u) e^{\alpha u} \dd u\Big)\\
    & = e^{-\alpha t}\Big( f(0) + \int_{t/2}^t \varepsilon(u) e^{\alpha u} \dd u + \int_{0}^{t/2} \varepsilon(u) e^{\alpha u} \dd u\Big)\\
    & \leq e^{-\alpha t}\Big( f(0) + \sup_{u \in [t/2,t]} (\varepsilon(u)) e^{\alpha t}  + ||\varepsilon||_{\infty} e^{\alpha t/2} \Big)\\
    & \leq e^{-\alpha t} f(0) + \sup_{u \in [t/2,t]} (\varepsilon(u))   + ||\varepsilon||_{\infty} e^{-\alpha t/2} \xrightarrow[t \to + \infty]{} 0.
\end{align*}
And $0 \leq f \leq f^0$ hence $f(t) \to 0 $ as $t \to \infty$. 
\end{proof}

\begin{lemma} \label{lemma:eq_diff}
     Let $f:\R_+ \to \R$ be a function such that there exists a locally Lipschitz function $u$ such that $u(x) = 0$ has a unique non trivial solution $x^* > 0$ and such that $u(x) \xrightarrow[x \to + \infty]{} -\infty$ (resp such that $u \geq 0$ on $[0,x^*]$) . 
    \begin{equation*}
        \frac{\dd f(t)}{\dd t} \leq u(f(t)) (\text{resp.} \geq).
    \end{equation*}
    Then, for all $t \in \R_+$, $f(t) \leq \max(f(0), x^*)$ (resp. $\geq \min(f(0), x^*)$).  
\end{lemma}
\begin{proof}
Suppose $f(0) < x^*$. We introduce $g = f-x^*$. By contradiction, we assume that there is a point $c>0$ such that $g(c) >0$. By continuity, and since $g(0) < 0 $, there is a point $b\in [0,c)$ such that $g(b) = 0$ and $g(t) > 0$ for $t \in (b,c]$. Thus, on $(b,c]$, we have 
\begin{align*}
   g'(t) \leq u(f(t)) \leq u(f(t)) - u(x^*) \leq L|f(t) - x^*| = L g(t)
\end{align*}
where $L$ is the local Lipschitz constant of $u$. By Grönwall, this entails $g(t) \leq 0 $ on $(b,c]$ which is a contradiction. 

We now suppose that $f(0) \geq x^*$. Since $x^*$ is the only strictly positive solution to $u(x) = 0$ and $u(x) \xrightarrow[x \to + \infty]{} -\infty$, we have $u(f(0)) < 0$. Similarly as before, We introduce $g = f-f(0)$ we assume that there is a point $c>0$ such that $g(c) >0$. By continuity, and since $g(0) = 0 $, there is a point $b\in [0,c)$ such that $g(b) = 0$ and $g(t) > 0$ for $t \in (b,c]$. Thus, on $(b,c]$, we have 
\begin{align*}
   g'(t) \leq u(f(t)) \leq u(f(t)) - u(f(0)) \leq L|f(t) - f(0)| = L g(t)
\end{align*}
where $L$ is the local Lipschitz constant of $u$. By Grönwall, this entails $g(t) \leq 0 $ on $(b,c]$ which is once again a contradiction. 

The case with the opposite inequality can be treated similarly by studying the function  $g := x^* -f$ when $f(0) > x^*$ and  $g := f(0) -f$ when  $0 \leq f(0) \leq x^*$, in which wase $u(f(0)) \geq 0$ by assumption. 
\end{proof}

\end{document}